\documentclass[10pt]{article}
\usepackage{xcolor}
\usepackage{tikz-cd}

\usepackage{persobase}
\usepackage{persomath}
\usepackage{persodrawing}

\usepackage{thm-restate}

\newcommand{\TV}{\operatorname{TV}}
\newcommand{\curve}{\mathcal{C}}

\title{Power quotients of surface groups and mapping class groups}

\author{R\'emi Coulon, Alessandro Sisto, Henry Wilton}

\newcommand{\Addresses}{{
  \bigskip
  \footnotesize
  
    \noindent
    \textsc{Université Bourgogne Europe, CNRS, IMB UMR 5584, 21000 Dijon, FR}\par\nopagebreak  \noindent
    \textit{E-mail address:} \texttt{remi.coulon@cnrs.fr}

    \bigskip
    \noindent
  \textsc{Department of Mathematics, Heriot-Watt University and Maxwell Institute for Mathematical Sciences, Edinburgh, UK}\par\nopagebreak \noindent
    \textit{E-mail address:} \texttt{a.sisto@hw.ac.uk}

\bigskip
\noindent
  \textsc{DPMMS, Centre for Mathematical Sciences, Wilberforce Road, Cambridge, CB3 0WB, UK}\par\nopagebreak \noindent
  \textit{E-mail address:} \texttt{h.wilton@maths.cam.ac.uk}\\[1ex]

}}

\setbftrue
\mcgfrenchfalse

\begin{document}

\maketitle

\begin{abstract}
Let $\Gamma$ be the fundamental group of a closed, orientable, hyperbolic surface $S$. The \emph{$n$-power quotient}, $\Gamma(n)$, is the quotient of $\Gamma$ by the $n$th powers of simple closed curves.  We prove an analogue of the Dehn--Nielsen--Baer theorem for suitable large values of $n$: the outer automorphism group of $\Gamma(n)$ is isomorphic to the quotient of the extended mapping class group of $S$ by $n$th powers of Dehn twists. There is also a corresponding description of the automorphism group as the quotient of the extended mapping class group of the corresponding once-punctured surface, and we relate these groups via  a Birman-type exact sequence. Along the way, and as consequences, we prove structural properties of $\Gamma(n)$ for suitable large values of $n$, including:  $\Gamma(n)$ is virtually torsion-free, acylindrically hyperbolic, infinitely presented, with solvable word problem and finite asymptotic dimension.
\end{abstract}

\tableofcontents

\section{Introduction}

\subsection{Main results}

Let $S$ be a connected, orientable, hyperbolic surface of finite type.
For any integer $n\geq 2$, the \emph{$n$-power quotient} of the fundamental group $\Gamma = \pi_1(S)$ is the quotient obtained by killing the $n$th power of every simple closed curve. 
We will denote it by $\Gamma(n)$. 
Analogously, we may define the \emph{$n$-power Dehn twist subgroup} $\DT^n(S)$ to be the subgroup of the (extended) mapping class group $\mcg[\pm] S$ generated by all $n$th powers of Dehn twists. 
The  \emph{$n$-power quotient} of $\mcg[\pm] S$ is then defined to be $\mcg[\pm] S/\DT^n(S)$.

Power quotients arise naturally in connection with Ivanov's notorious questions about the congruence subgroup property and the virtual first Betti numbers of mapping class groups \cite[Questions 1 and 7]{ivanov_fifteen_2006}. Much previous work focussed on representations of $\mcg[\pm] S/\DT^n(S)$, coming either from quantum representations \cite{aramayona_quotients_2019,funar_TQFT_1999,funar_profinite_2018,koberda_quotients_2016,masbaum_element_1999} or from the homology of finite covers of $S$ \cite{grunewald_arithmetic_2015,klukowski_simple_2025,malestein_simple_2019}. More recently, $\mcg[\pm] S/\DT^n(S)$ has been studied from a coarse-geometric perspective, as a Dehn filling of $\mcg[\pm] S$ \cite{BHMS,Dahmani:spinning,dahmani_dehn_2021,mangioni2024short,mangioni2025rigidity}.

We are motivated by the following analogy:  $\Gamma(n)$ is to the surface group $\pi_1(S)$ as $\mcg[\pm] S/\DT^n(S)$ is to the mapping class group of $S$.  
Thus, although these groups appear at first glance quite mysterious, we can hope to prove theorems about them by analogy with the comparatively well-understood case of mapping class groups.  Our techniques further develop the coarse-geometric perspective, bringing in methods from geometric small-cancellation theory, acylindrically and hierarchically hyperbolic groups, as well as the Rips--Sela machinery for studying automorphism groups \cite{Rips:1994jg}, especially the first author's work with Sela on automorphism groups of periodic groups \cite{Coulon:2021wg}. 

The Birman exact sequence explains the effect that puncturing a surface has on a mapping class group \cite{birman_mapping_1969}. Our first theorem is an analogue of the Birman exact sequence in this context. Throughout, $S_*$ denotes the punctured surface obtained by deleting a point from $S$.

\begin{theo}
\label{thm: Birman for quotients}
	Let $S$ be a closed, connected, orientable, hyperbolic surface with fundamental group $\Gamma$.
	There exists an exponent $N$ such that, for all integers $n \geq N$, there is an exact sequence
	\begin{equation*}
		1\to \Gamma(n)\to \mcg[\pm]{S_*}/\DT^n(S_*)\to\mcg[\pm] S/\DT^n(S)\to 1\,.
	\end{equation*}
\end{theo}

The exact sequence from the theorem is in fact a quotient of the classical Birman exact sequence. Aramayona--Funar state this theorem for all $n\geq 2$ \cite[Proposition 3]{aramayona_quotients_2019}.  However, their proof contains a gap -- for the argument to work one needs to know that the centre of $\Gamma(n)$ is trivial. (See their updated paper for a corrected statement \cite{aramayona_quotients_2024}.) This gap is filled by our \autoref{thm: 1st Structural results} for sufficiently large $n$. But it is not true in general that the centre of $\Gamma(n)$ is always trivial: $\Gamma(2)$ is a non-trivial finite abelian group; see \autoref{exa: Abelian power quotient}.

In the classical setting, the Dehn--Nielsen--Baer theorem identifies the outer automorphism group of the surface group with the extended mapping class group $\mcg[\pm]{S}$ \cite[Theorem 8.1]{farb_primer_2012}. Our second main theorem provides an analogue for power quotients, with more restrictive hypotheses on the exponent $n$.

\begin{theo}
\label{thm: DNB theorem for quotients}
	Let $S$ be a closed, connected, orientable, hyperbolic surface with fundamental group $\Gamma$.
	Then there exists $N>0$ such that, for all multiples $n$ of $N$, the natural maps 
	\begin{equation*}
		 \mcg[\pm]{S_*}\to\aut{\pi_1(S)} \quad \text{and} \quad \mcg[\pm]{S}\to \out{\pi_1(S)}
	\end{equation*}
	induce isomorphisms
	\begin{equation*}
		\aut{\Gamma(n)}\cong\mcg[\pm]{S_*}/\DT^n(S_*) \ \text{and} \ \out{\Gamma(n)}\cong\mcg[\pm]{S}/\DT^n(S)\,.
	\end{equation*}
\end{theo}

\begin{rema}
In fact, the maps $\mcg[\pm]{S_*}\to\aut{\Gamma(n)}$ and $\mcg[\pm]{S}\to \out{\Gamma(n)}$ are surjective for all sufficiently large $n$, by \autoref{thm:onto}. Our proof of injectivity, which follows from \autoref{res: kernels}, only holds for all multiplies of some constant $N$.
\end{rema}

Thus, a structural understanding of $\Gamma(n)$ is desirable. The nature of $\Gamma(n)$ as a group is, on the face of it, mysterious: it initially appears to occupy an intermediate status between a surface group and a Burnside group, and might appear to be closer to a Burnside group. We contribute some structural results that we hope will help to clarify its status, and which suggest that it is in fact closer to a surface group.

First, we obtain some algebraic information about $\Gamma(n)$.

\begin{theo}
\label{thm: 1st Structural results}
	Let $S$ be a closed, orientable, hyperbolic surface with fundamental group $\Gamma$.
	Then there exists an exponent $N>0$ such that, for all integers $n\geq N$, the following hold:
	\begin{enumerate}
	\item $\Gamma(n)$ is a direct limit of hyperbolic groups; \label{enu: Structure of Gamma(n) -- limit}
	\item $\Gamma(n)$ is infinite, with trivial centre, and every finite normal subgroup is trivial; \label{enu: Structure of Gamma(n) -- infinite, trivial centre}
	\item $\Gamma(n)$ is virtually torsion-free; \label{enu: Structure of Gamma(n) -- virtually torsion-free}
	\end{enumerate}	
\end{theo}

\begin{proof}
Item \ref{enu: Structure of Gamma(n) -- limit} is part of the statement of \autoref{res: approximating sequence}, item \ref{enu: Structure of Gamma(n) -- infinite, trivial centre} is \autoref{res: infinite + trivial finite radical}, and item \ref{enu: Structure of Gamma(n) -- virtually torsion-free} is \autoref{res: v torsion-free}.
\end{proof}

That $\Gamma(n)$ is infinite for large $n$ was already known by theorems of Ol'shanskii (for $n$ odd) \cite{olshanskii_periodic_1991} and Ivanov--Ol'shanskii (for $n$ even)  \cite{ivanov_hyperbolic_1996}, since $\Gamma(n)$ surjects the periodic quotient $\Gamma/\Gamma^n$.

More recent work on these groups has focussed on constructing linear representations with infinite image.  Koberda--Santharoubane constructed such using quantum SO(3) representations of mapping class groups \cite[Theorem 1.1]{koberda_quotients_2016}. A different construction, via the homology of finite covers, was given in the closed case by Klukowski \cite[Theorem 7]{klukowski_simple_2025}, building on an argument of Malestein--Putman in the case with punctures \cite{malestein_simple_2019}.  The profinite completion of $\Gamma(n)$ was studied by Funar--Lochak, again using quantum representations \cite{funar_profinite_2018}.

We will use the results mentioned so far to finally obtain some geometric information about $\Gamma(n)$.

\begin{restatable}[]{theo}{structural}
\label{thm: 2nd Structural results}
	Let $S$ be a closed, orientable, hyperbolic surface with fundamental group $\Gamma$.
	There is an integer $N>0$ such that, for all non-zero multiples $n$ of $N$, the following hold:
\begin{enumerate}
		\item $\Gamma(n)$ is acylindrically hyperbolic, but not lacunary hyperbolic (in particular it is not word-hyperbolic); \label{enu: Structure of Gamma(n) -- acylindrical hyperbolicity}
		\item $\Gamma(n)$ is infinitely presented; \label{enu: Structure of Gamma(n) -- infinitely presented}	
		\item $\Gamma(n)$ has a subgroup of finite index with infinite abelianisation; \label{enu: Structure of Gamma(n) -- vb1}		
		\item the word problem is solvable in $\Gamma(n)$; \label{enu: Structure of Gamma(n) -- solvable WP}
		\item $\Gamma(n)$ has finite asymptotic dimension; \label{enu: Structure of Gamma(n) -- finite asdim}
		\item $\Gamma(n)$ satisfies a strong version of the Tits Alternative -- every finitely generated subgroup either contains a non-abelian free group, is cyclic, or is infinite dihedral;\label{enu: Structure of Gamma(n) -- fg Tits alternative}
		\item there exists $T>0$ such that, for any generating set of $\Gamma(n)$, there are two elements of word length at most $T$ that freely generate a free group. \label{enu:growth}
	\end{enumerate}
\end{restatable}

In contrast to the known techniques for studying $\Gamma(n)$ via representations, which by definition cannot study the kernels of those representations, we hope that these results open the door to a genuine structural understanding of the whole of $\Gamma(n)$.

In order to prove \autoref{thm: 2nd Structural results}\ref{enu: Structure of Gamma(n) -- acylindrical hyperbolicity} we use that the groups $\mcg[\pm]{S_*}/\DT^n(S_*)$ are acylindrically hyperbolic for suitable values of $n$, which was proven in \cite{dahmani_dehn_2021}. In fact, these groups are also known to be hierarchically hyperbolic \cite{BHMS}, and  various items of \autoref{thm: 2nd Structural results} then follow immediately from known results. 

Using \autoref{thm: DNB theorem for quotients}, one might hope to translate an understanding of $\Gamma(n)$ back into an understanding of $\mcg[\pm] S/\DT^n(S)$. For instance:

\begin{coro}
\label{cor:residual}
	Let $S$ be a closed, connected, orientable, hyperbolic surface with fundamental group $\Gamma$.
	Then there exists an integer $N>0$ such that, for all non-zero multiples $n$ of $N$, the group $\Gamma(n)$ is residually finite if and only if $\mcg[\pm]{S_*}/\DT^n(S_*)$ is residually finite. 
	If $\Gamma(n)$ is conjugacy separable then $\mcg[\pm] S/\DT^n(S)$ is residually finite.
\end{coro}
\begin{proof}Given a finitely generated group $G$ with trivial centre, $G$ is residually finite if and only if its automorphism group is \cite{baumslag_automorphism_1963}. 
    The first part then follows from Theorems \ref{thm: DNB theorem for quotients} and \ref{thm: 1st Structural results}.
    For the second part we need Grossman's criterion stating that if the group $G$ is conjugacy separable and a certain condition on automorphisms, sometimes called Grossman's property A, holds (namely, that every pointwise inner automorphism of $G$ is inner), then $\out G$ is residually finite \cite{Grossman}. 
    We have $\mcg[\pm] S/\DT^n(S)\cong \out{\Gamma(n)}$ by \autoref{thm: DNB theorem for quotients}, so we need to know that the condition on automorphisms holds for $\Gamma(n)$. This is because $\Gamma(n)$ is acylindrically hyperbolic and has no non-trivial finite normal subgroups by Theorems \ref{thm: 1st Structural results} and \ref{thm: 2nd Structural results}, so we can apply \cite[Corollary 1.5]{AMS}.
\end{proof}

In \cite{BHMS}, it is shown that if all hyperbolic groups are residually finite, then $\mcg[\pm]{\Sigma}/\DT^n(\Sigma)$ is residually finite for $\Sigma=S$ or $\Sigma=S_*$, for suitable $n$.  Hence, \autoref{cor:residual} immediately implies the following conditional result.

\begin{coro}
\label{cor:conditional residual}
	If every hyperbolic group is residually finite, then there exists an integer $N>0$ such that, for all non-zero multiples $n$ of $N$, the group $\Gamma(n)$ is residually finite.
\end{coro}

By \autoref{thm: 1st Structural results}\ref{enu: Structure of Gamma(n) -- limit}, $\Gamma(n)$ is a limit of hyperbolic groups so, if every hyperbolic group is residually finite, then $\Gamma(n)$ is indeed a limit of finite groups in the space of marked groups. However, for infinitely presented groups, being a limit of finite groups does not imply residual finiteness.

\subsection{Questions}

Throughout this subsection, fix a closed surface $S$ of genus at least two.
There are several natural questions that arise. First of all, because of \autoref{cor:residual}, studying residual properties of $\Gamma(n)$ would yield important information on residual properties of quotients of mapping class groups by powers of Dehn twists.

\begin{ques}
    Is $\Gamma(n)$ residually finite for all sufficiently large $n$? Is it conjugacy separable?
\end{ques}

In light of \autoref{cor:conditional residual}, showing that the $\Gamma(n)$ are not residually finite would show that there exists a non-residually finite hyperbolic group. On the other hand, showing that the $\Gamma(n)$ are residually finite would go in the direction of removing the hypothesis of residual finiteness of (enough) hyperbolic groups from the results in \cite{BHMS} and \cite{wilton2024congruence}, and in particular it would be a step towards proving the congruence subgroup property for mapping class groups. 

In a different direction, $\Gamma(n)$ acts on a natural 2-dimensional CW-complex $X_n$ obtained by quotienting $\mathbb H^2$ by the kernel of $\Gamma \to \Gamma(n)$ and gluing in disks. If this complex were aspherical, one could obtain information about homological properties of $\Gamma(n)$ of potential interest, for instance because $\mcg S/\DT^n(S)$ acts on the homology of $\Gamma(n)$. We then ask:

\begin{ques}
    Is $X_n$ as above aspherical for all sufficiently large $n$? Is $X_n$ a classifying space for proper actions? 
\end{ques}

Note that, if $X_n$ is aspherical, then it is not hard to show that $H_2(\Gamma(n),\mathbb Q)$ is infinite dimensional. This would in particular show that $\Gamma(n)$ is infinitely presented for all sufficiently large $n$.

Some of our results hold for all sufficiently large $n$, while others only hold for all multiples of some large number. The underlying reason for this is that the results in \cite{Dahmani:spinning} (that \cite{dahmani_dehn_2021} relies on) are proven under these quantifiers, which are in fact typical of small-cancellation constructions over groups that contain torsion. However, we are not aware of a real obstruction for all our results (and the ones we rely on) to hold for all sufficiently large $n$, and we are investigating this.

Throughout this article we only consider surfaces with at most one puncture, but the Birman exact sequence holds for surfaces with more punctures as well. We expect that our results can be extended to these cases:

\begin{ques}
    Are there versions of the results in this paper that hold for surfaces with arbitrarily many punctures?
\end{ques}

\subsection{Automorphism groups of free groups}

Finally, we discuss the possibility of analogous results in the context of (outer) automorphisms of free groups. Denote by $F_r$ the free group or rank $r$ generated by $\{a_1, \dots, a_r\}$. An element of $F_r$ is \emph{primitive} if it belongs to a free basis of $F_r$. Primitive elements are sometimes seen as the analogues of simple closed curves in the context of free groups, and it is natural to study the quotient $F_r(n)$ of $F_r$ obtained by killing the $n$th power of every primitive element. The analogues of Dehn twists in this context are the \emph{transvections} $\lambda_{ij}$, sending $a_i$ to $a_ja_i$ and fixing $a_k$ for $k\neq i$. One can then define the group $\TV^n(F_r)$ to be the subgroup generated by $n$th powers of transvections, and consider the power quotient $\aut {F_r} /\TV^n(F_r)$.

As in the surface case, these groups have been studied extensively.  Bridson--Vogtmann proved that $\aut {F_r} /\TV^n(F_r)$ is infinite for $r\geq 3$ and all sufficiently large $n$, by noting a map to the automorphism group of the corresponding Burnside group $B(r,n)$, with a copy of  $B(r-1,n)$ in its image \cite{bridson_homomorphisms_2003}. For infinitely many $n$, Malestein--Putman exhibited a complex representation of $\aut {F_r} /\TV^n(F_r)$ with infinite image \cite{malestein_simple_2019}. Farb--Hensel approached the subject via homology of finite-index subgroups \cite{farb_moving_2017}.  So it is natural to wonder how many of our techniques apply in this setting.

A strategy similar to the one we used for $\Gamma(n)$ should prove an analogue of \autoref{thm: 1st Structural results}:  $F_r(n)$ is an infinite, virtually torsion-free limit of hyperbolic groups with trivial centre. However, beyond this we run into two major problems.

First, our proof that the natural map $\mcg[\pm]{S}/\DT^n(S)\to\out{\Gamma(n)}$ is injective relies on (\cite{dahmani_dehn_2021} which in turn relies on) \cite{Dahmani:spinning}, where all arguments rely on properties of annular subsurface projections in mapping class groups. There is no known analogue of those in the free group case.


Second, more seriously, our proof that $\mcg[\pm]{S}\to\out{\Gamma(n)}$ is surjective has no analogue, because $\aut{F_r}\to\aut{F_r(n)}$ is not surjective. 

\begin{exam}\label{exa: Surjectivity fails in the free case}
For any $p$ coprime to $n$, the map $\varphi$ fixing $a_2, \dots, a_r$ and sending $a_1$ to $a_1^p$ defines an automorphism of $F_r(n)$.  If $
\varphi$ descends from an automorphism $\tilde{\varphi}$ of $F_r$, then $\tilde{\varphi}$ acts as a linear map with determinant $p$ on the mod $n$ abelianisation $(\mathbb{Z}/n\mathbb{Z})^r$.
Thus,  if $p \neq \pm 1 \mod n$, this map does not come from an automorphism of $F_r$ (cf.\ \cite[Example~1.11]{Coulon:2021wg}).
\end{exam}

\subsection{Outline of the paper}

The arguments in this paper combine three different toolkits from geometric group theory: geometric small cancellation theory, acylindrical and hierarchical hyperbolicity, and Paulin--Rips--Sela's theory of limiting actions on $\R$-trees.

In \autoref{sec: Metabelian} we construct some useful finite metabelian quotients of $\Gamma(n)$. In \autoref{sec: Virtual first Betti number} we apply a theorem of Klukowski to prove that the virtual first Betti number of $\Gamma(n)$ is positive.
In \autoref{sec: hg} we recall the tools from geometric small-cancellation theory, which will be crucial for our study of $\Gamma(n)$.

The real work starts in \autoref{sec: approx periodic groups}, where we describe $\Gamma(n)$ as a limit of hyperbolic groups in \autoref{res: approximating sequence}. 
We use iterated geometric small cancellation theory to build by induction a sequence of non-elementary hyperbolic groups whose direct limit is exactly $\Gamma(n)$.
This construction is very similar to the one used to produce infinite Burnside quotients of a hyperbolic group, see \cite{olshanskii_periodic_1991, ivanov_hyperbolic_1996, Delzant:2008tu, Coulon:2014fr, Coulon:2018vp}.
Using this sequence of approximations, in \autoref{subsec: Gamma(n) prop} we derive the first main properties of $\Gamma(n)$, including \autoref{thm: 1st Structural results}.

The results of \autoref{subsec: Gamma(n) prop} are used as black boxes in \autoref{sec:diagram} and \autoref{sec: centralisers}; these two sections can be read independently of the previous ones. In \autoref{sec:diagram} we prove \autoref{thm: Birman for quotients}, the Birman exact sequence for $\Gamma(n)$. The key observation in this section is contained in \autoref{prop:big diagram}, which gives a commutative diagram involving the various objects of interest. In \autoref{sec: centralisers} we prove ``half'' of \autoref{thm: DNB theorem for quotients}, meaning injectivity of the natural maps, and deduce \autoref{thm: 2nd Structural results} about further properties of $\Gamma(n)$ including acylindrical hyperbolicity and solvable word problem. The proofs use acylindrical and hierarchical hyperbolicity of $\mcg{S}/DT^n(S)$ \cite{dahmani_dehn_2021, BHMS}.

The rest of the paper is devoted to the proof of surjectivity of the natural map $\aut{\Gamma}\to\aut{\Gamma(n)}$, adapting the methods of Coulon--Sela \cite{Coulon:2021wg}, which builds on Sela's previous work, see e.g.\ \cite{Rips:1994jg, sela_diophantine_2009}. 
In \autoref{sec: graph of groups} we recall fundamental facts about group actions on trees. 
The group $\Gamma(n)$ cannot act on a tree without global fixed points, so one cannot directly apply Sela's technique to this group.
However, using the hyperbolic approximation groups we produced in \autoref{res: approximating sequence}, in \autoref{sec: shortening argument} we adapt the shortening argument to our context.  We then use this in \autoref{sec: lifting} to lift (almost injective) morphisms $\Gamma \to \Gamma(n)$ to $\Gamma \to \Gamma$.

\subsection*{Acknowledgements} We would like to thank Giorgio Mangioni for useful comments on a draft of this paper.
The first author acknowledges support from the Agence Nationale de la Recherche under the grant \emph{GoFR} ANR-22-CE40-0004 as well as the Région Bourgogne-Franche-Comté under the grant \emph{ANER 2024 GGD}.
His institute receives support from the EIPHI Graduate School (contract ANR-17-EURE-0002).

%
\section{Metabelian periodic quotients of \texorpdfstring{$\Gamma$}{Γ}}\label{sec: Metabelian}
%

As in the introduction, let $S$ be a closed connected oriented surface of genus at least two.
Let $\Gamma=\pi_1(S)$ be its fundamental group, and let $\Gamma(n)$ be the quotient of $\Gamma$ by all $n$th powers of simple closed curves.

Before we start to study $\Gamma(n)$ geometrically, we first identify some useful finite quotients of $\Gamma(n)$. The starting point is the construction of an interesting metabelian finite group.

\begin{prop}\label{res: metabelian group}
For each $n> 2$, there is a finite group $Q_n$ with the following properties.
\begin{enumerate}
\item $Q_n$ is generated by 2 elements $A,B$; \label{enu: metabelian group -- 2-generator}
\item $Q_n$ is metabelian; \label{enu: metabelian group -- metabelian}
\item $Q_n$ is $n$-periodic; \label{enu: metabelian group -- n-periodic}
\item the generator $A$ and the commutator $[A,B]$ both have order $n$. \label{enu: metabelian group -- commutator order}
\end{enumerate}
\end{prop}
\begin{proof}
We construct $Q_n$ explicitly as a matrix group over a certain finite ring. Let $\xi$ be a primitive $n$th root of unity, and let $R$ denote the quotient of the ring $\Z[\xi]$ by the ideal generated by $n$. We define two elements $A,B\in GL_2(R)$ by
    \begin{equation*}
        A = \left[
        \begin{matrix}
            \xi & 0 \\
            0 & 1
        \end{matrix}
        \right]\quad , \quad
        B = \left[
        \begin{matrix}
            1 & 1 \\
            0 & 1
        \end{matrix}
        \right]\,,
        \end{equation*}
and let $Q_n$ be the subgroup of $GL_2(R)$ generated by $A$ and $B$. Since $R$ is a finite ring, $Q_n$ is finite by construction, which proves item \ref{enu: metabelian group -- 2-generator}.   

By induction, every element of $Q_n$ is of the form 
\begin{equation*}
        C=\left[
        \begin{matrix}
            \xi^k & z \\
            0 & 1
        \end{matrix}
        \right],
 \end{equation*}
for some $k \in \Z$ and $z \in R$. Now, direct computation shows that any commutator in $Q_n$ is of the form
     \begin{equation*}
       [C_1,C_2]= \left[
        \begin{matrix}
            1 & * \\
            0 & 1
        \end{matrix}
        \right]\,.
    \end{equation*}
    In particular $[Q_n,Q_n]$ is abelian, hence $Q_n$ is metabelian, proving item \ref{enu: metabelian group -- metabelian}. 
    
    To show that $Q_n$ is $n$-periodic, we compute 
     \begin{equation*}
       C^n = \left[
        \begin{matrix}
            \xi^{kn} & \chi_n(\xi^k)z \\
            0 & 1
        \end{matrix}
        \right],
    \end{equation*}
    where $\chi_n \in \Z[X]$ is the polynomial defined by 
    \begin{equation*}
        \chi_n(X) = 1 + X + X^2 + \dots + X^{n-1}.
    \end{equation*}
    As $\xi$ is an $n$th root of unity, $\xi^{kn} = 1$ and $\chi_n(\xi^k)$ equals zero or $n$ in $\Z[\xi]$ (depending on whether $\xi^k\neq 1$ or $\xi^k=1$, respectively).
    Therefore, by the definition of $R$, $C^n = {\rm Id}$, proving \ref{enu: metabelian group -- n-periodic}.

    Finally, we compute the orders of $A$ and $[A,B]$.  It is clear by construction that $A$ has order $n$.  The commutator is
    \begin{equation*}
	[A,B]
        = \left[
        \begin{matrix}
            1 & \xi - 1 \\
            0 & 1
        \end{matrix}
        \right].
    \end{equation*}
    The condition $A^k = {\rm Id}$ yields $k(\xi - 1) = 0$ in $R$.
    In other words, there is $u \in \Z[\xi]$ such that $k(\xi-1) = nu$ in $\Z[\xi]$.
    Let $d=\varphi(n)$, the degree of the $n$th cyclotomic polynomial, and note that $n>2$ implies that $d>1$.
    Since $\xi$ is a primitive $n$th root of unity, $u$ can be written
    \begin{equation*}
        u = u_0 + u_1 \xi + \dots u_{d-1} \xi^{d-1}.
    \end{equation*}
    where $u_0, \dots, u_{d-1} \in \Z$.
    Moreover, $\Q[\xi]$ is a degree $d$ extension of $\Q$.
    Hence the relation $k(\xi-1) = nu$ viewed in $\Q[\xi]$ yields
    \begin{equation*}
        nu_0 = -k,  \quad 
        nu_1 =  k,  \quad \text{and} \quad
        nu_\ell = 0, \quad \forall \ell \in \{2, \dots, d-1\}\,.
    \end{equation*}
    This shows that $k$ is a multiple of $n$, proving item \ref{enu: metabelian group -- commutator order}.
\end{proof}

The following lemma puts all simple curves into a standard form.

\begin{lemm}\label{res: scc standard form}
Let $\Gamma$ be the fundamental group of a closed, connected, orientable, hyperbolic surface and let $F_2=\langle a,b\rangle$ be free of rank 2. For any $\gamma\in G$  represented by a simple closed curve, there is a homomorphism $f:\Gamma\to F_2$ such that:
 \begin{enumerate}
\item $f(\gamma)=a$ if $\gamma$ is non-separating;
\item $f(\gamma)=[a,b]$ if $\gamma$ is separating.
\end{enumerate}
\end{lemm}
\begin{proof}
Consider the standard presentation
    \begin{equation*}
        \Gamma = \group{a_1,b_1, \dots, a_g, b_g \mid [a_1, b_1]\dots [a_g,b_g]}
    \end{equation*}
    of $\Gamma$.   Define $f:\Gamma\to F_2$ by 
   \begin{eqnarray*}
        f(a_1) = f(b_g) = a\\
        f(b_1) = f(a_g) = b
      \end{eqnarray*}     
   and $f(a_i)=f(b_i)=1$ for $1<i<g$, and note that the relation is satisfied, so this is a homomorphism.
   
   If $\gamma$ is non-separating then, by the change-of-coordinates principle, we may assume that $\gamma=a_1$ by applying an automorphism. Likewise, if $\gamma$ is separating, by change of coordinates we may assume that
   \[
   \gamma = [a_1,b_1]\ldots [a_i,b_i]
   \] 
for some $1\leq i<g$. In either case, $f$ is as required.
 \end{proof}

We denote by $\Gamma'= [\Gamma, \Gamma]$ and $\Gamma'' = [\Gamma', \Gamma']$ the first two groups in the derived series of $\Gamma$. Given $n \in \N\setminus\{0\}$, we write $\Gamma^n$ for the (normal) subgroup of $\Gamma$ generated by the $n$th power of every element. The quotient 
\begin{equation}
\label{eqn: metabelian quotient Gamma}
    M(n) = \Gamma / \Gamma'' \Gamma^n
\end{equation}
is the \emph{maximal $n$-periodic metabelian quotient} of $\Gamma$.
The combination of \autoref{res: metabelian group} and \autoref{res: scc standard form} gives precise information about the images of simple closed curves in $M(n)$. 

\begin{prop}
\label{res: metabelian quotient}    
    The image in $M(n)$ of any element $\gamma \in \Gamma$ representing a simple closed curve has order exactly $n$.
\end{prop}
\begin{proof} 
Every $n$-periodic metabelian quotient of $\Gamma$ factors through $M(n)$, so it suffices to exhibit a homomorphism from $\Gamma$ to an $n$-periodic metabelian group sending $\gamma$ to an element of order $n$. Let $f$ be provided by \autoref{res: scc standard form} and let $\rho:F_2\to Q_n$ send $a\mapsto A$ and $b\mapsto B$. Then $\rho\circ f(\gamma)$ has order $n$ in every case, which proves the result.
\end{proof}

\begin{coro}
\label{res: order scc}
	The image in $\Gamma(n)$ of a simple closed curve has order exactly $n$.
\end{coro}

\begin{proof}
	Let $\gamma \in \Gamma$ be a simple closed curve.
	By construction, its order in $\Gamma(n)$ divides $n$.
	However the projection $\Gamma \onto M(n)$ factors through $\Gamma \onto \Gamma(n)$.
	Moreover, the image of $\gamma$ in $M(n)$ has order exactly $n$, whence the result.
\end{proof}

\begin{rema}
Item \ref{enu: metabelian group -- metabelian} of \autoref{res: metabelian group} -- that $Q_n$ is metabelian -- is not needed to prove \autoref{res: order scc}. Indeed, any homomorphism $f:\Gamma\to Q_n$ factors through a smallest characteristic quotient, the intersection of the kernels of every surjection $\Gamma\to Q_n$, and  this quotient would have served equally well to prove \autoref{res: order scc}.
\end{rema}

%
\section{Virtual first Betti number}\label{sec: Virtual first Betti number}
%

The goal of this section is to explain how to use a recent theorem of Klukowski \cite{klukowski_simple_2025} (building on work of Malestein--Putman \cite{malestein_simple_2019}) to prove that, for large multiples $n$, $\Gamma(n)$ has virtually positive first Betti number.

The theorem of Klukowski is as follows. Let $S$ be a hyperbolic surface and $S'\to S$ a finite-sheeted covering. The \emph{simple closed curve homology} of $S'$ is the subspace $H_1^{\mathrm{scc}}(S';\Q) \leq H_1(S';\Q)$ spanned by lifts of simple closed curves in $S$. The following result is \cite[Theorem 1.1]{klukowski_simple_2025}.

\begin{theo}[Klukowski]\label{res: Klukowski}
Let $\Gamma$ be the fundamental group of a closed, orientable, hyperbolic surface. There is a finite-sheeted normal covering $S'$ of $S$ such that $H_1^{\mathrm{scc}}(S';\Q)$ is a proper subspace of $H_1(S';\Q)$.
\end{theo}

We use this to prove:

\begin{prop}\label{res: vb1}
Let $\Gamma$ be the fundamental group of a closed, orientable, hyperbolic surface. There is an $N$ such that, for all positive multiples $n$ of $N$, the group $\Gamma(n)$ has a subgroup of finite index with infinite abelianisation.
\end{prop}
\begin{proof}
Let $\Gamma'=\pi_1(S')$ for $S'$ as in \autoref{res: Klukowski}, let $Q = \Gamma/\Gamma'$ be the deck group and let $N=\#Q$. Consider a positive multiple $n = Nk$. Since $H_1^{\mathrm{scc}}(S';\Q)$ is a proper subspace of $H_1(S';\Q)$, we have a direct-sum factorisation
\[
H_1(S';\Z) = A\oplus \Z
\]
where $A$ contains every lift of a simple closed curve on $S$. In particular, we get a natural homomorphism onto an abelian group
\[
\eta':\pi_1(S') \to C=A/kA \oplus \Z
\]
with the property that, for every simple closed curve $\alpha$ on $S$, $\eta'(\alpha^n)=0$.

By the the Krasner--Kaloujnine theorem (see, e.g.\ \cite[Lemma 2.4]{wilton_virtual_2010}), $\eta'$ induces a homomorphism to the virtually abelian wreath product
\[
\eta:\Gamma \to C\wr Q\,.
\]
By definition, $\eta(\Gamma')$ is contained in the base of the wreath product $\bigoplus_Q C$. For each $q\in Q$, let $\gamma_q$ be a choice of lift to $\Gamma$. For $\alpha\in\Gamma'$, we have by construction
\[
\eta(\alpha) = (\eta'(\gamma_q\alpha \gamma_q^{-1}))_{q\in Q}\,.
\]
For any simple closed curve $\alpha$ on S, the conjugate $\gamma_q\alpha^N\gamma_q^{-1}$ is contained in $\Gamma'$, so $\eta’(\gamma_q\alpha^n\gamma_q^{-1})=1$. Thus, the order of $\eta(\alpha)$ divides $n$.  Furthermore, note that, for any $\beta\in\Gamma'$ with $\eta'(\beta)$ having infinite order, it follows that $\eta(\beta)$ also has infinite order.

This shows that $\eta$ induces a well-defined homomorphism (that we also call $\eta$) from $\Gamma(n)$ to the finitely generated, virtually abelian group $C\wr Q$, with infinite image, and the result follows.
\end{proof}

%
\section{Hyperbolic geometry}
%
\label{sec: hg}

%
\subsection{Actions on hyperbolic spaces}
%

In this section we state the main tool that we will use to study the groups $\Gamma(n)$, namely the small-cancellation theorem, as well as introducing further tools that will allow us to inductively study a sequence of groups ``approximating'' $\Gamma(n)$. 
We start by recalling basic notions of hyperbolic geometry.
For general introductions to the topic we refer the reader to \cite{Bridson:1999ky,Coornaert:1990tj,Ghys:1990ki} or Gromov's original article \cite{Gromov:1987tk}.

\paragraph{Four point inequality.}
Let $X$ be a metric length space.
For every $x,y \in X$, we write $\dist[X] xy$, or simply $\dist xy$, for the distance between $x$ and $y$.
If one exists, we write $\geo xy$ for a geodesic joining $x$ to $y$.
Given $x \in X$ and $r \in \R_+$, we denote by $B(x,r)$ the open ball of radius $r$ centred at $x$, i.e.
\begin{equation*}
	B(x,r) = \set{y \in X}{\dist xy < r}.
\end{equation*}
The \emph{punctured ball} of radius $r$ centred at $x$ is $\mathring B(x,r) = B(x,r) \setminus\{x\}$.
The Gromov product of three points $x,y,z \in X$ is
\begin{equation*}
	\gro xyz = \frac 12 \left[ \dist xz + \dist yz - \dist yz\right].
\end{equation*}
Let $\delta \in \R_+^*$.
For the rest of this section, we always assume that $X$ is $\delta$-hyperbolic, i.e. for every $x,y,z,t \in X$,
\begin{equation}
\label{eqn: four point hyp}
	\min\left\{ \gro xyt, \gro yzt \right\} \leq \gro xzt + \delta.
\end{equation}
We denote by $\partial X$ the boundary at infinity of $X$.
The Gromov product at $z$ naturally extends to $\bar X$.

\paragraph{Quasi-convex subsets.}
Let $\alpha \in \R_+$.
A subset $Y$ of $X$ is \emph{$\alpha$-quasi-convex} if $d(x,Y) \leq \gro y{y'}x + \alpha$, for every $x \in X$, and $y,y' \in Y$.
We denote by $\distV[Y]$ the length metric on $Y$ induced by the restriction of $\distV[X]$ to $Y$.
We say that $Y$ is \emph{strongly quasi-convex} if it is $2 \delta$-quasi-convex and for every $y,y' \in Y$,
\begin{equation}
\label{eqn: def strongly qc}
	\dist[X]y{y'} \leq \dist[Y]y{y'} \leq \dist[X]y{y'} + 8\delta.
\end{equation}
In particular, this forces $Y$ to be connected by rectifiable paths.
%

\paragraph{Isometries.}
An isometry $\gamma$ of $X$ is either \emph{elliptic} (its orbits are bounded), \emph{parabolic} (its orbits admit exactly one accumulation point in $\partial X$), or \emph{loxodromic} (its orbits admit exactly two accumulation points in $\partial X$).
In order to measure the action of an isometry $\gamma$ on $X$ we use the \emph{translation length} and the \emph{stable translation length}, respectively defined by
\begin{equation*}	
	\norm[X] \gamma = \inf_{x \in X}\dist {\gamma x}x
	\quad \text{and} \quad
	\snorm[X] \gamma = \lim_{n \to \infty} \frac 1n \dist{\gamma^nx}x.
\end{equation*}
If there is no ambiguity, we will omit the space $X$ from the notation.
These lengths are related by the following relation; see for instance Coornaert--Delzant--Papadopoulos \cite[Chapitre~10, Proposition~6.4]{Coornaert:1990tj}.
\begin{equation}
\label{eqn: regular vs stable length}
	\snorm \gamma\leq \norm \gamma \leq \snorm \gamma+ 8\delta 
\end{equation}
In addition, $\gamma$ is loxodromic if and only if $\snorm \gamma > 0$.
In such a case, the accumulation points of $\gamma$ in $\partial X$ are
\begin{equation*}
	\gamma^- = \lim_{n \to \infty} \gamma^{-n}x
	\quad \text{and} \quad
	\gamma^+ = \lim_{n \to \infty} \gamma^nx\,.
\end{equation*}
They are the only points of $X\cup\partial X$ fixed by $\gamma$.
The \emph{cylinder} of $\gamma$, denoted by $Y_\gamma$, is the set of all points $x \in X$ which are $20\delta$-close to some $L$-local $(1, \delta)$-quasi-geodesics joining $\gamma^-$ to $\gamma^+$ with $L > 12\delta$.
If $X$ is geodesic and proper, then $Y_\gamma$ is a closed strongly quasi-convex subset of $X$ \cite[Lemma~2.31]{Coulon:2014fr} which can be thought of as the ``smallest'' $\gamma$-invariant quasi-convex subset; see for instance \cite[Section~2.3]{Coulon:2018vp}.

\begin{defi}
\label{def: fix}
	Let $U$ be a set of isometries of $X$.
	The \emph{energy} of $U$ is 
	\begin{equation*}
		\nrj U = \inf_{x \in X} \max_{\gamma \in U} \dist {\gamma x}x.
	\end{equation*}
	Given $d \in \R_+$, we define $\fix{U,d}$ as 
	\begin{equation}
	\label{eqn: def fix}
		\fix{U,d} = \set{x \in X}{\forall \gamma \in U,\ \dist{\gamma x}x \leq d}.
	\end{equation}
\end{defi}

If $d > \max\{\lambda(U), 5\delta\}$, then $\fix{U, d}$ is $8\delta$-quasi-convex.
Moreover, for every $x \in X \setminus \fix{U,d}$, we have
\begin{equation}
\label{eqn: displacement outside fixed set}
	\sup_{\gamma \in U} \dist{\gamma x}x \geq 2d\left(x, \fix{U,d}\right) + d - 10\delta\,;
\end{equation}
see \cite[Lemma~2.8]{Coulon:2018vp}.
For simplicity, we write $\fix{U}$ for $\fix{U,0}$.
If the set $U = \{\gamma\}$ is reduced to a single isometry, then $\nrj U = \norm \gamma$ and we simply write $\fix{\gamma,d}$ for $\fix{U,d}$.

\begin{defi}
\label{def: thin isom}
	Let $\alpha \in \R_+$.
	\begin{itemize}
		\item An elliptic isometry $\gamma$ of $X$ is \emph{$\alpha$-thin at $x \in X$} if, for every $d \in \R_+$, the set $\fix{\gamma, d}$ is contained in $B(x, d/2 + \alpha)$.
		It is \emph{$\alpha$-thin} if it is $\alpha$-thin at some point of $X$.
		\item A loxodromic isometry $\gamma$ of $X$ is \emph{$\alpha$-thin} if, for every $d \in \R_+$ and for every $y \in \fix{\gamma,d}$, we have
		\begin{equation*}
			\gro{\gamma^-}{\gamma^+}y \leq \frac 12 (d - \norm \gamma) + \alpha\,.
		\end{equation*}
	\end{itemize}
\end{defi}

\begin{rema}
\label{rem: thin isom - loxo}
	It follows from the stability of quasi-geodesics that any loxodromic isometry $\gamma$ is $\alpha$-thin, where 
	\begin{equation*}
		\alpha \leq 100\left(\frac{\delta}{\snorm \gamma} + 1\right)\delta. 
	\end{equation*}
	In our context the thinness provides a way to control the action of an isometry $\gamma$ when we do not have any lower bound for its stable norm $\snorm \gamma$.
\end{rema}

The local-to-global statement below is a direct consequence of (\ref{eqn: displacement outside fixed set}).
Its proof is left to the reader.

\begin{lemm}
\label{res: local-to-global thinness}
	Let $\gamma$ be an isometry of $X$ and $\alpha \in \R_+$.
	\begin{enumerate}
		\item Assume that $\gamma$ is elliptic. 
		If $\fix{\gamma, 6\delta}$ is contained in the ball $B(x, \alpha)$ for some $x \in X$, then $\gamma$ is $\beta$-thin at $x$, where $\beta = \alpha + 2\delta$. 
 		\item Assume that $\gamma$ is loxodromic.
		If the set $\fix{\gamma, \norm \gamma + 5\delta}$ is contained in the $\alpha$-neighbourhood of some $L$-local $(1, \delta)$-quasi-geodesic from $\gamma^-$ to $\gamma^+$  with $L > 12\delta$, then $\gamma$ is $\beta$-thin with $\beta = \alpha + 6\delta$. 
	\end{enumerate}
\end{lemm}

\paragraph{Group action.}

Let $\Gamma$ be a group acting by isometries on a $\delta$-hyperbolic length space $X$.
We denote by $\Lambda(\Gamma)$ its \emph{limit set}, i.e. the set of accumulation points in $\partial X$ of the $\Gamma$-orbit of a point $x \in X$.
We say that $\Gamma$ is \emph{non-elementary} (\emph{for its action on $X$}) if $\Lambda(\Gamma)$ contains at least three points.
If $\Lambda(\Gamma)$ is empty (\resp contains exactly one point, exactly two points), then $\Gamma$ is called \emph{elliptic} (\resp \emph{parabolic}, \emph{lineal}).

In order to control the action of $\Gamma$, we use three invariants: the injectivity radius, the acylindricity parameter and the $\nu$-invariant.

\begin{defi}[Injectivity radius]
\label{def: inj rad}
    Let $U$ be a subset of $\Gamma$.
	The \emph{injectivity radius} of $U$ on $X$ is the quantity
	\begin{equation*}
		\inj UX = \inf \set{\snorm[X]\gamma}{\gamma \in U, \ \text{loxodromic}}.
	\end{equation*}
\end{defi}

\begin{defi}
\label{def: acyl inv}
	Let $d \in \R_+$.
	The \emph{(local) acylindricity parameter at scale $d$}, denoted by $A(\Gamma,X,d)$, is defined as 
	\begin{equation*}
		 A(\Gamma,X,d) = \sup_{U \subset \Gamma} \diam {\fix{U,d}},
	\end{equation*}
	where $U$ runs over all subsets of $\Gamma$ generating a non-elementary subgroup.
\end{defi}

\begin{defi}[$\nu$-invariant]
\label{def: nu-inv}
	The $\nu$-invariant $\nu(\Gamma, X)$ is the smallest integer $m$ with the following property:
	for every $\gamma_0, \gamma \in \Gamma$ with $\gamma$ loxodromic, if $\gamma_0$, $\gamma \gamma_0 \gamma^{-1}$, \dots, $\gamma^m \gamma_0 \gamma^{-m}$ generate an elementary subgroup, then so do $\gamma_0$ and $\gamma$.
\end{defi}

Note that, in the notation of the definition, if $\langle \gamma_0,\gamma\rangle$ is elementary then its limit set consists of the two points of $\partial X$ stabilised by $\gamma$.

\begin{lemm}
\label{res: nu invariant vs finite cyclic subgroups}
	Assume that $\Gamma$ has no parabolic subgroup.	
	If every elliptic subgroup of $\Gamma$ is finite cyclic, then $\nu(\Gamma,X) \leq 2$.
\end{lemm}

\begin{proof}
    Let $\gamma_0, \gamma \in \Gamma$ with $\gamma$ loxodromic.
    Assume that the subgroup 
    \begin{equation*}
    	E_2 = \group {\gamma_0, \gamma \gamma_0 \gamma^{-1}, \gamma^2 \gamma_0 \gamma^{-2}}
    \end{equation*}
    is elementary.
    Suppose first that the subgroup $E_1 =  \group {\gamma_0, \gamma \gamma_0 \gamma^{-1}}$ is elliptic, hence cyclic.
    Note that $\gamma_0$ and $\gamma \gamma_0 \gamma^{-1}$ have the same order, hence generate the same subgroup of $E_1$ (which actually coincides with $E_1$).
    Thus $\gamma$ normalises $\group {\gamma_0}$, and $\group{\gamma_0, \gamma}$ is elementary (as a cyclic extension of an elliptic subgroup).
    Suppose now that $E_1$ is lineal.
    In particular, it contains a loxodromic element $\gamma_1$ and $E_2$ is contained in $E(\gamma_1)= \stab{\{\gamma_1^\pm\}}$, the maximal elementary subgroup containing $\gamma_1$.
    However, $\gamma \gamma_1 \gamma^{-1}$ belongs to $E_2$, hence to $E(\gamma_1)$.
    Thus $\gamma$ necessarily belongs to $E(\gamma_1)$, which also contains $\gamma_0$.
    Consequently, $\group{\gamma_0, \gamma}$ is elementary, as required.
\end{proof}
The latter two quantities provide the following analogue of the Margulis lemma for manifolds with pinched negative curvature; see \cite[Proposition~3.5]{Coulon:2018vp}.

\begin{prop}
\label{res: margulis lemma}
	For every $d \in \R_+$, we have
	\begin{equation*}
		A(\Gamma,X, d) \leq \left[ \nu(\Gamma, X) + 3\right]d + A(\Gamma,X,400\delta) + 24\delta\,.
	\end{equation*}
\end{prop}

This motivates the next definition

\begin{defi}[Acylindricity]
\label{def: acylindricity}
	The \emph{(global) acylindricity parameter}, denoted by $A(\Gamma,X)$, is the smallest non-negative number such that, for every $d \in \R_+$,
	\begin{equation*}
		A(\Gamma,X, d) \leq \left[ \nu(\Gamma, X) + 3\right]d + A(\Gamma,X)\,.
	\end{equation*}
\end{defi}

We conclude this section with a fact about hyperbolic groups that is certainly well-known to specialists.
Thus, we only go through the main steps of the proof.

\begin{lemm}
\label{res: max ell sg malnormal}
	Suppose that $\Gamma$ is a hyperbolic group whose elliptic subgroups are abelian.
	The following are equivalent.
	\begin{enumerate}
		\item \label{enu: max ell sg malnormal - max}
		Every maximal elliptic subgroup of $\Gamma$ is malnormal.
		\item \label{enu: max ell sg malnormal - loxo}
		No non-trivial elliptic subgroup of $\Gamma$ is normalised by a loxodromic element.
	\end{enumerate}
\end{lemm}

\begin{proof}
	Note that \ref{enu: max ell sg malnormal - max} $\Rightarrow$ \ref{enu: max ell sg malnormal - loxo} is straightforward.
	Assume now that \ref{enu: max ell sg malnormal - loxo} holds.
	One proves first that $\Gamma$ is commutative transitive: that is, for every $\gamma_1, \gamma_2, \gamma_3 \in \Gamma \setminus \{1\}$, if $[\gamma_1, \gamma_2] = [\gamma_2, \gamma_3] = 1$, then $[\gamma_1, \gamma_3] = 1$.
	Indeed, for three such elements, $\group{\gamma_2}$ is central in $\Gamma_0 = \group{\gamma_1, \gamma_2, \gamma_3}$.
	It follows from our assumption that either $\gamma_2$ is elliptic, and thus so is $\Gamma_0$, or $\gamma_2$ is loxodromic, and then $\Gamma_0$ is cyclic.
	In both cases $\gamma_1$ and $\gamma_3$ commute.
	
	Consider now a maximal elliptic subgroup $E$ of $\Gamma$ and an element $\gamma \in \Gamma$ such that $E \cap \gamma E \gamma^{-1} \neq\{1\}$.
	By commutative transitivity, $\group {E, \gamma E \gamma^{-1}}$ is abelian.
	Note that $\group {E, \gamma E \gamma^{-1}}$ is elliptic.
	Indeed, it follows from \ref{enu: max ell sg malnormal - loxo} that every infinite abelian subgroup is cyclic, hence cannot contain a non-trivial elliptic subgroup.
	By maximality we deduce that $\gamma$ normalises $E$.
	Since $E$ is non-trivial, we get from \ref{enu: max ell sg malnormal - loxo} that $\gamma$ is elliptic. 	Therefore $\group{E, \gamma}$ is elliptic as well.
	Using again the maximality of $E$, we conclude that $\gamma \in E$.
\end{proof}

%
\subsection{Geometric small-cancellation theory}
\label{sec: sc}
%

\subsubsection{The small-cancellation theorem}

Let $Y$ be a metric length space and let $\rho>0$. 
The \emph{cone of radius $\rho$ over $Y$}, denoted by $Z_\rho(Y)$ or simply $Z(Y)$, is the quotient of $Y\times [0,\rho]$ by the equivalence relation that identifies all the points of the form $(y,0)$ for $y \in Y$.
If $x\in Z (Y)$, we write $x=(y,r)$ to say that $(y,r)$ represents $x$. 
We let $c=(y,0)$ be the \emph{apex} of the cone.
There is a metric on $Z(Y)$ that is characterised as follows \cite[Chapter~I.5]{Bridson:1999ky}. 
Let $x=(y,r)$ and $x'=(y',r')$ in $Z(Y)$. Then 
\begin{equation}
\label{eqn: metric cone}
	\cosh \dist x{x'} = \cosh r \cosh r' -\sinh r \sinh r' \cos\left( \min\left\{ \pi,\, \frac{\dist y{y'}}{\sinh\rho}\right\} \right).
\end{equation}

If $H$ is a group acting by isometries on $Y$, then $H$ acts by isometries on $Z(Y)$ by $hx=(hy,r)$. 
Note that $H$ fixes the apex of the cone.

From now on, we assume that $X$ is a proper, geodesic, $\delta$-hyperbolic space, where $\delta >0$.
We consider a group $\Gamma$ that acts properly co-compactly by isometries on $X$.
Let $\mathcal Q$ be a collection of  pairs $(H, Y) $ such that $Y$ is closed strongly-quasi-convex in $X$ and $H$ is a subgroup of $\stab Y$ acting co-compactly on $Y$. 
We suppose that:
\begin{itemize}
	\item $\stab Y$ is an infinite cyclic subgroup of $\Gamma$, for every $(H,Y)\in\mathcal Q$;
	\item $\mathcal Q$ is closed under the action of $\Gamma$ given by the rule $\gamma(H,Y)=(\gamma H\gamma^{-1}, \gamma Y)$; and
	\item $\mathcal Q/\Gamma$ is finite.
\end{itemize}
Furthermore, we let 
\begin{equation*}
	\Delta(\mathcal Q, X) = \sup \left\{ \diam {Y_1^{+5\delta} \cap Y_2^{+5\delta} } \mid (H_1, Y_1) \neq (H_2,Y_2) \in \mathcal Q\right\}
\end{equation*}
and
\begin{equation*}
	T(\mathcal Q, X)=\inf \{\|  h \|\mid h \in H \setminus \{1\},\, (H,Y) \in Q\}\,.
\end{equation*}
These two quantities respectively play the role of the lengths of the longest piece and the shortest relation in the context of usual small-cancellation theory.

\medskip
We fix a parameter $\rho > 0$, whose value will be made precise later.
Let $(H,Y) \in \mathcal Q$. We denote by $\distV[Y]$ the length metric on $Y$ induced by the restriction of $\distV$ to $Y$.
We write $Z(Y)$ for the cone of radius $\rho$ over the metric space $(Y, \distV[Y])$.
The \emph{cone-off space} of radius $\rho$ over $X$ relative to $\mathcal Q$, denoted by $\dot X_\rho(\mathcal Q)$ or simply $\dot X$, is the space obtained by gluing, for each pair $(H,Y)\in \mathcal Q$, the cone $Z_\rho(Y)$ on $Y$ along the natural embedding $\iota\colon Y \to Z_\rho(Y)$. 
We write $\mathcal C$ for the set of apices of $\dot X$.  
We endow $\dot X$ with the largest  metric $\distV[\dot X]$ such that the map $X\to \dot X$  and the maps $Z(Y)\to \dot X$ are $1$-Lipschitz \cite[Section 5.1]{Coulon:2014fr}.
The action of $\Gamma$ on $X$ then extends to an action by isometries on $\dot X$: given any $\gamma \in \Gamma$, a point $x= (y,r)$ in $Z(Y)$ is sent to the point $\gamma x=(\gamma y,r)$ in $Z(\gamma Y)$.

 \label{sec: quotient space}
Let $K$ be the (normal) subgroup generated by the subgroups $H$, where $(H,Y)$ runs over $\mathcal Q$.  
We let $\bar X= \dot X/K$ and $\bar \Gamma = \Gamma/K$. 
We denote by $\zeta$ the projection of $\dot X$ onto $\bar X$ and write $\bar x$ for $\zeta(x)$ for short. Furthermore, we denote by $\bar {\mathcal C}$ the image in $\bar X$ of the apices $\mathcal C$. 
We consider $\bar X$ as a metric space equipped with the quotient metric.
Note that the action of $\Gamma$ on $\dot X$ induces an action by isometries of $\bar \Gamma$ on $\bar X$.
The following theorem summarises Proposition~3.15 and Theorem~6.11 of \cite{Coulon:2014fr}.

\begin{theo}[Small-Cancellation Theorem]
\label{res: small cancellation theorem} 
	There exist $\delta_0$, $\delta_1$, $\Delta_0$, $\rho_0>0$ (that do not depend on $\Gamma$, $X$, or $\mathcal Q$) with the following properties.
	Let $\rho \geq \rho_0$.
	If  $\delta \leq \delta_0$, $\Delta(\mathcal Q, X) \leq \Delta_0$,  and $T(\mathcal Q, X) \geq 10\pi \sinh\rho$, then the following hold.
	\begin{enumerate}
		\item \label{enu: small cancellation theorem - hyp cone-off}
		$\dot X$ is $\dot \delta$-hyperbolic, with $\dot \delta \leq \delta_1$.
		\item \label{enu: small cancellation theorem - hyperbolicity}
		$\bar X$ is proper, geodesic, and  $\bar \delta$-hyperbolic, with $\bar \delta \leq \delta_1$.
		The action of $\bar \Gamma$ on $\bar X$ is proper and co-compact.
		\item \label{enu: small cancellation theorem - local isom}
		Let $r\in (0,\rho/20]$ and $x \in \dot X$. 
		If $\dist x{\mathcal C}\geq 2r$, then the projection $\zeta \colon \dot X \to \bar X$ induces an isometry from $B(x,r)$ onto $B(\bar x,r)$.
		\item \label{enu: small cancellation theorem - cone point}
		Let $(H,Y) \in \mathcal Q$. 
		If $c \in \mathcal C$ is the apex of the cone $Z(Y)$, then the projection $\Gamma \onto \bar \Gamma$ induces an isomorphism from $\stab Y/H$ onto $\stab{\bar c}$.
		\item \label{enu: small cancellation theorem - inj kernel}
		For every $x \in\dot X$, for every $\gamma \in K\setminus\{1\}$, we have $\dist{\gamma x}x \geq \min\{2r, \rho/5\}$, where $r = d(x, \mathcal C)$.
	\end{enumerate}
\end{theo}

\begin{rema}
	Note that the constants $\delta_0$ and $\Delta_0$ (\resp $\rho_0$) can be chosen arbitrarily small (\resp large).
	For the rest of this section, we will always assume that $\rho_0 > 10^{20}  \delta_1$ whereas $\delta_0, \Delta_0 < 10^{-10}\delta_1$.
	These estimates are absolutely not optimal.
	We chose them very generously to be sure that all the inequalities that we might need later will be satisfied.
	What really matters is their orders of magnitude, summarised below.
	\begin{equation*}
		\max\left\{\delta_0, \Delta_0\right\} \ll \delta_1  \ll \rho_0 \ll \pi \sinh \rho_0\,.
	\end{equation*}
	Up to increasing the value of $\dot \delta$ or $\bar \delta$ we can always assume that $\dot \delta = \bar \delta$.
	Nevertheless, we will keep two distinct notations to remember which space we are working in.
\end{rema}

\begin{nota}
	In this section we work with three metric spaces, namely $X$, its cone-off $\dot X$, and the quotient $\bar X$.
	Since the map $X \into \dot X$ is an embedding we use the same letter $x$ to designate a point of $X$ and its image in $\dot X$.
	We write $\bar x$ for its image in $\bar X$.
	Unless stated otherwise, we keep the notation $\distV$ (without mentioning the space) for the distances in $X$ or $\bar X$.
	The metric on $\dot X$ will be denoted by $\distV[\dot X]$.
\end{nota}

For the remainder of \autoref{sec: sc}, we assume that $\Gamma$, $X$ and $\mathcal Q$ satisfy the assumption of \autoref{res: small cancellation theorem}. We now complete the description of $\bar \Gamma$ and $\bar X$.

\subsubsection{Lifting morphisms}

The next statement is a reformulation of \cite[Proposition~A.5]{Coulon:2021wg}, and it will be crucial to show that all automorphisms of $\Gamma(n)$ come from automorphisms of $\Gamma$.

\begin{prop}
\label{res: sc - lifting morphism}
	Let $F$ be the free group generated by a finite set $S$.
	Let $\ell \in \N \setminus\{0\}$ and $V\subset F$ be the closed ball of radius $\ell$ (for the word metric relative to $S$).
	Let $\varphi \colon F \to \bar \Gamma$ be a morphism whose image does not fix a cone point in $\bar {\mathcal C}$.
	Assume that  
	\begin{equation*}
		\lambda_\infty\left(\varphi(S)\right) < \frac {\rho}{100\ell}\,,
	\end{equation*}
	where the energy is measured with the metric of $\bar X$. 
	Then there exists a map $\tilde \varphi \colon F \to \Gamma$ such that $\varphi = \pi \circ \tilde \varphi$ and $V \cap \ker \varphi = V \cap \ker \tilde \varphi$.
\end{prop}

\subsubsection{Isometries}

\paragraph{Angle cocycle.}
Let $(H,Y) \in \mathcal Q$.
Let $c$ be the apex of the cone $Z(Y)$ and $\bar c$ its image in $\bar X$.
We now focus on the behaviour of $\stab{\bar c}$.
According to our assumptions, $\stab Y$ is an infinite cyclic subgroup of $\Gamma$, so that $\stab {\bar c} \cong \stab Y / H$ is finite cyclic.
We think of $\stab{\bar c}$ acting on $B(\bar c, \rho)$ as a rotation group acting on a hyperbolic disc.
Let us make this idea more rigorous.

Let $(z_n)$ be a sequence points in $X$ converging to one of the end points $\xi$ in $\partial X$ of $Y$.
Fix a non-principal ultra-filter $\omega$.
We define a Busemann cocycle $b \colon X \times X \to \R$ at $\xi$ by
\begin{equation*}
	b(x,x') = \limo \left[\dist x{z_n} - \dist {x'}{z_n}\right], \quad \forall x,x' \in X.
\end{equation*}
Recall that $\stab Y$ is assumed to be infinite cyclic, hence it fixes $\xi$.
However, the cocycle $b$ may not be $\stab Y$-invariant.
Nevertheless $\stab Y$ is amenable.
Proceeding as in \cite[Section~4.3]{Coulon:2018vp} we can average the orbit of $b$ under $\stab Y$ to get a new cocycle $b_\xi \colon X \times X \to \R$, with the following properties:
for every $x,x' \in X$, for every $\gamma \in \stab Y$, we have
\begin{enumerate}
	\item $\abs{b_\xi(x,x') - b(x,x')} \leq 6\delta$\,,
	\item $b_\xi(\gamma x, \gamma x') = b_\xi(x,x')$\,, and
	\item $\abs{b_\xi(x,\gamma  x)} = \snorm \gamma $.
\end{enumerate}
We rescale $b_\xi$ to get a $\stab Y$-invariant cocycle $\theta_c \colon Y \times Y \to \R$ given by 
\begin{equation*}
	\theta_c(y,y') = \frac 1 {\sinh \rho} b_\xi(y,y'), \quad \forall y,y' \in Y,
\end{equation*}
which we extend radially to a cocycle 
\begin{equation*}
	\theta_c \colon Z(Y)\setminus\{c\} \times Z(Y)\setminus\{c\} \to \R
\end{equation*}
as follows:
if $x = (y,r)$ and $x' = (y',r')$ are two points of $Z(Y)\setminus\{c\}$, then $\theta_c (x,x') = \theta_c(y,y')$.
Let 
\begin{equation*}
	\Theta_{\bar c} = \frac 1{\sinh \rho} \inf_{h \in H \setminus\{1\}} \snorm h\,.
\end{equation*}
According to the (small-cancellation) assumption of \autoref{res: small cancellation theorem} we have $\norm h \geq 10\pi \sinh \rho$, for every $h \in H \setminus \{1\}$.
Recall that stable translation length is slightly shorter that usual translation length, see (\ref{eqn: regular vs stable length}), hence $\snorm h \geq 10\pi \sinh \rho - 8\delta$.
Since we chose $\rho_0$ very large compared to $\delta$, we get $\Theta_{\bar c} \geq 9\pi$. By construction, $\theta_c$ induces a $\stab {\bar c}$-invariant cocycle 
\begin{equation*}
	\theta_{\bar c} \colon \mathring B(\bar c, \rho) \times \mathring B(\bar c, \rho) \to \R/ \Theta_{\bar c} \Z\,.
\end{equation*}
The map $\stab {\bar c}\to  \R/ \Theta_{\bar c} \Z$ sending $\bar \gamma$ to $\theta_{\bar c}(\bar \gamma \bar x, \bar x)$ is an injective homomorphism, which does not depend on the point $\bar x \in \mathring B(\bar c, \rho)$.

Consider the action of $\stab Y$ on $\R$ defined as follows.
Any element $\gamma \in \stab Y$ acts on $\R$ by translation of length $b_\xi(x, \gamma x)$ for some (hence any) $x \in X$.
The quotient $\R / H$ is isometric to a circle whose length is $\Theta_{\bar c} \sinh\rho$.
We denote by $\mathcal D$ the cone of radius $\rho$ over $\R/ H$ and by $o$ its apex.
Observe that $\mathcal D$ is a hyperbolic cone (i.e.\ locally isometric to $\H^2$ everywhere except maybe at the cone point $o$) whose total angle at $o$ is $\Theta_{\bar c}$.

Let $y_0$ be an arbitrary point in $Y$.
One checks that the map $\chi \colon Y \to \R$ sending $y$ to $b_\xi(y_0, y)$ is a  $\stab Y$-equivariant $(1, 150\delta)$-quasi-isometric embedding; see Coulon \cite[Section~4.3]{Coulon:2018vp}.
Hence it induces a $\stab {\bar c}$-equivariant $(1, 150\delta)$-quasi-isometric embedding
\begin{equation*}
	\bar \chi \colon Z(Y/H) \to \mathcal D.
\end{equation*}
Note that $Z(Y/H)$, which is also isometric to $Z(Y)/H$, is endowed here with the metric defined by (\ref{eqn: metric cone}).
Although $Z(Y)/H$ can be identified, as a set of points, with the closed ball of radius $\rho$ centred at $\bar c$, the metric may be slightly different.
Nevertheless, in view of \cite[Propositions~A.2 and A.7]{Coulon:2021wg}, the map $\bar \chi \colon B(\bar c, \rho) \to \mathcal D$ induced by $\bar \chi \colon Z(Y/H) \to \mathcal D$ is a $(1,10\bar \delta)$-quasi-isometric embedding.
Moreover, for every $\bar x \in B(\bar c, \rho)$, we have $\dist{\bar c}{\bar x} = \dist o {\bar \chi(\bar x)}$.
As a consequence we get the following statement.

\begin{prop}
\label{res: sc - angle cocycle}
	For every $\bar x, \bar x' \in \mathring B(\bar c, \rho)$, we have  $\abs{\dist {\bar x}{\bar x'} - \ell} \leq 10 \bar \delta$, where
	\begin{equation}
	\label{eqn: family axioms - conical - 2}
		\cosh \ell = \cosh\dist {\bar c}{\bar x}\cosh\dist {\bar c}{\bar x'} - \sinh\dist {\bar c}{\bar x}\sinh\dist {\bar c}{\bar x'}  \cos\left( \min \left\{ \pi, \tilde \theta_{\bar c}(\bar x,\bar x') \right\}\right)
	\end{equation}
	and $ \tilde \theta_{\bar c}(\bar x,\bar x')$ is the unique representative of $\theta_{\bar c}(\bar x,\bar x')$ in $(-\Theta_{\bar c}/2, \Theta_{\bar c}/2]$.
	Moreover, $\theta_{\bar c}(\bar \gamma \bar x, \bar x) \neq 0$, for every $\bar x \in \mathring B(\bar c,\rho)$ and $\bar \gamma \in \stab {\bar c} \setminus \{1\}$.
\end{prop}

\paragraph{Thinness.} The following lemma gives us control on thinness of elements of $\bar \Gamma$.

\begin{lemm}
\label{res: sc - thin}
	Assume that there is $\beta \in \R_+$ such that every non-parabolic element $\gamma \in \Gamma$ is $\beta$-thin (for its action on $X$).
	Let $\bar \beta = \beta + 10\pi \sinh (100\bar \delta)$.
	Then for every $\bar \gamma \in \bar \Gamma\setminus\{1\}$, the following hold.
	\begin{enumerate}
		\item If $\bar \gamma$ is elliptic and does not fix a cone point in $\bar{\mathcal C}$, then $\bar \gamma$ is $\bar \beta$-thin.
		\item If $\bar \gamma$ fixes a cone point $\bar c \in \bar{\mathcal C}$, then it is $(\rho + \bar \beta)$-thin at $\bar c$.
		\item If $\bar \gamma$ is loxodromic, then $\bar \gamma$ is $\bar \beta$-thin.
	\end{enumerate}
\end{lemm}

\begin{proof}
	The three items respectively correspond to Propositions~A.11, A.15, and A.18 in \cite{Coulon:2021wg}.
	In  \cite{Coulon:2021wg}, Coulon and Sela work under the following additional global hypothesis: $\Gamma$ has no even torsion; every infinite elementary subgroup of $\Gamma$ is cyclic; and $\stab Y / H$ has no even torsion, for every $(H,Y) \in \mathcal Q$.
	Nevertheless, Propositions~A.11 and A.18 do not use any of these assumptions.
	Proposition~A.15 only uses the fact that $\stab Y$ is cyclic, for every $(H, Y) \in \mathcal Q$.
	Consequently, their conclusions still hold in our context.
\end{proof}

\subsubsection{Structure of elementary subgroups} We now analyse elliptic subgroups of $\bar\Gamma$.

\begin{lemm}[See for instance Coulon {\cite[Lemma~4.25]{Coulon:2018vp}}]
\label{res: projection elliptic}
	Let $F$ be a an elliptic subgroup of $\Gamma$.
	The projection induces an isomorphism from $F$ onto its image.
\end{lemm}

Given an elementary subgroup $\bar E \subset \bar \Gamma$, we say that $\bar E$ \emph{lifts} if there is an elementary subgroup $E$ of $\Gamma$ such that the projection $\Gamma \onto \bar \Gamma$ induces an isomorphism from $E$ onto $\bar E$.
Such a subgroup $E$ is called a \emph{lift} of $\bar E$.
Proposition~6.12 in \cite{Coulon:2014fr}, that we recall below, provides a criterion to lift elliptic subgroups of $\bar \Gamma$.

\begin{prop}
\label{res: lifting elliptic subgroups}
	Let $\bar F$ be an elliptic subgroup of $\bar \Gamma$.
	Either $\bar F$ lifts, or there is a vertex $\bar c \in \bar{\mathcal C}$ such that $\bar F$ is contained in $\stab{\bar c}$.
\end{prop}

The next statement focusses on lineal subgroups.
Recall that if $\bar E$ is a lineal subgroup of $\bar G$, then it fits into a short exact sequence 
\begin{equation*}
	1 \to \bar F \to \bar E \to \mathbf L \to 1
\end{equation*}
where $\bar F$ is the maximal elliptic normal subgroup of $\bar E$ and $\mathbf L$ is either $\Z$ of $\dihedral$.
Therefore, if we want to lift the lineal subgroup $\bar E$ into $\Gamma$, the first step is to ensure that one can lift $\bar F$ to an elliptic subgroup of $\Gamma$ in such a way that the action by conjugation of $\bar E$ on $\bar F$ also lifts.
This is essentially the purpose of the next statement in the case where $\mathbf L = \Z$.
For a more general study of lineal subgroups, we refer the reader to \cite[Section~4.6]{Coulon:2018vp}.

\begin{prop}
\label{res: lifting normalised elliptic sbgp}	
	Let $\bar F$ be a non-trivial elliptic subgroup of $\bar \Gamma$ that is normalised by a loxodromic element.
	Then $\bar F$ admits a lift in $\Gamma$ that is also normalised by a loxodromic element.
\end{prop}

\begin{proof}
Recall that we assume that for every $(H,Y) \in \mathcal Q$, the group $\stab Y$ is infinite cyclic.

	Let $\bar F$ be an elliptic subgroup of $\bar \Gamma$ normalised by a loxodromic element, say $\bar \gamma \in \bar \Gamma$.
	Without loss of generality, we can assume that $\bar F$ is non-trivial.
	For simplicity we let $\bar Z = Y_{\bar \gamma}$, seen as a subset of $\bar X$.
	Note that $\bar Z$ is contained in $\fix{\bar F, 100\bar \delta}$; see for instance \cite[Lemma~2.15]{Coulon:2018vp}.
	
	We claim that $\bar Z$ does not intersect $B(\bar c, \rho - 50\bar \delta)$, for any $\bar c \in \bar{\mathcal C}$.
	Assume on the contrary that it does.
	Since any two distinct vertices in $\bar {\mathcal C}$ are $2\rho$ far apart, it follows from the triangle inequality that $\bar F \subset \stab{\bar c}$.
	Let $(H,Y) \in \mathcal Q$ be such that the apex $c$ of $Z(Y)$ is a pre-image of $\bar v$.
	According to our assumption, the group $\stab Y$ is cyclic.
	Hence so is $\stab{\bar c} = \stab Y / H$, by \autoref{res: small cancellation theorem}\ref{enu: small cancellation theorem - cone point}.
	It follows from the description of apex stabilisers given in \cite[Proposition~4.12]{Coulon:2018vp} that $\bar c$ is the unique apex fixed by $\bar F$.
	However since $\bar \gamma$ normalises $\bar F$, the latter also fixes $\bar \gamma^k \bar c$, for every $k \in \Z$.
	We have reached a contradiction, which completes the proof of our claim.
	
	According to \cite[Proposition~3.21]{Coulon:2014fr}, there is a subset $Z \subset \dot X$ such that the map $\zeta \colon \dot X \onto \bar X$ induces an isometry from $Z$ onto $\bar Z$.
	Moreover, the projection $\pi \colon \Gamma \onto \bar \Gamma$ induces an isomorphism from $\stab Z$ onto $\stab{\bar Z}$.
	Denote by $F$ and $\gamma$ the respective pre-images of $\bar F$ and $\bar \gamma$ in $\stab Z$.
	In particular, $\gamma$ is an infinite order element (hence a loxodromic element of $\Gamma$) normalizing $F$.
\end{proof}

%
\subsubsection{Group invariants}
%

Finally, we obtain quantitative control on acylindricity and injectivity radius in $\bar\Gamma$.

\begin{prop}
\label{res: acyl quotient}
	The acylindricity parameter for the action of $\bar \Gamma$ on $\bar X$ satisfies
	\begin{equation*}
		A(\bar \Gamma, \bar X) \leq A(\Gamma, X) + \left[\nu(\Gamma,X) +4 \right]\pi \sinh(1000 \bar \delta)\,.
	\end{equation*}
\end{prop}

\begin{proof}
	Although the definition of $A(\Gamma, X)$ is slightly different, the proof works verbatim as in \cite[Proposition~5.30]{Coulon:2014fr}.
	See also \cite[Proposition~4.47]{Coulon:2018vp}.
\end{proof}

\begin{prop}
\label{res: injectivity radius}
	Let $\bar U$ be a subset of $\bar\Gamma$ and $U$ its pre-image in $\Gamma$.
	We denote by $\ell$ the infimum over the stable translation lengths (in $X$) of loxodromic elements of $U$ which do not belong to $\stab Y$ for some $(H,Y) \in \mathcal Q$.	
	Then 
	\begin{equation*}
		\inj{\bar U}{\bar X} \geq \min\left\{ \kappa \ell, \bar \delta\right\},
	\end{equation*}
	where $\kappa = \bar \delta /\pi \sinh (26\bar \delta)$.
\end{prop}

\begin{rema}
	By convention, if $\bar U$ does not contain any loxodromic element then $\inj{\bar U}{\bar X}$ is infinite, in which case the statement is void.
\end{rema}

\begin{proof}	
	In \cite[Proposition~6.13]{Coulon:2014fr}, we estimated $\inj {\bar \Gamma}{\bar X}$. 
	See also \cite[Proposition~5.51]{Coulon:2016if} or \cite[Proposition~4.45]{Coulon:2018vp}, for sharper estimates.
	The proof works in the same way.
	The important additional assumption here is that the pre-image of any element $\bar u \in \bar U$ belongs to $U$.
\end{proof}

%
\section{The group \texorpdfstring{$\Gamma(n)$}{Γ(n)} as a limit of hyperbolic groups}
%
\label{sec: approx periodic groups}

%
\subsection{A preferred class of groups}
%

\paragraph{General setting.}

\medskip
Let $\delta, \rho >0$.  In this section we consider triples $(\Gamma,X,\mathcal C)$ where:
\begin{itemize}
	\item $X$ is a $\delta$-hyperbolic length space and $\Gamma$ a group acting by isometries on $X$; and
	\item $\mathcal C$ is a $\Gamma$-invariant, $2\rho$-separated subset of $X$ -- that is $\dist c{c'} \geq 2\rho$ for all distinct $c,c' \in \mathcal C$.
\end{itemize}

\begin{voca}
	The elements of $\mathcal C$ are called \emph{apices} or \emph{cone points}.
	An element $\gamma \in \Gamma$ is \emph{conical} if it fixes a unique cone point.
	It is \emph{visible} if it is loxodromic or conical and \emph{elusive} otherwise.
	A subgroup $\Gamma_0 \subset \Gamma$ is \emph{elusive} if all its elements are elusive. We think of $\mathcal C$ as a thin-thick decomposition of $X$:
	the \emph{thin parts} are the balls $B(c, \rho)$, where $c$ runs over $\mathcal C$;
	the \emph{thick part}, that we denote by $X^+$, is 
	\begin{equation*}
		X^+ = X \setminus \bigcup_{c \in \mathcal C} B(c,\rho)\,.
	\end{equation*}
\end{voca}

We chose the word ``elusive'' to stress that fact that, in practice, the action of $\Gamma_0$ on $X$ will not provide much information on the structure of $\Gamma_0$.
	Since $\mathcal C$ is $\Gamma$-invariant, the collection of visible elements is invariant under conjugation.
	Hence, so is the collection of elusive subgroups.
	The latter is also stable under taking subgroups.

\begin{defi}
\label{def: preferred class}
We write $\mathfrak H(\delta, \rho)$  for the set of all triples $(\Gamma, X, \mathcal C)$ as above which satisfy the following additional properties.
\begin{enumerate}[label=(H\arabic*)]
	\item \label{enu: family axioms - acylindricity}
	 $A(\Gamma, X) \leq \delta$.
	
	\item \label{enu: family axioms - elem subgroups}
	Every elliptic subgroup of $\Gamma$ is cyclic.
	No non-trivial elliptic subgroup can be normalised by a loxodromic element.

	\item \label{enu: family axioms - elusive}
	Every non-trivial elusive element of $\Gamma$ is elliptic and $\delta$-thin.
	
	\item \label{enu: family axioms - conical - 1}
	Every element $\gamma \in \Gamma \setminus\{1\}$ fixing a cone point $c \in \mathcal C$ is $(\rho+ \delta)$-thin at $c$.
	
	\item \label{enu: family axioms - conical - 2}
	For every $c \in \mathcal C$, there exists $\Theta_c \in [2\pi, \infty)$ and a $\stab c$-equivariant cocycle
	\begin{equation*}
		\theta \colon \mathring B(c,\rho)  \times \mathring B(c,\rho)\to \R/ \Theta_c\Z
	\end{equation*}
	such that, for every $x,y \in \mathring B(c,\rho)$, we have $\abs{\dist xy - \ell} \leq \delta$, where
	\begin{equation}
		\cosh \ell = \cosh\dist cx\cosh\dist cy - \sinh\dist cx\sinh\dist cy \cos\left( \min \left\{ \pi, \tilde \theta(x,y) \right\}\right)
	\end{equation}
	and $\tilde \theta(x,y)$ is the unique representative of $\theta(x,y)$ in $(-\Theta_c/2, \Theta_c/2]$.
	Moreover, $\theta(\gamma x, x) \neq 0$, for every $\gamma \in \stab c\setminus \{1\}$, and $x \in \mathring B(c,\rho)$,
	
	\item \label{enu: family axioms - loxodromic}
	Every loxodromic element of $\Gamma$ is $\delta$-thin.
	\item \label{enu: family axioms - radial proj}
	For every $c \in \mathcal C$, for every $x \in B(c,\rho)$, there exists $y \in X^+$ such that $\dist xy = \rho - \dist xc$.
\end{enumerate}

\end{defi}

\begin{rema}
Let us comment on this set of axioms.
Note that we do not require $\Gamma$ to act properly on $X$.
All we will need is a weak form of acylindricity captured by \ref{enu: family axioms - acylindricity}.
This assumption is weaker than the usual acylindricity.
Indeed it does not provide any control on the injectivity radius for the action of $\Gamma$ on $X$.
Axiom~\ref{enu: family axioms - elem subgroups} refers to the algebraic structure of elementary subgroups of $\Gamma$.
Axioms~\ref{enu: family axioms - elusive}-\ref{enu: family axioms - loxodromic} give geometric information on the action of $\Gamma$.
In particular, it follows from \ref{enu: family axioms - elusive} and \ref{enu: family axioms - conical - 1} that $\Gamma$ has no parabolic subgroup.
In Axiom~\ref{enu: family axioms - conical - 2}, by cocycle we mean that for every $x,y,z \in \mathring B(c,\rho)$ we have $\theta(x,z) = \theta(x,y) + \theta(y,z)$.
The exact formula in (\ref{eqn: family axioms - conical - 2}) does not really matter, the point is that it provides uniform control on the action of $\stab c$, i.e.\ control that does not depend on the triple $(\Gamma, X, \mathcal C)$. 
It has the following interpretation:
the ball $B(c,\rho)$ is seen approximately as a cone of radius $\rho$ endowed with a hyperbolic metric whose total angle at the apex is $\Theta_c$.
The quantity $\theta(x,y)$ represents the angle at $c$ between $x$ and $y$.
With this idea in mind, one recognises in (\ref{eqn: family axioms - conical - 2}) the law of cosines in $\H^2$.
\end{rema}

We sometimes make an abuse of notation and write $\mathfrak H(\delta, \rho)$ to designate the set of groups $\Gamma$ instead of the set of triples $(\Gamma, X, \mathcal C)$.
One should keep in mind that such a group always comes with a preferred action on a hyperbolic space.
This is useful for defining energies (see below).
The class $\mathfrak H(\delta, \rho)$ is stable under taking subgroups in the following sense: if $(\Gamma, X, \mathcal C)$ is an element of $\mathfrak H(\delta, \rho)$ and $\Gamma_0$ a subgroup of $\Gamma$, then $(\Gamma_0, X, \mathcal C)$ also belongs to $\mathfrak H(\delta, \rho)$.

\begin{rema}
\label{rem: comparison with CS}
	This class of group is a variation of the class, also called $\mathfrak H(\delta, \rho)$, defined in \cite[Section~3.1]{Coulon:2021wg}.
	Let us highlight here the main differences between the two definitions.
	In this article, we study the quotient $\Gamma(n)$ for every (sufficiently large) exponent $n$ no matter whether $n$ is odd or even.
	In the even case, we need to work with a class $\mathfrak H(\delta, \rho)$ whose groups $\Gamma$ can have (some) even torsion.
	See \autoref{rem: even exponents} below.
	In particular such groups could contain infinite dihedral subgroups.
	This feature is incompatible with the requirements stated in \cite{Coulon:2021wg}.
	All the adjustment we made are designed to allow the existence of infinite dihedral subgroups in $\Gamma$.
	\begin{itemize}
		\item Assumption (H2) in \cite[Definition~3.2]{Coulon:2021wg} forces every infinite elementary subgroup of $\Gamma$ to be cyclic.
		This is no longer a requirement here. 
		However, since a loxodromic element cannot normalise a non-trivial finite normal subgroup, infinite elementary subgroups are isomorphic to either $\Z$ or the infinite dihedral group $\dihedral$.
		\item In comparison with Assumption (H1) in \cite[Definition~3.2]{Coulon:2021wg}, we do not require here that $\nu(\Gamma, X) = 1$.
		Indeed, this hypotheses would not be compatible with the existence of infinite dihedral groups in $\Gamma$.
		Nevertheless, our assumption on elliptic subgroups stated in \ref{enu: family axioms -  elem subgroups} is stronger than its equivalent in \cite{Coulon:2021wg}.
		By \autoref{res: nu invariant vs finite cyclic subgroups}, it forces $\nu( \Gamma, X) \leq 2$.
		This uniform control of the $\nu$-invariant will be enough for our purposes.
		\item We do not assume here that the groups $\Gamma$ are CSA (see the definition in \autoref{sec: graph of groups} below), since infinite dihedral groups are not CSA.
		This assumption is used in \cite{Coulon:2021wg} to ensures that any limit group over the class $\mathfrak H(\delta, \rho)$ is also CSA.
		In the present article, we do not conduct an in-depth study of limit groups over $\mathfrak H(\delta, \rho)$.
		On the contrary, for our applications, we know that the limit groups we will deal with are CSA (they are surface groups).
		Thus, this hypotheses is no longer required. \qedhere
	\end{itemize}
\end{rema}

\paragraph{Energies.}

Let $G$ be a group and $U$ a finite subset of $G$.
We write $\card U$ for its cardinality.
We are going to investigate morphisms of the form $\varphi \colon G \to \Gamma$, for some triple $(\Gamma, X, \mathcal C) \in \mathfrak H(\delta, \rho)$.
If we need to emphasise the space that $\Gamma$ is acting on, we write $\varphi\colon G \to (\Gamma, X)$ or $\varphi \colon G \to (\Gamma, X, \mathcal C)$.
Given $x \in X$, the $L^1$- and $L^\infty$-\emph{energies of $\varphi$ at the point $x$}  (\emph{with respect to $U$}) are defined by 
\begin{align*}
	\lambda_1(\varphi,U,x) & = \sum_{u \in U} \dist{\varphi(u)x}x, \\
	\lambda_\infty(\varphi,U,x) & = \sup_{u \in U} \dist{\varphi(u)x}x.
\end{align*}

\begin{defi}[Energy]
\label{def: energy}
	Let $(\Gamma, X, \mathcal C) \in \mathfrak H(\delta, \rho)$.
	Let $p \in \{1, \infty\}$.
	Given a homomorphism $\varphi \colon G \to (\Gamma,X)$, the $L^p$-\emph{energy of $\varphi$} (\emph{with respect to $U$}) is 
	\begin{equation*}
		\lambda_p(\varphi,U) = \inf_{x \in X}\lambda_p(\varphi,U,x) .
	\end{equation*}
\end{defi}

These energies are related as follows:
\begin{equation}
\label{eqn: comparing energies}
	 \lambda_\infty(\varphi,U) \leq \lambda_1(\varphi,U) \leq \card U\lambda_\infty(\varphi,U)\,.
\end{equation}
Later in the article, we will choose a point $o \in X$ minimizing one of these energies and treat it as a basepoint.
Nevertheless, it will be convenient to keep the basepoint in the thick part $X^+$ of the space $X$.
This motivates the following variation of energy.

\begin{defi}[Restricted energy]
	Let $(\Gamma, X, \mathcal C) \in \mathfrak H(\delta, \rho)$.
	Given a homomorphism $\varphi \colon G \to (\Gamma,X)$, the \emph{restricted energy of $\varphi$} (\emph{with respect to $U$}) is defined by
	\begin{equation*}
		\lambda^+_1(\varphi,U) = \inf_{x \in X^+}\lambda_1(\varphi,U,x)\,.
	\end{equation*}
	(Note that $x$ runs over the thick part $X^+$ instead of $X$.)
\end{defi}

\begin{lemm}[Coulon--Sela {\cite[Lemma~3.8]{Coulon:2021wg}}]
\label{res: comparing energies}
	Let $\delta, \rho \in \R_+^*$, with $\rho > 2 \delta$.
	Let $(\Gamma, X, \mathcal C) \in \mathfrak H(\delta, \rho)$. 
	Let $\varphi \colon G \to (\Gamma, X)$ be a homomorphism and $U$ a finite subset of $G$.
	Assume that the image of $\varphi$ is not abelian. Then 
	\begin{equation*}
		\lambda_\infty(\varphi,U) \leq \lambda_1(\varphi,U) \leq \lambda_1^+(\varphi,U) \leq 2 \card U \lambda_\infty(\varphi,U)\,.
	\end{equation*}
	Moreover, for every $\epsilon > 12\delta$, there exist $x \in X$ and $x^+ \in X^+$ such that 
	\begin{enumerate}
		\item $\lambda_\infty(\varphi,U,x) \leq \lambda_\infty(\varphi,U) + \epsilon$ and $\lambda_1(\varphi,U,x^+) \leq \lambda^+_1(\varphi,U) + \epsilon$, and
		\item $\dist x{x^+} \leq \card U \lambda_\infty(\varphi,U) + \epsilon$.
	\end{enumerate}
\end{lemm}

%
\subsection{Approximations of \texorpdfstring{$\Gamma(n)$}{Γ(n)}}
%

Let us now explain how the class of groups $\mathfrak H(\delta, \rho)$ introduced in the previous section naturally appears in the study of the group $\Gamma(n)$.
Let $(\Gamma_1, X_1)$ and $(\Gamma_2, X_2)$ be two pairs, where $X_i$ is a metric space and $\Gamma_i$ a group acting by isometries on $X_i$.
A \emph{morphism} $(\Gamma_1, X_1) \to (\Gamma_2, X_2)$ is a pair $(\pi, f)$ where $\pi$ is a homomorphism from $\Gamma_1$ to $\Gamma_2$ and $f$ a $\pi$-equivariant map from $X_1$ to $X_2$ (we do not require $f$ to be an isometry).
Such a morphism is called an \emph{epimorphism} if $\pi$ is onto (we do not require $f$ to be onto).
The statement below yields a sequence of groups approximating $\Gamma(n)$ and summarises all their properties.

\begin{theo}
\label{res: approximating sequence}
	There exist $\delta \in \R_+^*$ and a non-decreasing function $\rho \colon \N \to \R_+^*$, diverging to infinity, with the following properties.
	Let $\Gamma$ be the fundamental group of a closed orientable surface acting on the hyperbolic plane $X = \H^2$ by deck transformations.
	There exists an exponent $N \in \N$ such that, for every integer $n \geq N$, the group $\Gamma(n)$ is the direct limit of a sequence of non-elementary hyperbolic groups.
	
	\smallskip
	More precisely, there is a sequence
	\begin{equation*}
		(\Gamma_0, X_0) \onto (\Gamma_1, X_1) \onto \dots \onto (\Gamma_j, X_j) \onto (\Gamma_{j+1}, X_{j+1}) \onto \dots
	\end{equation*}
	with the following properties.
	\begin{enumerate}
		\item \label{enu: approximating sequence - init}
		$\Gamma_0 = \Gamma$ while $X_0$ is a rescaled version of $X$ that does not depend on $n$.
		\item \label{enu: approximating sequence - control}
		For every $j \in \N$, there is a subset $\mathcal C_j \subset X_j$ such that $(\Gamma_j, X_j, \mathcal C_j)$ belongs to the class $\mathfrak H(\delta, \rho(n))$.
		\item \label{enu: approximating sequence - cvg}
		The direct limit of the sequence $(\Gamma_j)$ is isomorphic to $\Gamma(n)$.
		\item \label{enu: approximating sequence - elliptic}
		Every elliptic subgroup of $\Gamma_j$ is contained in a malnormal cyclic subgroup of order $n$ generated by an element $u \in \Gamma_j$ whose image in $\Gamma(n)$ coincides with that of a simple closed curve.
		\item \label{enu: approximating sequence - metric}
		Let $j \in\N$.
		On the one hand, the map $X_j \to X_{j+1}$ is $\epsilon(n)$-Lipschitz, where $\epsilon(n) <1$ only depends on $n$, and converges to zero as $n$ tends to infinity.
		On the other hand, for every $x \in X_j$, the projection $\pi_j \colon \Gamma_j \onto \Gamma_{j+1}$ is one-to-one when restricted to the set
		\begin{equation*}
			\set{\gamma \in \Gamma_j}{\epsilon(n)\dist{\gamma x}x \leq \frac{\rho(n)}{100}}.
		\end{equation*}
		\item \label{enu: approximating sequence - lifting}
		Let $j \in \N\setminus\{0\}$.
		Let $F$ be the free group generated by a finite set $S$.
		Let $\ell \in \N$ and $V$ be the set of elements of $F$ whose length (seen as words over $S$) is at most $\ell$.
		Let $\varphi \colon F \to \Gamma_j$ be a homomorphism whose image does not fix a cone point in $\mathcal C_j$.
		Assume that  
		\begin{equation*}
			\lambda_\infty\left(\varphi, S\right) < \frac {\rho(n)}{100\ell}\,.
		\end{equation*}
		Then there exists a map $\tilde \varphi \colon F \to \Gamma_{j-1}$ such that $\varphi = \pi_{j-1} \circ \tilde \varphi$ and $V \cap \ker {\varphi} = V \cap \ker {\tilde \varphi}$.

	\end{enumerate}
\end{theo}

\begin{defi}
    We call any sequence of triples $(\Gamma_j, X_j, \mathcal C_j)$ satisfying the conclusion of \autoref{res: approximating sequence} an \emph{approximation sequence of $\Gamma(n)$}.
\end{defi}

\begin{rema}[About even exponents]
\label{rem: even exponents}
	Recall that if $(\Gamma_j, X_j, \mathcal C_j)$ belongs to $\mathfrak H(\delta, \rho)$ then any finite subgroup of $\Gamma_j$ is cyclic.
	Denote by $\Gamma^n$ the (normal) subgroup of $\Gamma$, generated by the $n$th power of every element $\gamma \in \Gamma$.
	A reader familiar with the study of periodic quotients $\Gamma / \Gamma^n$ for even exponent may be surprised that such groups could be used to approximate $\Gamma(n)$.
	Indeed, if $n$ is a (large) power of $2$, the quotient $\Gamma / \Gamma^n$ contains finite subgroups of the form $\dihedral[n] \times \dots \times \dihedral[n]$, with an arbitrarily large number of factors.
	On the contrary, every maximal finite subgroup of $\Gamma(n)$ is cyclic of order $n$ and generated by (the image of) a simple closed curve.
	See the induction assumption \ref{enu: induction step - elliptic} below.
	This important difference explains why we can ``easily'' handle both the case of even and odd exponents.
\end{rema}

\begin{rema}\
\label{rem: canonical proj asymp injective}
	With an induction, \autoref{res: approximating sequence}\ref{enu: approximating sequence - metric} implies that the sequence of finitely generated groups  $\Gamma(n)$ converges to the surface group $\Gamma$ in the space of marked groups. That is to say, every finite subset $F \subset \Gamma$ embeds into $\Gamma(n)$ under the natural quotient map, for all sufficiently large $n$. In this paper, we will not need to use any other property of the topology of marked groups except this definition. See Champetier--Guirardel \cite{Champetier:2005ic} for generalities on the space of marked groups.
\end{rema}

%
\subsubsection{The initial data}
%

The remainder of this section is devoted to the proof of \autoref{res: approximating sequence}.
To that end, we first fix first some auxiliary quantities.
Let $\delta_0$, $\delta_1$, $\Delta_0$, and $\rho_0$ be the parameters given by the small-cancellation theorem (\autoref{res: small cancellation theorem}).
Set
\begin{equation*}
	A = 10\pi \sinh\left(10^3\delta_1\right)\,, \quad
	\kappa = \frac{\delta_1 }{\pi \sinh (26\delta_1)}\,, \quad 
	\alpha = 20 \pi \sinh(100\delta_1)\,,
\end{equation*}
and \begin{equation*}
	\delta = \max \{ 10\delta_1, A, \alpha\}.
\end{equation*}
We choose once for all a non-decreasing map $\rho \colon \N \to \R_+$ which diverges to infinity, and satisfies
\begin{equation*}
	\rho(n) =o\left(\ln n\right)\,.
\end{equation*}
We define a rescaling parameter by
\begin{equation*}
	\epsilon(n) = \sqrt{\frac {10\pi \sinh \rho(n)}{n \kappa \delta_1}}.
\end{equation*}
Observe that 
\begin{equation*}
	\ln \epsilon(n) = - \frac 12 \ln n + \frac 12\rho(n) + O(1).
\end{equation*}
Thus, $\epsilon(n)$ and $\epsilon(n)\rho(n)$ converge to zero.
Therefore, there exists a critical exponent $N\geq 4$ such that, for every integer $n \geq N$, the following all hold:
\begin{align}
	\label{eqn: recale - rho0}
	\rho(n)  & \geq \rho_0\,; \\
	\label{eqn: recale - hyp}
	\epsilon(n) \delta_1 & \leq \delta_0\,; \\
	\label{eqn: recale - A}
	\epsilon(n) \left[A + 500\delta_1\right] & \leq \Delta_0\,; \\
	\label{eqn: recale - A bis}
	\epsilon(n)   & \leq 2 / 5\,; \\
	\label{eqn: recale - inj}  
	\epsilon(n) \kappa & \leq 1\,;\\
	\label{eqn: recale - thin}  
	\epsilon(n) (\rho(n) + \alpha) & \leq \alpha/2\,.
\end{align}

Let $\Gamma$ be the fundamental group of a closed orientable surface $S$ of genus at least $2$.
We choose a geometric structure on $S$, which allows us to identify the universal cover $X$ of $S$ with the hyperbolic plane $\mathbf H^2$.
The group $\Gamma$ acts on $X$ by deck transformations.
Up to rescaling $X$, we can assume that $X$ is $\delta$-hyperbolic for some $\delta \leq \delta_1$ and $A(\Gamma, X) \leq A$.
Similarly we can assume that every non-trivial isometry of $\Gamma$ is $\alpha$-thin.
We can also increase the value of the critical exponent $N$ given by (\ref{eqn: recale - hyp})-(\ref{eqn: recale - inj}) so that
\begin{equation*}
    \inj \Gamma X\geq \epsilon(n) \kappa \delta_1
\end{equation*}
for every $n \geq N$.
Finally, we fix once and for all an exponent $n \geq N$.
Recall that $\Gamma(n)$ stands for the quotient of $\Gamma$ by the normal subgroup generated by the $n$th power of every simple closed curve.
Denote by $U$ the set of all elements $u \in \Gamma$ whose image in $\Gamma(n)$ coincides with that of a simple closed curve on $S$.
Observe that the group $\group{u^n : u \in U}$ is the kernel of the projection $\Gamma \onto \Gamma(n)$.

\subsubsection{Construction by induction}

We are going to build by induction a sequence of triples $(\Gamma_j, X_j, \mathcal C_j)$, where:
\begin{itemize}
	\item $\Gamma_j$ is a quotient of $\Gamma$, and the projection $\Gamma \onto \Gamma(n)$ factors through $\Gamma \onto \Gamma_j$;
	\item $X_j$ is a $\delta_1$-hyperbolic length space endowed with an action by isometries of $\Gamma_j$; and
	\item $\mathcal C_j \subset X_j$ is a $\Gamma_j$-invariant, $2\rho(n)$-separated set of cone points.
\end{itemize}
Denoting by $U_j$ the image of $U$ in $\Gamma_j$, these data will satisfy the following properties.
\begin{enumerate}[label=(R\arabic*)]
	\item \label{enu: induction step - action}
	The action of $\Gamma_j$ is proper, co-compact and non-elementary. In particular, $\Gamma_j$ is hyperbolic.
	\item \label{enu: induction step - geometry}
 $A(\Gamma_j, X_j) \leq A$ and $\inj {U_j}{X_j} \geq \epsilon(n) \kappa \delta_1$.
	\item \label{enu: induction step - elliptic}
	Every elliptic subgroup of $\Gamma_j$ is contained in a malnormal cyclic subgroup of order $n$ generated by an element of $U_j$.
	\item \label{enu: induction step - thin}
	Every loxodromic element of $\Gamma_j$ is $\alpha$-thin.
	Every non-trivial elliptic element is either $\alpha$-thin or fixes a cone point in $\mathcal C_j$ and is $(\rho(n)+ \alpha)$-thin.
\end{enumerate}

\paragraph{The basis of the induction.}
We set $\Gamma_0 = \Gamma$, $X_0 = X$ and $\mathcal C_0 = \emptyset$.
We claim that $(\Gamma_0, X_0, \mathcal C_0)$ satisfies the above induction hypotheses.
Indeed \ref{enu: induction step - action} is straightforward, while \ref{enu: induction step - geometry} and \ref{enu: induction step - thin} follow from our choice of renormalization and the critical exponent $N$.
Since $\Gamma$ is torsion-free, \ref{enu: induction step - elliptic} holds.

\paragraph{The induction step.}
Let $j \in \N$.
Assume now that the triple $(\Gamma_j, X_j, \mathcal C_j)$ has already been defined.
For simplicity, we simply write $\epsilon$ and $\rho$ for $\epsilon(n)$ and $\rho(n)$ respectively.
We denote by $P_j$ the set of all loxodromic elements $u \in U_j$, which are not a proper power, and such that $\norm u \leq 10\delta_1$.
Let $K_j$ be the (normal) subgroup of $\Gamma_j$ generated by $\set{u^n}{u \in P_j}$ and $\Gamma_{j+1}$ the quotient of $\Gamma_j$ by $K_j$.

If $P_j$ is empty, then $\Gamma_{j+1} = \Gamma_j$.
In this situation, we let $X_{j+1} = \epsilon X_j$ and $\mathcal C_{j+1} = \mathcal C_j$.
They satisfy the conclusions \ref{enu: induction step - action}-\ref{enu: induction step - thin}.

Otherwise, we define the family $\mathcal Q_j$ by
\begin{equation*}
	\mathcal Q_j = \set{\left( \group {u^n}, Y_u\right)}{u \in P_j}.
\end{equation*}
We are going to prove that $\Gamma_{j+1}$ is a small cancellation quotient of $\Gamma_j$ by the family $\mathcal Q_j$.
To that end, we consider the action of $\Gamma_j$ on the rescaled space $\epsilon X_j$.
According to (\ref{eqn: recale - hyp}) this space is $\delta_0$-hyperbolic.
Since $\Gamma_j$ is hyperbolic, there are only finitely many conjugacy classes whose length is bounded from above by a given length.
Hence $ \Gamma_j\backslash \mathcal Q_j$ is finite.
Recall that the morphism $\Gamma \onto \Gamma(n)$ factors through $\Gamma \onto \Gamma_j$.
It follows that $U_j$ can also be described as the set of all elements in $\Gamma_j$ whose image in $\Gamma(n)$ coincides with that of a simple closed curve in $\Gamma$.

\begin{clai}
\label{res: order elliptic U}
	Every elliptic element in $U_j$ has order exactly $n$.
\end{clai}

\begin{proof}
	Let $u \in U_j$ be an elliptic element.
	According to \autoref{res: order scc}, the image of $u$ in $\Gamma(n)$ has order exactly $n$.
	Since $u$ is elliptic it follows from \ref{enu: induction step - elliptic} that the order of $u$ in $\Gamma_j$ is at most $n$.
	Thus it is exactly $n$.
\end{proof}

\begin{clai}
\label{res: even power dihedral}
	If $n \geq 4$ is even then $U_j \cap \group{u^{n/2}: u \in U_j} = \emptyset$.
\end{clai}

\begin{proof}
	Consider the projections
	\begin{equation*}
		\Gamma \onto \Gamma_j \onto \Gamma(n) \onto \Gamma(n/2).
	\end{equation*}
	In particular, the kernel of the epimorphism $\Gamma_j \onto \Gamma(n/2)$ contains the normal subgroup $\group{u^{n/2}: u \in U_j}$.
	The image in $\Gamma(n/2)$ of any simple closed curve $\gamma \in \Gamma$ has order exactly $n/2$ (\autoref{res: order scc}). 
	In particular, it is non trivial, whence the result.
\end{proof}

\begin{clai}
\label{res: elem closure relation}
	For every $u \in P_j$, we have $\stab{Y_u} = E(u) = \group u$.
\end{clai}

\begin{rema}
\label{rem: dihedral not involved in relations}
	This claim is one of the key observations that allows us to handle the case of even exponents.
	Indeed, even if the group $\Gamma_j$ contains a subgroup, say $H$, isomorphic to $\dihedral$, this subgroup will remain more or less ``untouched'' during the rest of this induction:
	one will never add new relations of the form $\gamma^n = 1$ for some $\gamma \in H$.
	In particular, this prevents the appearance of complicated finite subgroups in $\Gamma(n)$.
\end{rema}

\begin{proof}
	It follows from the definition of $Y_u$ that $\stab {Y_u} = E(u)$.
	Since no non-trivial finite subgroup of $\Gamma_j$ is normalised by a loxodromic element, $E(u)$ is isomorphic either to $\Z$ or $\dihedral$.
	We are going to rule out the second option.
	If $E(u) = \dihedral$, then it is generated by two involutions, say $\gamma_1$ and $\gamma_2$.
	Combining \ref{enu: induction step - elliptic} with \autoref{res: order elliptic U}, the exponent $n$ is even.
	Moreover, $\gamma_i$ can be written as
	\begin{equation*}
		\gamma_i = u_i^{n/2}
	\end{equation*}
	for some $u_i \in U_j$.
	In particular $u$ belongs to 
	\begin{equation*}
		E(u) = \group{u_1^{n/2}, u_2^{n/2}}.
	\end{equation*}
	This contradicts \autoref{res: even power dihedral}.
	Hence $E(u)$ is cyclic.
	Since $u$ belongs to $P_j$, it is not a proper power, hence $E(u) = \group u$.
\end{proof}

\begin{clai}
\label{res: induction - check sc hyp}
	The family $\mathcal Q_j$ satisfies the small-cancellation hypotheses, i.e.\ $\Delta(\mathcal Q_j, \epsilon X_j) \leq \Delta_0$ and $T(\mathcal Q_j, \epsilon X_j) \geq 10 \pi \sinh \rho$.
\end{clai}

\begin{proof}
	We start with the upper bound of $\Delta(\mathcal Q_j, \epsilon X_j)$.
	Let $u_1$ and $u_2$ be two elements of $P_j$ such that $(\group {u_1^n}, Y_{u_1})$ and $(\group {u_2^n}, Y_{u_2})$ are distinct.
	We first claim that $u_1$ and $u_2$ generate a non-elementary subgroup.
	Assume on the contrary that it is not the case.
	According to \autoref{res: elem closure relation} we have $\group {u_1} = \group{u_2}$.
	Therefore 
	\begin{equation*}
		(\group {u_1^n}, Y_{u_1}) = (\group {u_2^n}, Y_{u_2})\,,
	\end{equation*}
	which contradicts our assumption.
	Recall that $Y_{u_i}$ is contained in the $52\delta$-neighbourhood of $\fix{u_i, \norm{u_i} + 8 \delta} \subset \epsilon X_j$ \cite[Lemma~2.32]{Coulon:2014fr}.
	Since $u_1$ and $u_2$ generate a non-elementary subgroup we get from the definition of the acylindricity parameter that
	\begin{equation*}
		\diam{Y_{u_1}^{+5\delta} \cap Y_{u_2}^{+5\delta}} 
		\leq 18 \left[\nu(\Gamma_j,X_j) + 3\right]\epsilon\delta_1 + A(\Gamma_j,\epsilon X_j)\,.
	\end{equation*}
	Since every elliptic subgroup of $\Gamma_j$ is cyclic, $\nu(\Gamma_j, X_j) \leq 2$ (\autoref{res: nu invariant vs finite cyclic subgroups}).
	By assumption $A(\Gamma_j,\epsilon X_j) \leq \epsilon A$.
	Hence
	\begin{equation*}
		\diam{Y_{u_1}^{+5\delta} \cap Y_{u_2}^{+5\delta}} 
		\leq \epsilon \left[A+ 90\delta_1\right].
	\end{equation*}
	Using (\ref{eqn: recale - A}) we get $\Delta(\mathcal Q_j, \epsilon X_j) \leq \Delta_0$.
	It follows from our assumption that 
	\begin{equation*}
		\inj {U_j}{\epsilon X_j}
		\geq \epsilon \inj {U_j}{X_j}
		\geq \epsilon^2 \kappa \delta_1 
		\geq \frac{10\pi \sinh \rho}n.
	\end{equation*}
	Let $(H,Y) \in \mathcal Q_j$.
	By construction, $H$ is generated by the $n$th power of a loxodromic element of $U_j$.
	Consequently, for every $h \in H$,
	\begin{equation*} 
		\norm[\epsilon X_j]h
		\geq n \inj {U_j}{\epsilon X_j}
		\geq 10\pi \sinh \rho.
	\end{equation*}
	It follows that $T(\mathcal Q_j,\epsilon X_j)\geq 10\pi \sinh \rho$.
\end{proof}

On account of the previous claim, we can now apply the small-cancellation theorem (\autoref{res: small cancellation theorem}) to the action of $\Gamma_j$ on the rescaled space $\epsilon  X_j$ and the family $\mathcal Q_j$.
We denote by $\dot X_j$ the space obtained by attaching on $\epsilon X_j$, for every $(H,Y) \in \mathcal Q_j$, a cone of radius $\rho$ over the set $Y$.
The space $X_{j+1}$ is the quotient of $\dot X_j$ by $K_j$.
According to \autoref{res: small cancellation theorem}, $X_{j+1}$ is a $\delta_1$-hyperbolic geodesic space and $\Gamma_{j+1}$ acts properly, co-compactly by isometries on it.
We write $\mathcal C_{j+1}$ for the image in $X_{j+1}$ of the set of cone points in $\dot X_j$. 
Note that $\mathcal C_{j+1}$ is not the image in $X_{j+1}$ of $\mathcal C_i \subset X_j$.
In addition, we write $U_{j+1}$ for the image of $U_j$ in $\Gamma_{j+1}$.

Observe that $\Gamma_{j+1}$ has been obtained from $\Gamma_j$ by adjoining relations of the form $u^n = 1$ where the image of $u$ in $\Gamma(n)$ coincides with that of a simple closed curve in $\Gamma$.
Consequently the morphism $\Gamma_j \onto \Gamma(n)$ factors through the projection $\Gamma_j \onto \Gamma_{j+1}$.
In addition, the pre-image of $U_{j+1}$ in $\Gamma_j$ is actually $U_j$.

We now prove that the triple $(\Gamma_{j+1},  X_{j+1}, \mathcal C_{j+1})$ satisfies the induction hypotheses.
We already mentioned that $\Gamma_{j+1}$ acts properly co-compactly on $X_{j+1}$, which corresponds to \ref{enu: induction step - action}.
Point~\ref{enu: induction step - geometry} is a consequence of the following statement.

\begin{clai}
	The parameters $A(\Gamma_{j+1},X_{j+1})$ and $\inj{U_{j+1}}{X_{j+1}}$ satisfy
	\begin{equation*}
		A(\Gamma_{j+1},X_{j+1})  \leq A
		\quad \text{and} \quad
		\inj{U_{j+1}}{X_{j+1}}\geq \epsilon \kappa \delta_1\,.
	\end{equation*}
\end{clai}

\begin{proof}
	We start with the upper bound of $A(\Gamma_{j+1}, X_{j+1})$.
	Recall that $\nu(\Gamma_j, X_j)$ is at most $2$.
	Hence \autoref{res: acyl quotient} yields
	\begin{equation*}
		A(\Gamma_{j+1}, X_{j+1}) 
		\leq A(\Gamma_j, \epsilon X_j) + [\nu(\Gamma_j,X_j) +4]\pi \sinh(10^3 \delta_1)
		\leq (3/5+ \epsilon)A\,.
	\end{equation*}
	Using (\ref{eqn: recale - A bis}) we obtain $A(\Gamma_{j+1}, X_{j+1}) \leq A$.
	
	We now focus on the injectivity radius of $U_{j+1}$.
	Let $u \in U_j$ not belong to $\stab Y$ for some $(H,Y) \in \mathcal Q_j$.
    Observe that 
	\begin{equation*}
		\snorm[\epsilon X_j] {u} 
		\geq \norm [\epsilon X_j]{u} - 8\epsilon\delta_1
		> 2\epsilon\delta_1\,.
	\end{equation*}
    The first inequality is indeed a consequence of (\ref{eqn: regular vs stable length}) while the second one follows from the definition of the set $P_j$ used to build $\mathcal Q_j$: any element that does not belong to $P_j$ has a translation length larger that $10\delta_1$ (for the metric of $X$ before rescaling).
	As we noticed, the pre-image of $U_{j+1}$ in $\Gamma_j$ is exactly $U_j$.
	Hence, \autoref{res: injectivity radius} combined with (\ref{eqn: recale - inj}) yields 
	\begin{equation*}
		\inj{U_{j+1}}{X_{j+1}}	
		\geq \min\left\{ \epsilon \kappa\delta_1, \delta_1 \right\}
		\geq  \epsilon \kappa\delta_1\,. \qedhere
	\end{equation*}
\end{proof}

Let us prove \ref{enu: induction step - elliptic} for $\Gamma_{j+1}$.
Note first that it follows from \autoref{res: max ell sg malnormal} and  \ref{enu: induction step - elliptic}  for the group $\Gamma_j$, that any elliptic subgroup of $\Gamma_j$ that is normalised by a loxodromic element is trivial.
Consequently the same holds for $\Gamma_{j+1}$, by \autoref{res: lifting normalised elliptic sbgp}.
Let $F$ be an elliptic subgroup of $\Gamma_{j+1}$.
According to \autoref{res: lifting elliptic subgroups}, either $F$ is the isomorphic image of an elliptic subgroup of $\Gamma_j$, or $F$ embeds in $\stab{Y_u} / \group{u^n} = \Z/n\Z$ for some $u \in P_j$.
In view of Assumption~\ref{enu: induction step - elliptic} for $\Gamma_j$, we deduces that $F$ is contained in a cyclic subgroup of order $n$ generated by some element $u \in U_{j+1}$.
It then follows from \autoref{res: max ell sg malnormal} applied in $\Gamma_{j+1}$ that $\group u$ is malnormal.

Point~\ref{enu: induction step - thin} is a consequence of the following statement.

\begin{clai}
	Let $\gamma \in \Gamma_{j+1}\setminus\{1\}$.
	\begin{itemize}
	\item If $\gamma$ fixes a cone point $c\in\mathcal C_{j+1}$, then $\gamma$ is $(\rho(n)+\alpha)$-thin at $c$.
	\item Otherwise $\gamma$ is $\alpha$-thin.
\end{itemize}
\end{clai}

\begin{proof}
	Let 
	\begin{equation*}
		\beta = \epsilon(\rho + \alpha)
		\quad\text{and}\quad
		\bar \beta = \beta + 10\pi \sinh(100\delta_1) = \beta + \alpha/2.
	\end{equation*}
	It follows from \ref{enu: induction step - thin} that every non-trivial element in $\Gamma_j$ is $\beta$-thin for its action on the rescaled space $\epsilon X_j$.
	According to \autoref{res: sc - thin}, if $\gamma$ fixes a cone point $c\in \mathcal C_{j+1}$, then $\gamma$ is $(\rho +\bar\beta)$-thin at $c$.
	It is $\bar\beta$-thin otherwise.
	However, $\epsilon(\rho + \alpha) \leq \alpha/2$ by (\ref{eqn: recale - thin}), whence the result.
\end{proof}

We have proven that the triple $(\Gamma_{j+1}, X_{j+1}, \mathcal C_{j+1})$ satisfies Properties \ref{enu: induction step - action}-\ref{enu: induction step - thin}, which completes the induction process.

\paragraph{Additional properties.}
Recall that $\delta$ has been fixed so that 
\begin{equation*}
	\delta \geq \max\{10 \delta_1, A, \alpha\}.
\end{equation*}
Consequently, for every $j \in \N$, the triple $(\Gamma_j, X_j, \mathcal C_j)$ satisfies all the axioms defining $\mathfrak H(\delta, \rho)$, except maybe \ref{enu: family axioms - conical - 2} describing the action of $\Gamma_j$ around a cone point.
The latter is just a reformulation of \autoref{res: sc - angle cocycle}, hence $(\Gamma_j, X_j, \mathcal C_j)$ belongs to $\mathfrak H(\delta, \rho)$.
This proves \autoref{res: approximating sequence}\ref{enu: approximating sequence - control}.
Note that $\Gamma_0 = \Gamma$ and $X_0$ is just a rescaled version of $X$, hence \autoref{res: approximating sequence}\ref{enu: approximating sequence - init} holds.
Observe also that \autoref{res: approximating sequence}\ref{enu: approximating sequence - elliptic} is simply a reformulation of \ref{enu: induction step - elliptic}.

We continue the above discussion with some additional information on the projection $\pi_j \colon \Gamma_j \to \Gamma_{j+1}$.
More precisely, the next two claims respectively prove that the sequence $(\Gamma_j, X_j, \mathcal C_j)$ satisfies items~\ref{enu: approximating sequence - metric} and \ref{enu: approximating sequence - lifting} of \autoref{res: approximating sequence}.

\begin{clai}
\label{res: claim proj}
	The map $\zeta \colon X_j \to X_{j+1}$ is $\epsilon$-Lipschitz.
	Moreover, for every $x \in X_j$, the projection $\pi_j \colon \Gamma_j \onto \Gamma_{j+1}$ is one-to-one when restricted to the set
	\begin{equation*}
		\set{\gamma \in \Gamma_j}{\epsilon\dist{\gamma x}x \leq \frac{\rho(n)}{100}}.
	\end{equation*}
\end{clai}

\begin{proof}
	Recall that the map $\epsilon X_j \to X_{j+1}$ is $1$-Lipschitz.
	Hence $X_j \to X_{j+1}$ is $\epsilon$-Lipschitz (and $\pi_j$-equivariant by construction).
	The point $x$ belongs to $X_j$.
	In particular it is at distance at least $\rho$ from any cone point of the cone-off space $\dot X_j$ built on $\epsilon X_j$.
	The claim is now a consequence of \autoref{res: small cancellation theorem}\ref{enu: small cancellation theorem - inj kernel}. 
\end{proof}

\begin{clai}
\label{res: claim lifting}
	Let $F$ be the free group generated by a finite set $S$.
	Let $\ell \in \N$ and $V$ be the set of elements of $F$ whose length (seen as words over $S$) is at most $\ell$.
	Let $\varphi \colon F \to \Gamma_{j+1}$ be a homomorphism whose image does not fix a cone point in $\mathcal C_{j+1}$.
	Assume that  
	\begin{equation*}
		\lambda_\infty\left(\varphi, S\right) < \frac {\rho(n)}{100\ell}\,.
	\end{equation*}
	Then there exists a map $\tilde \varphi \colon F \to \Gamma_j$ such that $\varphi = \pi_j \circ \tilde \varphi$ and $V \cap \ker {\varphi} = V \cap \ker {\tilde \varphi}$.
\end{clai}

\begin{proof}
	This is a direct application of \autoref{res: sc - lifting morphism}.
\end{proof}


\subsubsection{The group $\Gamma_\infty$}

Denote now by $\Gamma_\infty$ the direct limit of the sequence
\begin{equation*}
    \Gamma = \Gamma_0 \onto \Gamma_1 \onto \cdots \onto \Gamma_i \onto \Gamma_{i+1} \onto \cdots
\end{equation*}
Let $U_\infty$ be the image of $U$ in $\Gamma_\infty$.
Given an element $\gamma \in \Gamma$, we often abuse notation and continue to write $\gamma$ for its image in $\Gamma_i$, $\Gamma_\infty$, or $\Gamma(n)$.

\begin{lemm}
\label{res: isom Gamma infty / Gamma(n)}
    The groups $\Gamma_\infty$ and $\Gamma(n)$ are isomorphic.
    In particular, $U_\infty$ is the image in $\Gamma_\infty$ of the set of all simple closed curves.
\end{lemm}

\begin{proof}
    It follows from the construction that, for every $i \in \N$, the set $U_i$ is the image of $U$ in $\Gamma_i$.
    Consequently, all the relations added to build $\Gamma_\infty$ from $\Gamma$ are $n$th power of elements in $U$.
    Therefore, the canonical projection $\Gamma \onto \Gamma(n)$ factors through $\Gamma \onto \Gamma_\infty$.
    Let us now prove that these two projections have the same kernel.
    To that end it suffices to prove that the $n$th power of every element in $U$ is trivial in $\Gamma_\infty$.
    Let $\gamma \in U$.
    Recall that, for every $i \in \N$, the map $X_i \to X_{i+1}$ is $\epsilon$-Lipschitz, where $\epsilon < 1$.
    Consequently,
    \begin{equation*}
        \norm[X_i] \gamma \leq \epsilon^i \norm[X]\gamma.
    \end{equation*}
    Therefore, there is $i \in \N$ such that $\norm[X_i] \gamma \leq 10\delta_1$. 
    If $\gamma$ is elliptic in $\Gamma_i$ then, according to \ref{enu: induction step - elliptic}, we have $\gamma^n = 1$ in $\Gamma_i$.
    If $\gamma$ is loxodromic in $\Gamma_i$, then it follows from the construction that $\gamma^n$ belongs to $K_i$, hence $\gamma^n = 1$ in $\Gamma_{i+1}$.
    In both cases $\gamma^n = 1$ in $\Gamma_\infty$, whence the result.
\end{proof}

The previous lemma corresponds to \autoref{res: approximating sequence}\ref{enu: approximating sequence - cvg} and completes the proof of \autoref{res: approximating sequence}.

%
\subsection{First properties of \texorpdfstring{$\Gamma(n)$}{Γ(n)}}
\label{subsec: Gamma(n) prop}
%

In this subsection, we deduce some properties of $\Gamma(n)$ from \autoref{res: approximating sequence}, including the results of \autoref{thm: 1st Structural results}. Throughout, the exponent $N$ is the one given by \autoref{res: approximating sequence}, and we fix an integer $n \geq N$.
Let $(\Gamma_j)$ denote the sequence of hyperbolic groups approximating $\Gamma(n)$ that \autoref{res: approximating sequence} provides.

We first analyse the finite subgroups of $\Gamma(n)$.

\begin{prop}
\label{res: finite subgroup}
 	Every finite subgroup of $\Gamma(n)$ is contained in a malnormal cyclic subgroup of order $n$, which is the image of a simple closed curve in $\Gamma$.
\end{prop}
\begin{proof}
	By construction, every finite subgroup $E$ of $\Gamma(n)$ lifts to a finite subgroup $E_i$ of $\Gamma_i$, for some $i \in \N$.
	According to \autoref{res: approximating sequence}\ref{enu: approximating sequence - elliptic}, $E_i$ is contained in $\group {u_i}$ for some $u_i \in \Gamma_i$ whose image $u$ in $\Gamma(n)$ coincides with that of a simple closed curve.
	Hence $E \subset \group{u}$.
	Observe that according to \autoref{res: order scc}, the subgroup $\group u$ is isomorphic to $\Z / n \Z$.
	If $E$ were not malnormal, the same would hold for the image of $E_i$ in $\Gamma_j$, for some sufficiently large $j \geq i$.
	This would contradict \autoref{res: approximating sequence}\ref{enu: approximating sequence - elliptic}, which completes the proof.
\end{proof}

This enables us to control the centre of $\Gamma(n)$, which will play a crucial role in the proofs of our main theorems. 

\begin{prop}
\label{res: infinite + trivial finite radical}
 	The group $\Gamma(n)$ is infinite with trivial centre and no non-trivial finite normal subgroup.
\end{prop}

\begin{proof}
	Suppose first that $\Gamma(n)$ is finite.
	It follows that $\Gamma(n)$ is finitely presented, hence $(\Gamma_j)$ eventually stabilises.
	Consequently $\Gamma(n)$ is isomorphic to $\Gamma_j$ for some sufficiently large $j \in \N$, hence is hyperbolic and non-elementary.
	This contradicts the fact that $\Gamma(n)$ is finite and proves the first part of the statement.
	
	Combined with \autoref{res: finite subgroup}, we get that $\Gamma(n)$ has no non-trivial finite normal subgroup.
	Let $\gamma$ be an element in the centre of $\Gamma(n)$.
	There are $j \in \N$ and a pre-image $\gamma_j \in \Gamma_j$ of $\gamma$ that is central (in $\Gamma_j$).
	Since $\Gamma_j$ is non-elementary hyperbolic, $\gamma_j$ has finite order.
	It follows that $\group \gamma$ is a finite normal subgroup of $\Gamma(n)$.
	Thus it is trivial.
\end{proof}

We get a form of the Tits alternative for finitely presented subgroups, for all sufficiently large $n$. Note that \autoref{thm: 2nd Structural results}\ref{enu: Structure of Gamma(n) -- fg Tits alternative} upgrades this to a Tits alternative for all finitely \emph{generated} subgroups, for all large \emph{multiples} $n$.

\begin{prop}
\label{res: elementary subgroup}
	Let $H$ be a finitely presented subgroup of $\Gamma(n)$ that does not contain a non-abelian free subgroup. 
	Then $H$ is isomorphic to a subgroup of $\Z / n \Z$, $\Z$, or $\dihedral$.
	Moreover, the latter case only arises if $n$ is even.
\end{prop}
\begin{proof}
	We already observed that every finite subgroup of $\Gamma(n)$ is contained in a cyclic subgroup of order $n$ (\autoref{res: finite subgroup}).
	Thus, without loss of generality, we can assume that $H$ is infinite.
	Since $H$ is finitely presented, there is $j \in \N$ and a subgroup $H_j \subset \Gamma_j$ such that the natural map $\Gamma_j \onto \Gamma(n)$ induces an isomorphism from $H_j$ onto $H$.
	In particular, $H_j$ does not contain a non-abelian free subgroup, hence is elementary and infinite.
	According to \autoref{res: approximating sequence}\ref{enu: approximating sequence - elliptic}, $H_j$ does not normalise a non-trivial finite subgroup.
	Consequently $H_j$, and hence $H$, is isomorphic to either $\Z$ or $\dihedral$.
	Note that, if $n$ is odd, then $\Gamma(n)$ has no even torsion by \autoref{res: finite subgroup}, hence $\Gamma(n)$ does not contain any subgroup isomorphic to $\dihedral$.
\end{proof}

Combining the metabelian quotients from \autoref{sec: Metabelian} with the understanding of the torsion coming from \autoref{res: finite subgroup}, we find that $\Gamma(n)$ is  virtually torsion-free.

\begin{prop}
\label{res: v torsion-free}
 	The group $\Gamma(n)$ is virtually torsion-free.
\end{prop}

\begin{proof}
	By construction, the projection $\Gamma \onto M(n)$ factors through $\Gamma \onto \Gamma(n)$.
	We write $K$ for the kernel of the resulting morphism $\Gamma(n) \onto M(n)$.
	As a periodic solvable group, $M(n)$ is finite, hence $K$ has finite index in $\Gamma(n)$.
	We claim that $K$ is torsion free.
	Let $\gamma \in K$ be a finite order element.
	By \autoref{res: finite subgroup}, $\gamma$ is of the form $ \gamma = u^k$, where $k \in \Z$ and $u \in \Gamma(n)$ is the image of a simple closed curve. 
	In particular, $u^k = 1$ in $M(n)$.
	According to \autoref{res: metabelian quotient} the image $u$ in $M(n)$ has order exactly $n$.
	Therefore $k$ is a multiple of $n$, hence $\gamma$ is trivial in $\Gamma(n)$.
\end{proof}

\section{A Birman-type theorem}
\label{sec:diagram}

In this section we prove \autoref{thm: Birman for quotients}, the analogue of the Birman exact sequence in our setting. 

\subsection{A commutative diagram}

Our first goal in this section is to establish the fundamental commutative diagram of \autoref{fig: Fundamental commutative diagram}.

\begin{prop}
\label{prop:big diagram}
Let $S$ be a closed, connected, orientable, hyperbolic surface. 
For all integers $n$, there exists a commutative diagram as in \autoref{fig: Fundamental commutative diagram}. The top row is the Birman exact sequence, and the first row of vertical arrows are the natural quotient maps.

\begin{figure}[htp]
	\begin{center}
		\begin{tikzcd}
			1\arrow{r}&\pi_1(S)\arrow{r}\arrow{d}& \mcg[\pm]{S_*} \arrow{r}\arrow{d}& \mcg[\pm] S \arrow{r}\arrow{d} &1\\
			1 \arrow{r}&\Gamma(n) \arrow{r}\arrow{d}{=} &\mcg[\pm]{S_*}/\DT^n(S_*)\arrow{r}\arrow{d}& \mcg[\pm] S/\DT^n(S) \arrow{r}\arrow{d}& 1\\
			1 \arrow{r}&\Gamma(n) \arrow{r}&\aut{\Gamma(n)} \arrow{r}& \out{\Gamma(n)} \arrow{r}& 1
		\end{tikzcd}
	\end{center}
	\caption{}\label{fig: Fundamental commutative diagram}
\end{figure}
\end{prop}


The proof of the proposition occurs in several steps.  The first row of the diagram is the famous Birman exact sequence \cite[Theorem 4.6]{farb_primer_2012}. We only state the case when $S$ is closed, and we state the version for extended mapping class groups.

\begin{theo}[Birman]\label{res: Birman}
	Let $S$ be closed, connected, orientable surface of genus at least two. There is an exact sequence
	\begin{equation*}
		1\to \pi_1(S)\to \mcg[\pm]{S_*}\to\mcg[\pm] S\to 1\,,
	\end{equation*}
	where $S_*$ is the resulting of removing one puncture from $S$.
\end{theo}

The injective map from $\pi_1(S)$ is called the \emph{point-pushing map}, and the surjective map $\mcg[\pm]{S_*}\to\mcg[\pm] S$ is called the \emph{forgetful map}.

We define the second row of the diagram by pushing the Birman exact sequence down to a corresponding sequence for $\Gamma(n)$. First, we define the `point-pushing' map for $\Gamma(n)$.

\begin{lemm}
\label{lem: Point-pushing for power quotient}
	The point-pushing map $\pi_1(S)\to\mcg[\pm]{S_*}$ descends to a homomorphism $\Gamma(n)\to\mcg[\pm]{S_*}/\DT^n(S_*)$.
\end{lemm}

\begin{proof}
	It suffices to show that the point-pushing map sends the $n$th power of a simple closed curve $\gamma$ into $\DT^n(S_*)$. The point-pushing map sends a simple closed curve $\gamma$ to the product of Dehn twists $T_{\gamma^+}T_{\gamma^-}$, where $\gamma^+$ is $\gamma$ pushed slightly to the left and $\gamma^-$ is $\gamma$ pushed slightly to the right \cite[Fact 4.7]{farb_primer_2012}. Since $\gamma^+$ and $\gamma^-$ are disjoint, their Dehn twists commute, and so the point-pushing map sends  $\gamma^n$ to
	\begin{equation*}
		(T_{\gamma^+}T_{\gamma^-})^n=T_{\gamma^+}^nT_{\gamma^-}^n\in \DT^n(S_*)\,.
	\end{equation*}
	This completes the proof.
\end{proof}

Since the forgetful map $\mcg[\pm]{S_*}\to\mcg[\pm] S$  sends Dehn twists on the punctured surface $S_*$ to Dehn twists on the closed surface $S$, it also descends to a map $\mcg[\pm]{S_*}/\DT^n(S_*)\to\mcg[\pm] S/\DT^n(S)$. At this point, we have defined the horizontal maps in the second row of \autoref{fig: Fundamental commutative diagram}, and established that the top row of squares commute.

The third row of \autoref{fig: Fundamental commutative diagram} is the natural sequence defining the outer automorphism group of $\Gamma(n)$.

\medskip
Our next task is to define the second column of vertical arrows.  
The map $\Gamma(n) \to \Gamma(n)$ is simply the identity. 
The map $\mcg[\pm]{S_*}/\DT^n(S_*)\to\aut{\Gamma(n)}$ is defined by the next lemma. 

\begin{lemm}
\label{lem: Map to Aut}
	The natural homomorphism $\mcg[\pm]{S_*}\to\aut{\pi_1(S)}$ descends to a homomorphism $\mcg[\pm]{S_*}/\DT^n(S_*)\to\aut{\Gamma(n)}$.
\end{lemm}

\begin{proof} 
 	It follows from the Dehn--Nielsen--Baer theorem that the subgroup of $n$th powers of simple closed curves is characteristic in $\pi_1(S)$. Hence, there is a natural map $\aut{\pi_1(S)}\to\aut{\Gamma(n)}$. 
	To prove the lemma, it suffices to show that the image of $\DT^n(S_*)$ under the composition
	\begin{equation*}
		\mcg[\pm]{S_*}\to\aut{\pi_1(S)}\to\aut{\Gamma(n)}
	\end{equation*}
	is trivial. We check this for the different topological types of simple closed curves $\gamma$ on $S$.
	
	If $\gamma$ is non-separating then cutting along $\gamma$ realises $\pi_1(S)$ as an HNN extension
	\begin{equation*}
		\pi_1(S)=\pi_1(S_0)*_{\langle\gamma\rangle}
	\end{equation*}
	where $S_0$ is the result of cutting $S$ along $\gamma$.
	We may choose a curve $\alpha$ that intersects $\gamma$ once as a stable letter. 
	The Dehn twist $T_\gamma$ fixes every element in $\pi_1(S_0)$ while sending $\alpha$ to $\alpha\gamma$. 
	Thus, $T_\gamma^n(\alpha)=\alpha\gamma^n$, so $T_\gamma^n$ acts trivially on $\Gamma(n)$.
	
	If $\gamma$ is separating then cutting along $\gamma$ realises $\pi_1(S)$ as an amalgamated free product
	\begin{equation*}
		\pi_1(S)=\pi_1(S_1)*_{\langle\gamma\rangle}\pi_1(S_2)
	\end{equation*}
	where $S_1$ and $S_2$ are the two components that result from cutting $S$ along $\gamma$. 
	Without loss of generality, the base point is contained in $S_1$. Then the Dehn twist $T_\gamma$ fixes the elements of $\pi_1(S_1)$, while mapping an arbitrary element $\beta$ of $S_2$ to the conjugate $\gamma\beta\gamma^{-1}$. 
	Thus, $T_\gamma^n(\beta)=\gamma^n\beta\gamma^{-n}$ so, again, $T_\gamma^n$ acts trivially on $\Gamma(n)$, as required.
\end{proof}

Similarly, we can now define the map  $\mcg[\pm]{S}/\DT^n(S)\to\out{\Gamma(n)}$.

\begin{lemm}
\label{lem: Map to Out}
	The natural homomorphism $\mcg[\pm] S\to\out{\pi_1(S)}$ descends to a homomorphism $\mcg[\pm] S/\DT^n\to\out{\Gamma(n)}$.
\end{lemm}

\begin{proof}
	Again, because the subgroup of $n$th powers of simple closed curves is characteristic in $\pi_1(S)$, there is a natural map $\out{\pi_1(S)}\to\out{\Gamma(n)}$. To prove the lemma, it suffices to show that the image of $\DT^n(S)$ under the composition
	\begin{equation*}
		\mcg[\pm] S\to\out{\pi_1(S)}\to\out{\Gamma(n)}
	\end{equation*}
	is trivial. We will furthermore make use of the naturality of the homomorphisms $\mcg[\pm]{S_*} \to \aut{\pi_1(S)}$ and $\mcg[\pm] S\to \out{\pi_1(S)}$,
	which lead to the following commutative diagram.
	\begin{center}
		\begin{tikzcd}
			\mcg[\pm]{S_*}\arrow{r}\arrow{d}& \mcg[\pm] S\arrow{d}\\
			\aut{\Gamma(n)} \arrow{r}& \out{\Gamma(n)}
		\end{tikzcd}
	\end{center}
	Let $T_\gamma$ be a Dehn twist on $S$. By the definition of the forgetful map, $T_\gamma$ lifts to a Dehn twist on $S_*$, denoted by $T_{\gamma'}$.  
	By \autoref{lem: Map to Aut}, $T_{\gamma'}^n$ maps to the trivial element in $\aut{\Gamma(n)}$, and hence in $\out{\Gamma(n)}$. 
	Thus, by the commutativity of the diagram, $T_\gamma^n$ has trivial image in $\out{\Gamma(n)}$, as required.
\end{proof}

\begin{proof}[Proof of \autoref{prop:big diagram}]
The above lemmas and discussion construct all the maps in \autoref{fig: Fundamental commutative diagram}, and show that all squares in the top row commute. It is now easy to see that the bottom row of squares also commute. Indeed, the vertical arrows in the bottom row all descend from the vertical arrows in the fundamental commutative diagram for mapping class groups:
	\begin{center}
		\begin{tikzcd}
			1\arrow{r}&\pi_1(S)\arrow{r}\arrow{d}& \mcg[\pm]{S_*} \arrow{r}\arrow{d}& \mcg[\pm] S \arrow{r}\arrow{d} &1\\
			1 \arrow{r}&\pi_1(S) \arrow{r}&\aut{\pi_1(S)} \arrow{r}& \out{\pi_1(S)} \arrow{r}& 1.
		\end{tikzcd}
	\end{center}
\end{proof}

Having established the commutative diagram of \autoref{fig: Fundamental commutative diagram}, our next task is to establish the exactness of the rows. As already mentioned above, the top row is the Birman exact sequence, while the exactness of the bottom row holds whenever the centre of $\Gamma(n)$ is trivial, which holds by \autoref{res: infinite + trivial finite radical}.  So it remains to prove exactness of the middle row.

\subsection{A Birman-type theorem}

\autoref{thm: Birman for quotients} follows immediately from \autoref{res: infinite + trivial finite radical} and the next statement.

\begin{theo}
\label{thm:exact}
   Let $S$ be closed, connected, orientable surface of genus at least two. 
   Let $n \in \N$.
   If $\Gamma(n)$ has trivial centre, then the short sequence
	\begin{equation*}
		1\to \Gamma(n)\to \mcg[\pm]{S_*}/\DT^n(S_*)\to\mcg[\pm] S/\DT^n(S)\to 1
	\end{equation*}
	from \autoref{fig: Fundamental commutative diagram}  is exact.
\end{theo}

It is not true that the centre of $\Gamma(n)$ is trivial for all $n$: $\Gamma(2)$ is both non-trivial and abelian.

\begin{exam}\label{exa: Abelian power quotient}
The group $\Gamma(2)$ is isomorphic to the finite abelian group
\[
H_1(S,\Z/2\Z)\,.
\]
Indeed, for every pair of elements $\alpha$, $\beta$ of a standard generating set for $\pi_1(S)$, there is a choice of $\epsilon\in\{\pm1\}$ such that $\alpha\beta^\epsilon$ is also a simple closed curve, so the images of $\alpha$ and $\beta$ commute in $\Gamma(2)$. In particular, the centre $Z(\Gamma(2))$ is non-trivial.
\end{exam}

The next lemma establishes exactness at $\Gamma(n)$.

 \begin{lemm}
\label{lem: The point-pushing map is injective}
	The point-pushing map $\Gamma(n)\to\mcg[\pm]{S_*}/\DT^n(S_*)$ is injective, whenever $\Gamma(n)$ has trivial centre.
\end{lemm}

\begin{proof}
	By the commutativity of the bottom left square in \autoref{fig: Fundamental commutative diagram}, the map $\Gamma(n)\to\aut{\Gamma(n)}$ factors through the point-pushing map. 
	Since $\Gamma(n)$ has trivial centre, the map $\Gamma(n)\to\aut{\Gamma(n)}$ is one-to-one.
	Consequently, the point-pushing map is also injective.
\end{proof}

We then establish exactness at the middle term of the sequence

\begin{lemm}
\label{lem: Exactness in the middle}
	The image of the point-pushing map 
	\begin{equation*}
		\Gamma(n)\to\mcg[\pm]{S_*}/\DT^n(S_*)
	\end{equation*}
	is equal to the kernel of the forgetful map 
	\begin{equation*}
		\mcg[\pm]{S_*}/\DT^n(S_*)\to \mcg[\pm] S/\DT^n(S)
	\end{equation*}
	whenever $\Gamma(n)$ has trivial centre.
\end{lemm}

\begin{proof}
	Because both maps descend from the corresponding maps in the Birman exact sequence, and because the composition of the forgetful and point-pushing maps in the Birman exact sequence is trivial, it follows that the image is contained in the kernel. 
	It remains to show that the kernel of the forgetful map is in the image of the point-pushing map.
	
	Suppose, therefore, that $\phi\in\mcg[\pm]{S_*}$ is such that $\phi \DT^n(S_*)$ is in the kernel of the forgetful map. 
	Since every Dehn twist on $S$ lifts to a Dehn twist in $S_*$, the forgetful map restricts to a surjection $\DT^n(S_*)\to \DT^n(S)$, and therefore $\phi$ may be chosen in the kernel of the forgetful map $\mcg[\pm]{S_*}\to\mcg[\pm] S$. 
	By the exactness of the Birman sequence, $\phi$ is in the image of the point-pushing map, and so $\phi \DT^n(S_*)$ is in the image of the point-pushing map from $\Gamma(n)$, as required.
\end{proof}

Finally, we are ready to prove the analogue of the Birman exact sequence.

\begin{proof}[Proof of \autoref{thm:exact}]
In view of \autoref{lem: The point-pushing map is injective} and \autoref{lem: Exactness in the middle}, it suffices to prove the exactness at $\mcg[\pm] S/\DT^n(S)$.
This is a direct consequence of the surjectivity of the map $\mcg[\pm]{S_*}\to \mcg[\pm] S$.
\end{proof}

\section{Centralisers in power quotients}
\label{sec: centralisers}

\subsection{Dehn filling mapping class groups}

The purpose of this section is to prove injectivity of the natural maps
\[
\mcg[\pm]{S_*}/\DT^n(S_*)\to\aut{\Gamma(n)}
\]
and
\[
\quad \mcg[\pm]{S}/\DT^n(S)\to\out{\Gamma(n)}
\]
provided by \autoref{prop:big diagram}, for all large multiples $n$. Later, we will prove that, for every sufficiently large integer $n$, the natural map $\aut \Gamma \to \aut{\Gamma(n)}$ is onto (\autoref{thm:onto}). This will complete the proof of \autoref{thm: DNB theorem for quotients}, by establishing an isomorphism between $\mcg[\pm]{S_*}/\DT^n(S_*)$ and $\aut{\Gamma(n)}$ (\resp between  $\mcg[\pm]{S}/\DT^n(S)$ and $\out{\Gamma(n)}$).

In the remainder of this section, we assume that $n$ is a sufficiently large integer so that $\Gamma(n)$ is infinite with trivial centre, as in \autoref{res: infinite + trivial finite radical}. 

We need some understanding of the centraliser of the point-pushing subgroup of $\mcg[\pm]{S_*}/\DT^n(S_*)$.  To gain this understanding, we exploit the techniques of \cite{dahmani_dehn_2021} and its sequel \cite{BHMS}, which study the power quotient $\mcg{S_*}/\DT^n(S_*)$ as a kind of combinatorial Dehn filling of the mapping class group. 

\begin{rema}
The reader should note that, following \cite{dahmani_dehn_2021} and \cite{BHMS}, the next few results only apply to the \emph{orientation-preserving} mapping class group $\mcg{S}$. When we finally come to prove the main result of this section, \autoref{res: kernels}, we will reduce from the extended mapping class group $\mcg[\pm]{S}$ to this case by consider the homology $H_1(S,\Z/n\Z)$. We expect all results of \cite{dahmani_dehn_2021} and \cite{BHMS} to hold for extended mapping class groups, but prefer to cite them as stated.
\end{rema}

We first recall Bowditch's definition of an acylindrical action, and Osin's definition of an acylindrically hyperbolic group.

\begin{defi}\label{def: Acylindrical action}
Suppose a group $G$ acts by isometries on a metric space $X$. The action is called \emph{acylindrical} if, for every $C>0$, there are constants $R,K$ so that the number of group elements $g\in G$ satisfying 
\[
d(x,gx)\leq C\quad\mathrm{and}\quad d(y,gy)\leq C
\]
is bounded by $K$, for every $x,y\in X$ with $d(x,y)\geq R$. A group $G$ that admits a non-elementary acylindrical action on a hyperbolic metric space is called \emph{acylindrically hyperbolic}.
\end{defi}

For any surface $S$ of finite type, Masur--Minsky proved that the curve complex $\curve(S)$ is hyperbolic \cite{masur_geometry_1999}, while  Bowditch proved that the action of the mapping class group on the curve complex $\curve(S)$ is acylindrical \cite{bowditch_tight_2008}; in particular, mapping class groups are acylindrically hyperbolic. The reader is referred to \cite{osin_acylindrically_2016} for background material on acylindrically hyperbolic groups. 

Let us introduce some notation. For a fixed hyperbolic surface of finite type $S$, let $X$ denote the curve complex $\curve(S)$. The results of this section apply to all positive multiples $n$ of some large positive integer $N$. For such an $n$, we write $\bar{X}$ for the quotient space $\DT^n(S)\backslash X$ and $\psi:X\to\bar{X}$ for the quotient map. We denote $\phi:\mcg{S}\to\mcg{S}/\DT^n(S)$ the quotient homomorphism. Note that $\mcg{S}/\DT^n(S)$ acts naturally by isometries on $\bar{X}$, making $\psi$ a $\mcg{S}$-equivariant map.

The main result of \cite{dahmani_dehn_2021} is that, for all large multiples, the power quotient $\mcg{S}/\DT^n(S)$ is acylindrically hyperbolic. Even better, using the results of the sequel \cite{BHMS} and the theory of hierarchically hyperbolic spaces, we find that the action of the power quotient on the natural quotient of the curve complex is acylindrical. 

\begin{theo}\label{res: improved DHS theorem}
There is a positive integer $N$ such that, for all positive multiples $n$ of $N$, the quotient space $\bar{X}$ is hyperbolic, and the natural action of $\mcg{S}/\DT^n(S)$  on $\bar{X}$ is acylindrical. In particular,  $\mcg{S}/\DT^n(S)$ is acylindrically hyperbolic.
\end{theo}
\begin{proof}
The quotient $\bar{X}$ is hyperbolic by \cite[Theorem 3.1]{dahmani_dehn_2021}. Furthermore, $\bar{X}$ is the top-level hyperbolic space for a hierarchically hyperbolic groups structure on $\mcg{S}/\DT^n(S)$ by \cite[Theorem 7.3, Theorem 7.1-(II)]{BHMS}, so \cite[Theorem K]{HHS1} implies that the action of $\mcg{S}/\DT^n(S)$ is acylindrical.
\end{proof}

Our intention is to mimic some of the proofs from the appendix of \cite{wilton2024congruence}. The results of that paper are proved under the hypothesis that all hyperbolic groups are residually finite. In order to give unconditional proofs, we need the following facts, all of which can be extracted from \cite{BHMS}.

\begin{lemm}\label{res: Needed facts}
Fix a basepoint $x_0\in X$, with image $\bar x_0$ in $\bar X$. Let $Q<\mcg{S}$ be a convex-cocompact subgroup (i.e.\ the orbit map $Q\to Q.x_0\subseteq X$ is a quasi-isometric embedding).  There is a positive integer $N$ such that, for all positive multiples $n$ of $N$, the following hold.
\begin{enumerate}
  \item\label{item:inj} The map $\phi|_Q$ is injective and the image $\phi(Q)$ is convex-cocompact.

  \item\label{item:cobound_isom} For all $q\in Q$ we have that $\psi|_{[x_0,qx_0]}$ is an isometric embedding.

  \item\label{item:quad_lift} For every geodesic quadrangle $\bar\gamma_1\ast\dots\ast \bar\gamma_4$ in $\bar X$ there exists a geodesic quadrangle $\gamma_1\ast\dots\ast \gamma_4$ in $X$ with $\psi(\gamma_i)=\bar\gamma_i$.

  \item\label{item:lift_cobound} Moreover, in the setting of the previous item, if a $\bar\gamma_i$ is a translate of a geodesic of the form $\psi([x_0,qx_0])$ for some $q\in Q$, then the corresponding $\gamma_i$ is a translate of $[x_0,qx_0]$.
 \end{enumerate}
\end{lemm} 
\begin{proof}
As we shall explain, these statements are all special cases of results in \cite{BHMS}. The results of that paper are proved conditionally under the hypothesis that all hyperbolic groups are residually finite; this hypothesis is used to enable an induction to proceed, via the construction of further quotients. However, the quotient maps $\phi$ and $\psi$ that we consider here are the first level of the induction, and make no use of the residual finiteness hypothesis. In particular, the proofs of \cite{BHMS} prove these facts unconditionally; see \cite[Theorem 7.3]{BHMS} and the discussion above it.

Bearing that in mind, item \ref{item:inj} corresponds to \cite[Theorem 7.1-(III)]{BHMS}, item \ref{item:cobound_isom} corresponds to \cite[Lemma 8.44]{BHMS}, while items \ref{item:quad_lift} and \ref{item:lift_cobound} correspond to  \cite[Proposition 8.13-(VI)]{BHMS}.
\end{proof}

We now apply these geometric results to get some algebraic information about the images $\bar{g}$ in $\mcg{S}/\DT^n(S)$ of pseudo-Anosovs $g\in \mcg{S}$. Given a group $G$ acting acylindrically on a hyperbolic space, and an element $g\in G$ acting loxodromically, the \emph{elementary closure} of $g$ in $G$ is defined to be
	\begin{equation*}
	 	E_G(g)=\left\{h\in G\,\mid \,\exists\, m,n\in \mathbb{Z}\setminus\{0\}\mbox{ s.t. }h g^m h^{-1}=g^n\right\}.
	 \end{equation*}
Note that the elementary closure, by definition, contains the centraliser $C_G(g)$. Recall that, for $G$ and $g$ as above, $\langle g\rangle$ has finite index in $E_G(g)$ \cite[Lemma 6.8]{osin_acylindrically_2016}. 

\begin{prop}\label{res: Non-elementary power quotient}
For any pair of independent pseudo-Anosovs $g_1,g_2\in \mcg{S}$,  there is a positive integer $N$ such that, for all positive multiples $n$ of $N$:
\begin{enumerate}
\item $g_i$ acts loxodromically on the hyperbolic space $\bar{X}$, for both $i=1,2$; \label{enu: Non-elementary power quotient -- loxodromic}
\item the subgroup $\langle \bar{g}_1,\bar{g}_2\rangle$ is non-elementary; \label{enu: Non-elementary power quotient -- non-elementary}
\item $\phi$ induces surjections
\[
E_{\mcg{S}}(g_i)\to  E_{\mcg{S}/\DT^n(S)}(\bar{g}_i)\,,
\]
for each $i=1,2$.\label{enu: Non-elementary power quotient -- E(g)}
\end{enumerate}
\end{prop}
\begin{proof} 
For a large enough integer $k$, the subgroup $Q=\langle g_1^k,g_2^k\rangle$ is free and convex-cocompact for the action on $X$ \cite[Theorem 1.4]{farb_convex_2002}. By \autoref{res: Needed facts}\ref{item:inj} we have that, for $n$ a large multiple, $Q$ embeds into $\mcg{S}/\DT^n(S)$ and quasi-isometrically embeds into $\bar{X}$. This proves both items \ref{enu: Non-elementary power quotient -- loxodromic} and \ref{enu: Non-elementary power quotient -- non-elementary}.
    
 To prove item \ref{enu: Non-elementary power quotient -- E(g)}, we follow the argument for the claim within the proof of \cite[Theorem 2.8]{wilton2024congruence}, which we reproduce here for convenience. Without loss of generality, we take $i=1$. Note that the inclusion $\phi(E_{\mcg{S}}(g_1))<E_{\mcg{S}/\DT^n(S)}(\bar{g}_
 1)$ is clear from the definition of elementary closure.

The fact that the action is acylindrical yields that, for all $C\geq 0$, there exists $K\geq 0$ such that the following holds. If $\gamma\in \mcg{S}$ is such that for some integers $k,l$ with $|k|,|l|>K$ we have
\[
d_X(x_0,\gamma x_0),d_X(g_1^kx_0,\gamma g_1^lx_0)\leq C\,,
\]
then $\gamma\in E_{\mcg{S}}(g_1)$. (That is to say, if distinct cosets of $\langle g_1\rangle$ fellow-travel for an arbitrarily long distance, they represent the same coset of $E_{\mcg{S}}(g_1)$.)

 Suppose now that we have $\bar\gamma\in\mcg{S}/\DT^n(S)$ such that $\bar\gamma\bar{g}_1^k\bar\gamma^{-1}=\bar{g}_1^l$ for integers $k,l\neq 0$; we have to show that $\bar\gamma\in \phi(E_{\mcg{S}}(g_1))$. For the $N$ given by the property stated above with $C=d_{\bar X}(\bar x_0,\bar{\gamma}\bar x_0)$, we can assume up to passing to powers that we have $|k|,|l|>N$. We now have a geodesic quadrangle $\bar{\mathcal Q}$ in $\bar X$ with vertices $\bar x_0$, $\bar\gamma\bar x_0$, $\bar \gamma\bar g_1^k\bar x_0=\bar{g}_1^l\bar{\gamma}x_0$, $\bar g_1^l\bar x_0$, which has two sides of length $C$.
 
 By \autoref{res: Needed facts}\ref{item:cobound_isom}, we can take the sides $[\bar x_0,\bar g_1^l\bar x_0]$ and $\bar{\gamma}[\bar x_0,\bar g_1^k\bar x_0]$ to be images in $\bar X$ of translates of geodesics $[x_0, g_1^lx_0]$ and $[x_0, g_1^k x_0]$. By \autoref{res: Needed facts}\ref{item:quad_lift}, this geodesic quadrangle can be lifted to $X$. By \autoref{res: Needed facts}\ref{item:lift_cobound}, it can in fact be lifted in a way that the vertices of the lifted quadrangle are of the form $x_0$, $\gamma x_0$, $\gamma g_1^lx_0$, $g_1^kx_0$ for some $\gamma\in \mcg{S}$ with $\phi(\gamma)=\bar\gamma$. The structure of the lifted quadrangle implies that $\gamma\in E_{\mcg{S}}(g_1)$, so $E_{\mcg{S}}(g_1)$ surjects onto $E_{\mcg{S}/\DT^n(S)}(\bar{g}_1)$, as required.
\end{proof}

\begin{rema}
    One can improve \autoref{res: Non-elementary power quotient}\ref{enu: Non-elementary power quotient -- E(g)} from surjections to isomorphisms using \cite[Theorem 7.1-(I)]{BHMS} (which says that we can assume that $\phi$ is injective on any given finite set), but we will not need this.
\end{rema}

With these results in hand, we can produce elements of $\Gamma(n)$ with controlled centralisers in the power quotient $\mcg{S_*}/\DT^n(S_*)$ of the mapping class group of the once-punctured surface $S_*$.

\begin{prop}
\label{prop: Centraliser of the point-pushing subgroup}
	Let $S$ be a closed, connected, oriented surface of genus at least two, and let $S_*$ be the result of removing a single puncture.
	There is $N > 0$ such that, for all multiples $n$ of $N$, the point-pushing subgroup contains two elements whose centralisers in $\mcg{S_*}/\DT^n(S_*)$ have trivial intersection.
\end{prop}
\begin{proof}
We use the fact that the action of the surface group $\pi_1(S)$ -- viewed as the point-pushing subgroup of $\mcg{S_*}$ -- on the curve graph $\mathcal C(S_*)$ is non-elementary \cite[Lemma 7.2]{osin_acylindrically_2016}.
	
	We first argue that no non-trivial finite subgroup of $\mcg{S_*}$ is normalised by $\pi_1(S)$. By \cite[Lemma 5.5]{Hull} (which applies by non-elementarity) there exists a unique maximal finite subgroup $K$ that is normalised by $\pi_1(S)$. 
	Since $\pi_1(S)$ is normal, for any $g\in \mcg{S_*}$ the conjugate $gKg^{-1}$ is also a maximal finite subgroup normalised by $\pi_1(S)$, so that $gKg^{-1}=K$, and $K$ is normal. But $\mcg{S_*}$ does not contain non-trivial finite normal subgroups: this is a consequence of two exercises in Ivanov's monograph \cite[Exercises 11.4b and 11.5c]{ivanov_subgroups_1992}; see also \cite[Lemma 10.7]{mangioni2025rigidity} for a complete proof. Hence $K$ is trivial, as required.

 By the the lack of non-trivial subgroups normalised by $\pi_1(S)$ and \cite[Lemma 5.6]{Hull}, there exist independent pseudo-Anosovs $g_1,g_2\in \pi_1(S)$ such that
 	\begin{equation*}
		 E_{\mcg{S}/\DT^n(S)}(g_i)=\langle g_i\rangle
	\end{equation*}
for both $i=1,2$. In particular, the elementary closures of $g_1$ and $g_2$ have trivial intersection. 

Let $N$ be as in  \autoref{res: Non-elementary power quotient}.  The elementary closure $E_{\mcg{S}}(g_i)$ is infinite cyclic, and surjects the infinite group $E_{\mcg{S}/\DT^n(S)}(\bar{g}_i)$ by \autoref{res: Non-elementary power quotient}\ref{enu: Non-elementary power quotient -- E(g)}; therefore, $\phi$ induces isomorphisms
\[
\Z\cong E_{\mcg{S}}(g_i)\cong E_{\mcg{S}/\DT^n(S)}(\bar{g}_i)
 \]
for both $i=1,2$. In particular, if the intersection
\[
E_{\mcg{S}/\DT^n(S)}(\bar g_1)\cap E_{\mcg{S}/\DT^n(S)}(\bar g_2)
\]
were non-trivial, it would follow that
\[
\langle \bar{g}_1,\bar{g}_2\rangle \leq E_{\mcg{S}/\DT^n(S)}(\bar g_1)= E_{\mcg{S}/\DT^n(S)}(\bar g_2)\,,
\]
contradicting \autoref{res: Non-elementary power quotient}\ref{enu: Non-elementary power quotient -- non-elementary}. Therefore, the intersection
\[
E_{\mcg{S}/\DT^n(S)}(\bar g_1)\cap E_{\mcg{S}/\DT^n(S)}(\bar g_2)
\]
is trivial, whence so is the intersection of their centralisers, as required.
\end{proof}

\autoref{prop: Centraliser of the point-pushing subgroup} enables us to understand the kernel of the action of $\mcg[\pm]{S_*}$ on $\Gamma(n)$.

\begin{prop}
\label{res: kernels}
	Let $S$ be a closed, connected, orientable surface of genus at least two. 
	There is an integer $N > 0$ such that, for all multiples $n$ of $N$, the kernels of the maps 
	\begin{equation*}
		\mcg[\pm]{S_*}\to\aut{\Gamma(n)} \quad \text{and} \quad \mcg[\pm]{S}\to \out{\Gamma(n)}
	\end{equation*}
	are $\DT^n(S_*)$ and $\DT^n(S)$ respectively.
\end{prop}
\begin{proof}
\autoref{lem: Map to Aut} asserts that $\DT^n(S_*)$ is contained in the kernel of the map $\mcg[\pm]{S_*}\to\aut{\Gamma(n)}$, for all sufficiently large $n$. To prove the reverse inclusion, consider $k\in \mcg[\pm]{S_*}$ that acts trivially on $\Gamma(n)$. 

The action on $H_1(S_*,\Z/n\Z)$ defines a surjective homomorphism 
\[
\mcg[\pm]{S_*}\to GL_{2g}(\Z/n\Z)
\]
that factors through $\aut{\Gamma(n)}$. Since orientation-reversing mapping classes map to matrices of determinant $-1$, it follows that $k$ must be orientation-preserving, as long as $n>2$. Therefore, $k\in \mcg{S_*}$.

Now let $N>2$, $g_1$ and $g_2$ be as in \autoref{prop: Centraliser of the point-pushing subgroup}. The action of $k$ on $\Gamma(n)$ is via conjugation by $\phi(k)$, so the fact that $k$ acts trivially means that $\phi(k)$ commutes with both $\bar{g}_1$ and $\bar{g}_2$. That is, $\phi(k)$ is in the intersection of their centralisers, so $\phi(k)=1$.

This shows that the map $\mcg[\pm]{S_*}/\DT^n(S_*)\to \aut{\Gamma(n)}$ is injective, and injectivity of  $\mcg[\pm]{S}/\DT^n(S)\to \out{\Gamma(n)}$ then follows from the five lemma.
\end{proof}


\begin{rema}
    Notice that the statement of \autoref{res: kernels} does not hold for genus one, in which case $\pi_1(S)\cong\mathbb Z^2$ and $\mcg[\pm]{S}\cong {\rm GL}(2,\mathbb Z)$.
    Indeed, in that case $\Gamma(n)$ is finite, hence so is its automorphism group, while $\mcg[\pm]{S}/\DT^n(S)$ is infinite for all large multiples $n$ (since it is the quotient of the hyperbolic group $\mcg[\pm]{S}$ by the $n$th power of an infinite order element \cite[Theorem~3]{Olshanskii:1993dr}). 
\end{rema}

\subsection{More properties of \texorpdfstring{$\Gamma(n)$}{Γ(n)}}

To conclude this section, we prove \autoref{thm: 2nd Structural results}, which we restate below.


\structural*

\begin{proof}
	By \autoref{thm:exact}, $\Gamma(n)$ is a normal subgroup of $\mcg[\pm]{S_*}/\DT^n(S_*)$. 
	The  latter is acylindrically hyperbolic by \cite{dahmani_dehn_2021} so, by \cite[Corollary 1.5]{osin_acylindrically_2016}, to conclude that $\Gamma(n)$ is acylindrically hyperbolic we only need to know that it is infinite, which is part of \autoref{res: infinite + trivial finite radical}. 

    Let us prove now that $\Gamma(n)$ is not lacunary hyperbolic.
    We first claim that any action of $\Gamma(n)$ on an $\R$-tree has a global fixed point.
    Since $\Gamma$ is a surface group, we can find a finite generating set $T \subset \Gamma$ with the following properties: every element of $T$ can be represented by a path that is freely homotopic to a simple close loop; moreover for every $t_1, t_2 \in T$, there is $\epsilon \in \{\pm 1\}$ such that $t_1 t_2^\epsilon$ can also be represented by a path that is freely homotopic to a simple close loop.
	Consequently the image in $\Gamma(n)$ of $T$ (which we still denote by $T$ for simplicity) is such that every element in $T$ has finite order.
	Moreover, for every $t_1, t_2 \in T$, there is $\epsilon \in \{\pm 1\}$ such that $t_1 t_2^\epsilon$ has finite order as well.
	Serre's lemma \cite[\S{I.6.5}, Corollary 2]{Serre:1980aa} then implies that any action by isometries of $\Gamma(n)$ on a simplicial tree has a global fixed point, and indeed it is well known that the proof applies equally well to $\R$-trees.
	We noticed that $\mcg[\pm]{S}/\DT^n(S)$ is infinite by  \autoref{res: improved DHS theorem}, and injects into $\out{\Gamma(n)}$  by  \autoref{res: kernels}.
    It follows then from \cite[Corollary 3.4]{Coulon:2019ab} that $\Gamma(n)$ cannot be lacunary hyperbolic, for otherwise $\out{\Gamma(n)}$ would be locally finite.
    This proves item \ref{enu: Structure of Gamma(n) -- acylindrical hyperbolicity}.

	Recall that $\Gamma(n)$ is the direct limit of a sequence of hyperbolic groups $(\Gamma_j)$ (\autoref{res: approximating sequence}).
	Suppose now that $\Gamma(n)$ were finitely presented.
	The sequence $(\Gamma_j)$ would necessarily stabilise, so that $\Gamma(n)$ would be hyperbolic, hence lacunary hyperbolic.
    This would contradict the previous discussion, so \ref{enu: Structure of Gamma(n) -- infinitely presented} follows. 
    
    Item \ref{enu: Structure of Gamma(n) -- vb1} is  \autoref{res: vb1}.

	To prove items \ref{enu: Structure of Gamma(n) -- solvable WP} and \ref{enu: Structure of Gamma(n) -- finite asdim}, we first note that $\mcg[\pm]{S_*}/\DT^n(S_*)$ is hierarchically hyperbolic \cite{BHMS}. 
	Such groups have solvable word problem by \cite[Corollary 7.5]{behrstock_hierarchically_2019} (which goes via Bowditch's proof that coarse median groups satisfy quadratic isoperimetric inequalities \cite{Bow_coarse_median}) and finite asymptotic dimension \cite[Theorem A]{hhs_asdim}. 
	Since both properties pass to finitely generated subgroups, items \ref{enu: Structure of Gamma(n) -- solvable WP} and \ref{enu: Structure of Gamma(n) -- finite asdim} follow by \autoref{thm:exact}.
	Similarly, for item \ref{enu: Structure of Gamma(n) -- fg Tits alternative}, every finitely generated subgroup of (a subgroup of) a hierarchically hyperbolic group either contains a nonabelian free subgroup or is virtually abelian \cite{HHS:boundary,HHS:boundarycorrection}. In the latter case, \autoref{res: elementary subgroup} implies that the subgroup is cyclic or infinite dihedral.

    Item \ref{enu:growth} follows from the proof of \cite[Corollary 8.1]{Zalloum:flipping}. This result is stated for virtually torsion-free hierarchically hyperbolic groups, but it also holds for their virtually torsion-free subgroups that act non-elementarily on the main hyperbolic space, because the proof relies on \cite[Corollary 7.5]{Zalloum:flipping} which applies in that generality. (Recall that $\Gamma(n)$ is virtually torsion-free by \autoref{thm: 1st Structural results}\ref{enu: Structure of Gamma(n) -- virtually torsion-free}.) See also \cite{HHS:growth}, whose arguments probably yield the same results.
\end{proof}

\section{Graphs of groups, graphs of actions}
\label{sec: graph of groups}
%

At this point, we have proved all the main results except the surjectivity of the natural map $\aut{\Gamma}\to\aut{\Gamma(n)}$. The rest of the paper builds up to \autoref{thm:onto}, which proves surjectivity. In this section we recall fundamental facts about group actions on trees.

\subsection{\texorpdfstring{$G$}{G}-trees}

Let $G$ be a group.
A \emph{$G$-tree} is a simplicial tree $S$ endowed with a simplicial action of $G$ without inversion.
This action is \emph{trivial} if $G$ fixes a point.
It is \emph{minimal} if $S$ does not contain any proper $G$-invariant subtree.
Given a $G$-tree, one of the following holds: the action of $G$ is trivial, $G$ fixes an end, or $G$ contains loxodromic elements and there exists a minimal $G$-invariant subtree.
A \emph{splitting} of $G$ is a non-trivial minimal $G$-tree.
Bass--Serre theory builds a correspondence between minimal $G$-trees and \emph{graph of groups decompositions of $G$} \cite{Serre:1980aa}.
In this article we favor the point of view of $G$-trees.

\begin{defi}[Root splitting]
\label{def: root splitting}
	Let $G$ be a group. 
	A \emph{root splitting} of $G$ is a decomposition of $G$ as a non-trivial amalgamated product $G = A \ast_C B$, where $B$ is an abelian group and $C$ is a proper, finite index subgroup of $B$.
\end{defi}

\begin{exam}\label{exa: Root splittings of surface groups}
Let $S$ be a closed, connected, hyperbolic surface. If $S$ is non-orientable, then cutting along the boundary of a M\"obius band exhibits a root splitting of $\pi_1(S)$.

However, it is a standard fact that, if $S$ is orientable, then $\pi_1(S)$ has no root splittings. Indeed, one first argues that any such splitting is realised by cutting along an essential simple closed curve, so both $A$ and $B$ are realised as fundamental groups of subsurfaces of $S$. To conclude, one notes that there is no compact orientable surface with abelian fundamental group and one essential boundary component.
\end{exam}

\subsection{Graphs of actions}

Graphs of actions were formalised by Levitt to decompose the action of a given group on an $\R$-tree \cite{Levitt:1994aa}.
We follow here Guirardel's treatment \cite[Definition~4.3]{Guirardel:2004li}.

\begin{defi}
\label{def: graph of actions}
	A \emph{graph of actions on $\R$-trees} $\Lambda$ consists of the following data.
	\begin{enumerate}
		\item A group $G$ and a $G$-tree $S$ (called the \emph{skeleton} of the graph of actions).
		\item An $\R$-tree $Y_v$ (called \emph{vertex tree}) for each vertex $v$ of $S$.
		\item An \emph{attaching point} $p_e \in Y_{t(e)}$ for each oriented edge $e$ of $S$.
	\end{enumerate}
	Moreover, all these data should be invariant under $G$, i.e. 
	\begin{itemize}
		\item $G$ acts on the disjoint union of the vertex trees so that the projection $Y_v \mapsto v$ is equivariant,
		\item for every edge $e$ of $S$, for every $g \in G$, we have $p_{ge}=gp_e$. 
	\end{itemize}
    If all the edge groups of $S$ are abelian, we speak of a \emph{graph of actions over abelian groups}.
\end{defi}

Given a graph of actions one builds an $\R$-tree $T_\Lambda$ endowed with an action by isometries of $G$: $T_\Lambda$ is obtained from the disjoint union of all trees $Y_v$ by identifying, for every edge $e$ of $S$, the attaching points $p_e \in Y_{v_1}$ and $p_{\bar e} \in Y_{v_2}$ where $v_1$ (\resp $v_2$) is the terminal (\resp initial) vertex of $e$.
We say that the action of $G$ on an $\R$-tree $T$ \emph{decomposes as a graph of actions} if there is a graph of actions $\Lambda$ and a $G$-equivariant isometry from $T$ onto $T_\Lambda$.

\begin{rema*}
	If no confusion can arise, we write $\Lambda$ for both the graph of actions and its underlying skeleton.
	Similarly we speak of $\Lambda$ as a splitting of $G$.
\end{rema*}

\begin{defi}
\label{def: action type on R-tree}
	Let $H$ be a group acting by isometries on an $\R$-tree $Y$.
	We say that the action is
	\begin{enumerate}
		\item \emph{simplicial} if $Y$ is simplicial and the action of $H$ on $Y$ is simplicial,
		\item \emph{axial} if $Y$ is a line and the image of $H$ in $\isom Y$ is a finitely generated group acting with dense orbits on $Y$, and
		\item of \emph{Seifert type} if the action has a kernel $N$ and the faithful action of $H/N$ on $Y$ is dual to an arational measured foliation on a closed $2$-orbifold with boundary.
		\end{enumerate}
\end{defi}

%
\section{The shortening argument}
%
\label{sec: shortening argument}

The goal of this section is to prove the following shortening statement. The collection $\mathfrak H (\delta,\rho)$ was defined in Definition \ref{def: preferred class}.

\begin{theo}[Shortening argument]
\label{res: shortening argument}
	Let $\delta >0$.
	Let $\Gamma$ be the fundamental group of a closed, connected, orientable, hyperbolic surface, with fixed finite generating set $U$. 
	Let $\varphi_k \colon \Gamma \to \Gamma_k$ be a sequence of homomorphisms where $(\Gamma_k, X_k, \mathcal C_k) \in \mathfrak H (\delta,\rho_k)$, and where both $(\rho_k)$ and $\lambda_\infty(\varphi_k, U)$ diverge to infinity.
	Assume that, for every finite subset $W \subset \Gamma \setminus\{1\}$, there is $k_0 \in \N$ such that, for every integer $k \geq k_0$, we have $W \cap \ker \varphi_k = \emptyset$.
	Then there exists $\kappa > 0$ with the following property: for infinitely many $k \in \N$, there exist automorphisms $\alpha_k \in \aut \Gamma$ such that 
	\begin{equation*}
		\lambda^+_1\left(\varphi_k \circ \alpha_k, U\right) \leq (1- \kappa) \lambda^+_1\left(\varphi_k,U\right) \,.
	\end{equation*}
\end{theo}

\begin{rema}
In fact, by invoking JSJ theory, the above theorem holds for $\Gamma$ any finitely generated, torsion-free, non-abelian, CSA, one-ended group without root splittings. (See, for instance, \cite{Guirardel:2017te} for a comprehensive development of JSJ theory.)  For simplicity, we merely state the theorem in the case that we will need it, when $\Gamma$ is a surface group.
\end{rema}

The proof is a (simplified) variation of Coulon--Sela \cite[Theorem~7.46]{Coulon:2021wg}.
We recall here the important steps of the argument highlighting the main differences.
We refer the reader to \cite{Coulon:2021wg} for all the details. 


%
\subsection{Action on a limit tree}
\label{sec: action on a limit tree}
%
As in the statement of \autoref{res: shortening argument}, we let $\Gamma$ be the fundamental group of a closed, connected, orientable, hyperbolic surface, and let $U$ be a finite generating set of $\Gamma$.

Let $\delta \in \R_+^*$ and $(\rho_k)$ be a  a sequence of real numbers.
For every $k \in \N$, we choose a morphism $\varphi_k \colon \Gamma \to (\Gamma_k, X_k, \mathcal C_k)$ where $(\Gamma_k, X_k, \mathcal C_k)$ belongs to $\mathfrak H(\delta, \rho_k)$.
We make the following assumptions
\begin{itemize}
	\item For every finite subset $W \subset \Gamma \setminus\{1\}$, there is $k_0$ such that, for every integer $k \geq k_0$, we have $W \cap \ker \varphi_k = \emptyset$.
	(This exactly means that the sequence of morphisms $\varphi_k \colon \Gamma \to \Gamma_k$ converges to ${\rm id} \colon \Gamma \to \Gamma$ in the space of marked groups.)
	\item The radius $\rho_k$ and the energy $\lambda_\infty (\varphi_k, U)$ diverge to infinity.
\end{itemize}
Since $\Gamma$ is not abelian, we can assume without loss of generality that the image of $\varphi_k$ is not abelian, provided $k \in \N$ is large enough.
According to \autoref{res: comparing energies}, the energies of $\varphi_k$ are related by 
\begin{equation}
\label{eqn: limit tree - comparing energies}
	\lambda_\infty(\varphi_k, U) \leq \lambda^+_1(\varphi_k, U) \leq 2 \card U \lambda_\infty(\varphi_k, U)\,.
\end{equation}
In addition, there exists a point $o_k \in X_k$ with the following properties:
\begin{itemize}
	\item  $\lambda_\infty(\varphi_k, U, o_k) \leq \lambda_\infty(\varphi_k, U) + 20\delta$;
	\item $o_k$ is at a distance at most $\card U \lambda_\infty(\varphi_k, U) + 20\delta$ from a point (almost) minimizing the restricted energy $\lambda^+_1(\varphi_k, U)$.
\end{itemize}
For simplicity, we let 
\begin{equation*}
	\epsilon_k = \frac 1 {\lambda_\infty\left(\varphi_k, U\right)}\,.
\end{equation*}
It follows from our assumptions that $(\epsilon_k)$ converges to zero.
In the remainder of this section, unless mentioned otherwise, we work with the rescaled space $\epsilon_k X_k$, i.e.
\begin{equation*}
	\dist[\epsilon_k X_k] x{x'} = \epsilon_k \dist[X_k] x{x'}
\end{equation*}
for every $x,x' \in X_k$.
We endow $\epsilon_k X_k$ with the action by isometries of $G$ induced by $\varphi_k$.
This space is $\delta_k$-hyperbolic, where $\delta_k = \epsilon_k \delta$ converges to zero.

\subsubsection{Ultra-limit of metric spaces}
We fix a non-principal ultra-filter $\omega \colon \mathcal P(\N) \to \{0, 1\}$.
Recall that a property $P_k$ holds \oas if 
\begin{equation*}
	\omega(\set{k \in \N}{P_k\ \text{holds}}) = 1\,.
\end{equation*}
A real-valued sequence $(u_k)$ is \oeb if there exists $M >0 $ such that $\abs{u_k} \leq M$ \oas.
It converges to $\ell$ along $\omega$ if, for every $\epsilon >0 $, we have $\abs{u_k - \ell} < \epsilon$ \oas.  
Every {\oeb} sequence admits an $\omega$-limit \cite{Bourbaki:1971aa}


We consider now an ultra-limit of metric spaces. 
We refer the reader to Dru\c tu and Kapovich \cite[Chapter~10]{Drutu:2018aa} for a detailed exposition of this construction.
As the sequence $(\delta_k)$ converges to zero, the limit space
\begin{equation*}
	X_\omega = \limo \left(\epsilon_k X_k, o_k\right)
\end{equation*}
is an $\R$-tree.
The action of $\Gamma$ on $X_k$ induces an action of $\Gamma$ on $X_\omega$ without global fixed points; this follows from standard arguments originally due to Bestvina and Paulin \cite{Bestvina:1988iv,Paulin:1991fx}, see for instance \cite[Theorem 4.4]{Groves:2019aa}. 
We denote by $T$ the minimal $\Gamma$-invariant subtree of $X_\omega$.
We will observe in \autoref{rem: faithful action} below that the action of $\Gamma$ on $T$ (hence on $X_\omega$) is faithful.


\subsubsection{Peripheral structure}
We endow the tree $X_\omega$ with an additional structure that will keep track of the elliptic elements fixing apices from $\mathcal C_k$.
First, let 
\begin{equation*}
	\Pi_\omega \mathcal C_k
	= \set{(c_k) \in \Pi_{\N} \mathcal C_k}{\dist {o_k}{c_k} - \epsilon_k\rho_k\ \text{is \oeb}}\,.
\end{equation*}
We write $\mathcal C_\omega$ for the quotient of $\Pi_\omega \mathcal C_k$ by the equivalence relation that identifies two sequences $(c_k)$ and $(c'_k)$ if $c_k = c'_k$ \oas.
Given a sequence $(c_k) \in \Pi_\omega \mathcal C_k$, we write $c = [c_k]$ for its image in $\mathcal C_\omega$. The action of $\Gamma_k$ on $X_k$ induces an action of $\Gamma$ on $\mathcal C_\omega$.

Let $c = [c_k]$ be a point in $\mathcal C_\omega$.
We define a map $b_c \colon X_\omega \to \R$ by sending $x = \limo x_k$ to
\begin{equation*}
	b_c(x) = \limo \left( \dist{c_k}{x_k} - \epsilon_k\rho_k \right).
\end{equation*}
It follows from our definition of $\mathcal C_\omega$ and the triangle inequality that the quantity $\dist {c_k}{x_k} - \epsilon_k\rho_k$ is \oeb.
Hence $b_c$ is well defined.
Note also that the collection $(b_c)$ is $\Gamma$-invariant in the following sense: $b_{gc}(gx) = b_c(x)$, for all $c \in \mathcal C_\omega$, $x \in X_\omega$, and $g \in \Gamma$.
The map $b_c$ can be seen as the Busemann function at a point of $X_\omega \cup \partial X_\omega$ \cite[Section~7.2]{Coulon:2021wg}.
In particular, $b_c$ is convex and $1$-Lipschitz.

\begin{defi}[Peripheral subtrees]
	Let $r \in \R_+$ and $c = [c_k]$ be a point in $\mathcal C_\omega$.
We define the \emph{open} and \emph{closed} \emph{peripheral subtrees} centred at $c$ by
\begin{equation*}
	P(c,r) = \set{x \in X_\omega}{b_c(x) < -r}
	\quad \text{and} \quad
	\bar P(c,r) = \set{x \in X_\omega}{b_c(x) \leq -r}.
\end{equation*}
\end{defi}

For simplicity, we write $P(c)$ and $\bar P(c)$ for $P(c,0)$ and $\bar P(c,0)$ respectively.
Peripheral subtrees are \emph{strictly convex}, i.e.\ if $x$ and $y$ are two points in $\bar P(c,r)$, then the interior of the geodesic $\geo xy$ is contained in $P(c,r)$.
As $(b_c)$ is $\Gamma$-invariant, we get that $gP(c,r) = P(gc,r)$, for every $g \in \Gamma$.
The same statement holds for closed peripheral subtrees.
The separation of conical points $X_k$ has the following consequence \cite[Lemma~7.5]{Coulon:2021wg}.

\begin{lemm}
\label{res: separation conical point}
	Let $c, c' \in \mathcal C_\omega$.
	If $\bar P(c) \cap \bar P(c')$ is non-degenerate, then $c = c'$.
\end{lemm}

\begin{voca*}
	A subtree of $X_\omega$ is \emph{non-degenerate} if it contains at least two points.
	It is \emph{transverse} if its intersection with any peripheral subtree is degenerate.
\end{voca*}

\subsubsection{Group action}
Let us give an overview of the action of $\Gamma$ on $X_\omega$ (\resp $T$).
As explained in \autoref{rem: comparison with CS}, the set of axioms defining the collection $\mathfrak H(\delta, \rho_k)$ slightly differs from \cite{Coulon:2021wg}.
We only highlight here the relevant differences.
Whenever a statement is given without a justification, it means that the proof given in \cite{Coulon:2021wg} either only relies on axioms that are common to \cite{Coulon:2021wg} and our approach, or uses results already proven to hold in our context.

Recall that an element of $\Gamma_k$ is called \emph{elusive} if it is neither loxodromic nor conical with respect to the action on $X_k$. 

\begin{defi}
\label{def: jsj - subgroups}
	An element $g \in \Gamma$ is \emph{elusive} if $\varphi_k(g) \in \Gamma_k$ is elusive \oas.
	An element of $\Gamma$ is \emph{visible} if it is not elusive.
	A subgroup $\Gamma$ is \emph{elusive} if all its elements are elusive.
	It is \emph{visible} otherwise.
\end{defi}

The first observation is a consequence of the thinness property stated in Axiom \ref{enu: family axioms - elusive}, see  \cite[Lemma~7.9]{Coulon:2021wg}
\begin{lemm}
\label{res: elusive fixed point}
	Any non-trivial elusive element in $\Gamma$ fixes a unique point in $X_\omega$.
\end{lemm}

Given a cone point $c \in \mathcal C_\omega$, we denote by $\Gamma_c$ the stabiliser of $c$ (for the action of $\Gamma$ on $\mathcal C_\omega$).
\begin{lemm}
\label{res: stab conical point}
	Let $c \in \mathcal C_\omega$.
	The subgroup $\Gamma_c$ is a maximal abelian subgroup of $\Gamma$, provided it is not trivial.
	It coincides with the global stabiliser of $\bar P(c)$.
	For every non-trivial element $g \in \Gamma_c$, we have $\fix g \subset \bar P(c)$.
\end{lemm}

\begin{proof}
	This result is a combination of Lemmas~7.12 and 7.10 in \cite{Coulon:2021wg}.
\end{proof}

%

Recall that an \emph{arc} is a subset of $X_\omega$ homeomorphic to $[0,1]$.

\begin{lemm}
\label{res: abelian arc stabilisers}
	Let $[x,y]$ be an arc in $X_\omega$.
	Assume that its pointwise stabiliser $H$ is non-trivial. 
	Then $H$ is visible, torsion-free, and abelian.
\end{lemm}

\begin{proof}
	This proof slightly differs from the one given in \cite[Lemma~7.15]{Coulon:2021wg}.
	By assumption $\Gamma$ is torsion free, hence so is $H$.
	As we mentioned above, non-trivial elusive elements fix at most one point.
	Consequently, every non-trivial element of $H$ is visible.
	Consider now $g_1, g_2 \in H \setminus\{1\}$.
	We write $x = \limo x_k$ and $y = \limo y_k$.
	By definition, there exists a sequence $(d_k)$ converging to zero such that $x_k,y_k \in \fix{W_k,d_k}$, where $W_k = \{ \varphi_k(g_1), \varphi_k( g_2)\}$.
	
	We claim that the subgroup $E_k \subset \Gamma_k$ generated by $W_k$ is elementary \oas.
	On the one hand, \ref{enu: family axioms - acylindricity} tells us that $A(\Gamma_k, \epsilon_kX_k) \leq \delta_k$.
	On the other hand, combining \ref{enu: family axioms - elem subgroups} with \autoref{res: nu invariant vs finite cyclic subgroups} we get $\nu(\Gamma_k, \epsilon_k X_k) \leq 2$.
	If $E_k$ is not elementary, then the very definition of the acylindricity parameter (\autoref{def: acylindricity}) yields 
	\begin{equation*}
		\dist {x_k}{y_k} \leq \diam { \fix{W_k, d_k}} \leq 5d_k + \delta_k.
	\end{equation*}
	Passing to the limit we get $x = y$ which contradicts our assumption and completes the proof of our claim.
	
	We now prove that $g_1$ and $g_2$ commute.
	As we observed previously, $E_k$ cannot be parabolic.
	It follows from \ref{enu: family axioms - elem subgroups} that $E_k$ is either abelian or isomorphic to $\dihedral$ \oas.
	In the former case, $\varphi_k(g_1)$ and $\varphi_k(g_2)$ commute $\oas$, hence so do $g_1$ and $g_2$.
	We now claim that the latter case never occurs.
	Observe that any hyperbolic element of $E_k = \dihedral$ is contained in the cyclic subgroup $\Z \subset \dihedral$.
	Therefore, up to permuting $g_1$ and $g_2$, we can assume that $\varphi_k(g_1)$ has order two \oas.
	Hence so does $g_1$, which contradicts the fact that $H$ is torsion-free.
\end{proof}

The next lemma is proved as \cite[Lemma~7.16]{Coulon:2021wg}.

\begin{lemm}
\label{res: tree-graded - transverse tripod}
	Let $x$, $y$, $z$ be three points of $X_\omega$.
	We assume that they are not contained in a single peripheral subtree and they do not lie on a transverse geodesic.
	Then the pointwise stabiliser of the tripod $[x,y,z]$ is trivial.
\end{lemm}

\begin{rema}
\label{rem: faithful action}
	Note that $T$ cannot be reduced to a line (since 
	\autoref{res: abelian arc stabilisers} would force $\Gamma$ to be abelian) or contained in a single peripheral subtree (otherwise $\Gamma$ would be abelian again by \autoref{res: stab conical point}).
	Consequently it follows from the previous statement that the action of $\Gamma$ on $T$ is faithful.
\end{rema}

A subtree $Y$ of $X_\omega$ is \emph{stable} if, for every arc $I\subset Y$, the pointwise stabilisers of $I$ and $Y$ coincide.
The action of $\Gamma$ on $X_\omega$ is \emph{super-stable} if any arc with non-trivial stabiliser is stable.

\begin{lemm}
\label{res: transverse super stable}
	If $H$ is a subgroup of $\Gamma$ preserving a transverse subtree $Y$ of $X_\omega$, then the action of $H$ on $Y$ is super-stable.
\end{lemm}

\begin{proof}
	The proof is a combination of Lemmas~\ref{res: abelian arc stabilisers} and \ref{res: tree-graded - transverse tripod} as noticed by Rips--Sela in \cite[Proposition~4.2]{Rips:1994jg}.
	For more details we refer the reader to \cite[Lemma~7.17]{Coulon:2021wg}.
\end{proof}

Recall that a subgroup $H$ of $\Gamma$ is \emph{small} (for its action on $X_\omega$)  if it does not contain two elements acting hyperbolically on $X_\omega$ whose corresponding axes have a bounded intersection.
The next statement is another important consequence of the above study.

\begin{prop}
\label{res: tree-graded - small subgroups}
	Let $H$ be a small subgroup of $\Gamma$.
	If $H$ is not elliptic (for its action on $X_\omega$), then $H$ is visible and abelian.
	In addition, if $H$ is not cyclic, then either there exists $c \in \mathcal C_\omega$ such that $H$ is contained in $\Gamma_c$ or $H$ preserves a transverse subtree.
\end{prop}

\begin{proof}
	The proof is the same as the one of \cite[Proposition~7.19]{Coulon:2021wg}.
	Note that, although $\Gamma_k$ may contain subgroups isomorphic to $\dihedral$, the group $\Gamma$ cannot contain non-abelian solvable subgroups as it is a surface group.
\end{proof}

%
\subsection{Decomposition and shortening of the action}
%

The first step in the proof of \autoref{res: shortening argument} is to decompose the action of $\Gamma$ on the limit tree $T$ into a graph of actions.

\begin{theo}
\label{res: final decomposition tree}
	The action of $\Gamma$ on $T$ decomposes as a graph of actions over abelian groups $\Lambda$.
	Each vertex action of $\Lambda$ is either peripheral, simplicial, axial or of Seifert type.
	Moreover:
	\begin{enumerate}
		\item if $L_v$ is a Seifert type vertex group of $\Lambda$, then $L_v$ is quadratically hanging with trivial fibre and the corresponding vertex tree is transverse in $T$;
		\item if $L_v$ is an axial vertex group of $\Lambda$, then the corresponding vertex tree is transverse in $T$.
	\end{enumerate}
\end{theo}

\begin{rema*}
	One can prove that the peripheral components are also simplicial.
	Nevertheless we distinguish them as they cannot always be shortened.
\end{rema*}

\begin{proof}[Sketch proof of \autoref{res: final decomposition tree}]
	This statement corresponds to Theorem~7.44 in \cite{Coulon:2021wg} and is obtained in exactly the same way.
	The proof has two steps.
	\begin{enumerate}
		\item
		First one produces a preliminary decomposition of $T$ as a graph of actions over abelian groups whose corresponding vertex subtrees are either maximal peripheral subtrees (i.e.\ non-degenerate subtrees of the form $T \cap \bar P(c)$ for some $c \in \mathcal C_\omega$) or transverse subtrees.
		This amounts to proving that the collection $\mathcal Y$ which consists of maximal peripheral subtrees and transverse subtrees form a transverse covering of $T$, i.e.
		\begin{itemize}
			\item \label{enu: transverse covering - intersection}
			every two distinct elements of $\mathcal Y$ have a degenerate intersection, and
			\item \label{enu: transverse covering - cover}
			any bounded arc of $T$ can be covered by finitely many elements of $\mathcal Y$.
		\end{itemize}
		The proof given in \cite[Proposition~7.41]{Coulon:2021wg} relies on \autoref{res: tree-graded - small subgroups} and an accessibility criterion.
		More precisely, one exploits the fact that freely indecomposable limit groups over the class of groups $\mathfrak H(\delta, \rho)$ admit a JSJ decomposition, which covers an important part of the proof; see Section~6 in \cite{Coulon:2021wg}.
		However, the limit group which appears here is exactly $\Gamma$, which is by assumption a hyperbolic surface group, and the JSJ decomposition is trivial.
		
		\item The next steps consist in further decomposing the transverse subtrees.
		If $H$ is a subgroup of $\Gamma$ preserving a transverse subtree $Y$ of $T$, we know by \autoref{res: transverse super stable} that the action of $H$ on $Y$ is super-stable.
		Therefore, this action decomposes as a graph of actions whose components are either axial, simplicial or of Seifert type. 
		See for instance Guirardel \cite[Theorem~5.1]{Guirardel:2008ik}. 
		More precisely, a relative approach can be used to make sure that the decomposition of each transverse subtree refines the previous decomposition.
		The details of this second part work verbatim as in \cite[Proposition~7.41]{Coulon:2021wg}.\qedhere
	\end{enumerate}
\end{proof}

We are now in position to explain the proof of \autoref{res: shortening argument}

\begin{proof}[Sketch proof of \autoref{res: shortening argument}]
	Let $\delta >0$.
	Recall that $\Gamma$ is the fundamental group of a closed, orientable, hyperbolic surface.
	We choose a sequence of morphisms $\varphi_k \colon \Gamma \to \Gamma_k$ where $(\Gamma_k, X_k, \mathcal C_k) \in \mathfrak H_\delta(\rho_k)$ and where $(\rho_k)$ and $\lambda_\infty(\varphi_k, U)$ diverge to infinity.
	We assume that this sequence converges to ${\rm id} \colon \Gamma \to \Gamma$ in the topology of marked groups.
	Each morphism induces an action of $\Gamma$ on $X_k$ rescaled by the energy $\lambda_\infty(\varphi_k, U)$.
	As explained above, we let these metric structure degenerate to get an action of $\Gamma$ on an $\R$-tree $X_\omega$ with torsion-free abelian arc stabilisers.
	We denote by $T$ the minimal $\Gamma$-invariant subtree of $X_\omega$, on which $\Gamma$ acts faithfully.
	We keep all the notations introduced in \autoref{sec: action on a limit tree}.
	
	According to \autoref{res: final decomposition tree}, this action decomposes into a graph of actions $\Lambda$ whose vertex tress are either peripheral, simplicial, axial or of Seifert type.
	Following the path taken in \cite[Section~7.5]{Coulon:2021wg}, we exploit the structure of these four kind of vertex actions to shorten the sequence of morphisms $(\varphi_k)$.
	Since all the morphisms $\varphi_k \colon \Gamma \to \Gamma_k$ trivially factor through ${\rm id} \colon \Gamma \to \Gamma$, unlike in \cite{Coulon:2021wg} there is no need to work with well chosen covers of the limit group.
	One can directly shorten the morphism with an automorphism of $\Gamma$.
	For technical reasons, we do not try to reduce the $L^\infty$-energy $\lambda_\infty(\varphi_k, U)$, but the restricted energy $\lambda_1^+(\varphi_k, U)$.
	This allows us to work with a base point $x = \limo x_k$ in $T$ such that $x_k$ is a point in the thick part $X_k^+$ of $X_k$ which achieves the minimum restricted energy -- up to an error negligible compared to $\lambda_1^+(\varphi_k, U)$.
	More precisely, we claim that there is $\ell \in \R_+^*$ and a sequence of automorphisms $\alpha_k \in \aut G$, such that 
	\begin{equation*}
		\lambda_1^+(\varphi_k \circ \alpha, U) \leq \lambda_1^+( \varphi_k, U) - \frac \ell{\epsilon_k}, \quad \oas\,.
	\end{equation*}
	
	\begin{itemize}
		\item 
		The case of axial components and Seifert-type components is now standard.
		More precisely, there is $\alpha \in \aut \Gamma$ such that for every element $u \in U$,
		\begin{itemize}
			\item if the geodesic $\geo x{ux}$ has a non-degenerate intersection with an axial type component or Seifert type component of $\Lambda$, then $\dist{\alpha(u)x}x < \dist {ux}x$,
			\item otherwise $\alpha(u) = u$.
		\end{itemize}
		See Theorem~7.48 and 7.49 in \cite{Coulon:2021wg}.
		Note in particular that, since $\Gamma$ is 
		a surface group, every solvable subgroup of $\Gamma$ is abelian, hence Theorem~7.49 applies in our context as well.
		Unfolding the definition of the limit action on $T$, we deduce that there is $\ell \in \R_+^*$ such that if $\Lambda$ contains an axial type component or a surface type component then 
		\begin{equation*}
			\lambda_1^+(\varphi_k \circ \alpha, U) \leq \lambda_1^+( \varphi_k, U) - \frac \ell{\epsilon_k}, \quad \oas,
		\end{equation*}
		See \cite[Proposition~7.50]{Coulon:2021wg}.
		
		\item 	
		With a similar classical analysis, we prove that if $\Lambda$ contains a simplicial component that is not peripheral, then there exist $\ell \in (0, 1)$ and a sequence of automorphisms $\alpha_k \in \aut \Gamma$ such that 
		\begin{equation*}
			\lambda_1^+(\varphi_k \circ \alpha_k, U) \leq \lambda_1^+( \varphi_k, U) - \frac \ell{\epsilon_k}, \quad \oas\,;
		\end{equation*}
		see \cite[Proposition~7.61]{Coulon:2021wg}.
		
		\item 
		In view of the above discussion, we are left to handle the case where $\Lambda$ consists only of peripheral components.
		Let $c \in \mathcal C_\omega$ such that the peripheral subtree $\bar P(c)$ has a non-degenerate intersection with $T$.
		Denote by $v_c$ the vertex corresponding to this component in the skeleton $S$ of the graph of action $\Lambda$.
		Axiom~\ref{enu: family axioms - conical - 2} leads to a fine description of the action of $G_c$ on the corresponding peripheral tree $\bar P(c)$.
		See for instance Lemmas~7.3 and 7.52 in \cite{Coulon:2021wg}.
		This careful analysis allows us again to shorten the action of $\Gamma$, using Dehn twists whose effect is to fold edges in $S$.
		More precisely \cite[Theorem~7.51]{Coulon:2021wg} ensures that if either
		\begin{itemize}
			\item	the link of $v_c$ in $S$ contains at least two $\Gamma_c$-orbits, or
			\item	the pointwise stabiliser of $\bar P(c)$ has infinite index in $\Gamma_c$,
		\end{itemize}
		then there is $\ell \in (0, 1)$ and a sequence of automorphisms $\alpha_k \in \aut \Gamma$ such that 
		\begin{equation*}
			\lambda_1^+(\varphi_k \circ \alpha_k, U) \leq \lambda_1^+( \varphi_k, U) - \frac \ell{\epsilon_k}, \quad \oas\,,
		\end{equation*}
		in which case the claim is proved.
	\end{itemize}
	
	The only situation, where the claim is not proven yet is when the graph of groups decomposition of $\Gamma$ associated to $\Lambda$ has the following form
	
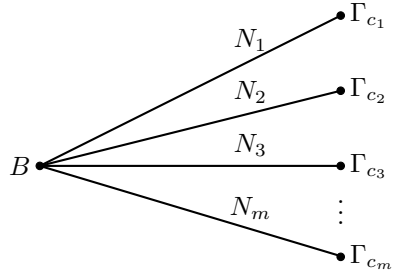
\begin{center}
		\begin{tikzpicture}
			\def \a{4}
			\def \b{1}
			\def \r{0.05}
						
			\draw[fill=black] (0,0) circle (\r) node[left, anchor=east]{$B$};
			\draw[fill=black] (\a,2*\b) circle (\r) node[right, anchor=west]{$\Gamma_{c_1}$};
			\draw[fill=black] (\a,\b) circle (\r) node[right, anchor=west]{$\Gamma_{c_2}$};
			\draw[fill=black] (\a,0) circle (\r) node[right, anchor=west]{$\Gamma_{c_3}$};
			\draw[fill=black] (\a,-0.5\b) node[anchor=center]{$\vdots$};
			\draw[fill=black] (\a,-1.2\b) circle (\r) node[right, anchor=west]{$\Gamma_{c_m}$};
			
			\draw[thick] (0,0) -- (\a, 2*\b) node[pos=0.7, above, anchor=south]{$N_1$};
			\draw[thick] (0,0) -- (\a, \b) node[pos=0.7, above, anchor=south]{$N_2$};
			\draw[thick] (0,0) -- (\a, 0) node[pos=0.7, above, anchor=south]{$N_3$};
			\draw[thick] (0,0) -- (\a, -1.2*\b) node[pos=0.7, above, anchor=south]{$N_m$};
		\end{tikzpicture}
	\end{center}
	where $B$ is the stabiliser of a transverse component, $c_1, \dots, c_m \in \mathcal C_\omega$ are cone points, and the pointwise stabiliser of each $\bar P(c_i)$, denoted by $N_i$, has finite index in $\Gamma_{c_i}$.
	In this settings, $\Gamma$ admits a root splitting, which contradicts \autoref{exa: Root splittings of surface groups}.
	Hence the proof of the claim is complete.
	
	According to (\ref{eqn: limit tree - comparing energies}), the quantities $1/ \epsilon_k = \lambda_\infty(\varphi_k,U)$ and $\lambda_1^+(\varphi_k, U)$ have the same order of magnitude.
	Therefore our claim yields
	\begin{equation*}
		\lambda_1^+(\varphi_k \circ \alpha_k, U) \leq (1 - \kappa) \lambda_1^+( \varphi_k, U), \quad \oas\,,
	\end{equation*}
	where $\kappa$ is a parameter that does not depends on $k$.
\end{proof}

%
\section{Lifting automorphisms}
%
\label{sec: lifting}

The goal of this section is to prove \autoref{thm:onto}, which says that the natural map $\aut{\Gamma} \to \aut{\Gamma(n)}$ is onto for every large exponent $n$.

Let $\Gamma$ be the fundamental group of a closed compact surface $S$ of genus at least two.
We fix a hyperbolic metric on $S$ so that its universal cover $X$ is isometric to the hyperbolic plane $\mathbf H^2$.

%
\subsection{Approximating morphisms}
%

Let $N$ be the critical exponent given by \autoref{res: approximating sequence}.
Fix an integer $n \geq N$.
Let $(\Gamma_j, X_j, \mathcal C_j)$ be a sequence of approximations of $\Gamma(n)$.
For every $j \in \N$, we write $\sigma_j$ for the natural epimorphism $\sigma_j \colon \Gamma_j \onto \Gamma(n)$.
Similarly, we write $\sigma$ for the projection $\sigma \colon \Gamma \to \Gamma(n)$.


\begin{defi}
\label{def: approx of morphism}
 	Let $G$ be a group and $\varphi \colon G \to \Gamma(n)$ a morphism.
	An \emph{approximation of $\varphi$} is a morphism $\tilde \varphi \colon G \to \Gamma_j$ for some $j \in \N$ such that $\varphi = \sigma_j \circ \tilde \varphi$.
	Alternatively, we say that $\tilde \varphi$ \emph{lifts $\varphi$}.
	Such an approximation is \emph{minimal} if $j$ is minimal for the above property.
\end{defi}

Recall that $\Gamma(n)$ is the direct limit of $\Gamma_j$. 
If $G$ is finitely presented, any morphism $\varphi \colon G \to \Gamma(n)$ admits a minimal approximation.
The next statement provides a lower bound on the energies of minimal approximations. It is proved as \cite[Proposition~3.14]{Coulon:2021wg} using \autoref{res: approximating sequence}\ref{enu: approximating sequence - lifting}.

\begin{prop}
\label{res: approx - energy minimal approx}
	Let $G = \group{ U \mid R}$ be a finitely presented group.
	Let $\ell$ be the length of the longest relation in $R$ (seen as words over $U$).
	Let $\varphi \colon G \to \Gamma(n)$ be a morphism whose image is not abelian.
	If $\varphi_j \colon G \to \Gamma_j$ is a minimal approximation of $\varphi$, then either $j = 0$ or 
	\begin{equation*}
		\lambda_\infty(\varphi_j, U) \geq \frac {\rho(n)}{100\ell}\,.
	\end{equation*}
\end{prop}



Motivated by the shortening argument, and bearing in mind our goal of lifting from $\Gamma(n)$ to $\Gamma$, it is natural to to try to minimise energy among all lifts of a morphism. This consideration motivates the next definition.

\begin{defi}
\label{def: short morphism}
	Let $\epsilon > 0$.
	Let $G$ be a group generated by a finite set $U$.
	Let $\mathcal P \subset {\rm Hom}(G,\Gamma(n))$ be a set of morphisms from $G$ to $\Gamma(n)$.
	We say that a morphism $\varphi \in \mathcal P$ is \emph{$\epsilon$-short among $\mathcal P$} if it admits a minimal approximation $\varphi_i \colon G \to \Gamma_i$ such that, for every $\psi \in \mathcal P$ and for every minimal approximation $\psi_j \colon G \to \Gamma_j$ of $\psi$, either
	\begin{itemize}
		\item $i<j$, or
		\item $i=j$ and $\lambda_1^+(\varphi_i, U) \leq \lambda_1^+(\psi_j,U) + \epsilon$.
	\end{itemize}
	In this context $\varphi_i$, is called an \emph{optimal approximation} of $\varphi$.
\end{defi}

%
\subsection{Lifting automorphisms}
%

We are nearly ready to prove surjectivity. The next result applies the shortening argument to show that every $\epsilon$-short morphism either kills an element of a certain finite subset, or admits a lift with energy below a certain threshold.

\begin{prop}
\label{res: lifting morphism - killing elements}
	Let $\Gamma$ be the fundamental group of a closed, orientable, hyperbolic surface and let $U$ be a finite generating set.
	Let $X$ be a geodesic hyperbolic space endowed with a proper and co-compact action of $\Gamma$ (e.g.\ the hyperbolic plane).
	There exist a finite subset $W \subset \Gamma \setminus\{1\}$, an energy level $\lambda >0$, and an exponent $N \in \N$ with the following properties.
	
	For every integer $n \geq N$ and for every morphism $\varphi \colon \Gamma \to \Gamma(n)$, there exists $\alpha \in \aut \Gamma$ such that one of the following holds:
	\begin{itemize}
		\item $W \cap \ker(\varphi \circ \alpha) \neq \emptyset$; or
		\item there is an endomorphism $\psi \colon \Gamma \to \Gamma$ lifting $\varphi \circ \alpha$ such that $\lambda_\infty(\psi, U) \leq \lambda$.
	\end{itemize}
\end{prop}

\begin{proof}
	Let $N$ be the critical exponent given by \autoref{res: approximating sequence}.
	Consider any finite presentation $\Gamma = \group{U \mid R}$.
	Let $\ell$ be the length of the longest relation in $R$ (seen as words over $U$).
	Since the map  $\rho \colon \N \to \R_+$ diverges to infinity, for every $k \in \N$ there exists $M_k \in \N$ such that, for every integer $n \geq M_k$, 
	\begin{equation*}
		\frac {\rho(n)}{100\ell} \geq k\,.
	\end{equation*}
	Fix a non-decreasing exhaustion $(W_k)$ of $\Gamma \setminus \{1\}$ by finite subsets.
	Since $\Gamma$ is non-abelian, we can always assume that $W_k$ contains a non-trivial commutator, for every $k \in \N$.
	
	Assume now that the statement is false.
	In particular, for every $k \in \N$, there is an integer $n_k \geq \max\{k, M_k, N\}$ and morphism $\varphi_k \colon \Gamma \to \Gamma (n_k)$ such that, for every $\alpha \in \aut \Gamma$, the following holds:
	\begin{enumerate}
		\item \label{res: lifting morphism - stably one-to-one - kernel}
		$W_k \cap \ker(\varphi_k \circ \alpha) = \emptyset$; and
		\item \label{res: lifting morphism - stably one-to-one - lift}
		if $\varphi_k \circ \alpha$ admits a lift $\psi \colon \Gamma \to \Gamma$, then $\lambda_\infty(\psi, U) \geq k$.
	\end{enumerate}
	Note that the image of any $\varphi_k$ is not abelian, since $W_k$ contains a non-trivial commutator.
%
%
%
%
%
	For each $k \in \N$, \autoref{res: approximating sequence} provides an approximating sequence 
	\begin{equation*}
		(\Gamma_{j,k}, X_{j,k}, \mathcal C_{j,k})_{j \in \N}
	\end{equation*}
	for $\Gamma(n_k)$.
	Let $\epsilon>0$.
	Up to pre-composing $\varphi_k$ with an element of $\aut \Gamma$, we can assume that, for every $k \in \N$, the morphism $\varphi_k$ is $\epsilon$-short among the set
	\begin{equation*}
		\mathcal P_k = \set{\varphi_k \circ \alpha}{\alpha \in \aut \Gamma}.
	\end{equation*}
	We denote by $\psi_k \colon \Gamma \to \Gamma_{j_k, k}$ an optimal approximation of $\varphi_k$.	
	
	By construction, $\ker \psi_k$ is contained in $\ker \varphi_k$.
	It follows from \ref{res: lifting morphism - stably one-to-one - kernel} that, for every finite subset $W \subset \Gamma \setminus\{1\}$, the intersection $W \cap \ker \psi_k$ is eventually empty.
	(That is, the sequence $\psi_k \colon \Gamma \to \Gamma_{j_k, k}$ converges to ${\rm id} \colon \Gamma \to \Gamma$ in the space of marked groups.)
	We next claim that
	\begin{equation*}
		\lambda_\infty(\psi_k, U) \geq k
	\end{equation*}
	for every $k \in \N$.
	If $j_k = 0$, the conclusion follows from \ref{res: lifting morphism - stably one-to-one - lift}.
	Suppose now that $j_k \geq 1$.
	Observe that the image of $\psi_k$ is not abelian, since if it were, the image of  $\varphi_k$ would also be abelian.
	Recall that $\ell$ is the length of the longest relation defining $\Gamma$.
	By \autoref{res: approx - energy minimal approx}, we have 
	\begin{equation*}
		\lambda_\infty(\psi_k, U) \geq \frac{\rho(n_k)}{100\ell},
	\end{equation*}
	and the conclusion follows from our choice of $M_k$.
	Note that our claim implies that $\lambda_\infty(\psi_k, U)$ diverges to infinity.
	
	By \autoref{res: shortening argument}, there exists $\kappa > 0$, as well as automorphisms $\alpha_k\in \aut \Gamma$, such that
	\begin{equation*}
		\lambda_1^+(\psi_k \circ \alpha_k, U) < (1-\kappa)\lambda_1^+(\psi_k, U) 
	\end{equation*}
	for infinitely many $k \in \N$.
	Observe that $\psi_k \circ \alpha_k$ is an approximation of $\varphi_k \circ \alpha_k$.
	Since $\varphi_k$ is $\epsilon$-short, this approximation is necessarily minimal.
	Since $\psi_k$ is optimal, it satisfies
	\begin{equation*}
		\lambda_1^+(\psi_k, U) \leq \lambda_1^+(\psi_k \circ \alpha_k, U) + \epsilon \leq (1-\kappa)\lambda_1^+(\psi_k,U)  + \epsilon\,,
	\end{equation*}
	so $\lambda_1^+(\psi_k, U)\leq \epsilon/\kappa$. However,  recall that $\lambda_\infty(\psi_k, U)$ diverges to infinity. By (\ref{eqn: limit tree - comparing energies}), $\lambda_1^+(\psi_k, U)$ also diverges to infinity, which leads to the desired contradiction.
\end{proof}

Since there are only finitely many morphisms of a bounded energy, we can quickly upgrade the proposition: after twisting with an automorphism, every morphism either kills an element of a certain finite subset, or lifts to a monomorphism.

\begin{coro}
\label{res: lifting morphism - killing element for hyperbolic groups}
	Let $\Gamma$ be the fundamental group of a closed, orientable, hyperbolic surface, and let $U$ be a finite generating set for $\Gamma$.
	There exist a finite subset $W \subset \Gamma \setminus \{1\}$ and an exponent $N \in \N$ with the following properties.
	For every integer $n \geq N$ and for every morphism $\varphi \colon \Gamma \to \Gamma(n)$, one of the following holds:
	\begin{itemize}
		\item there is an automorphism $\alpha \in \aut \Gamma$ such that $W \cap \ker(\varphi \circ \alpha) \neq \emptyset$; or
		\item there is a monomorphism $\tilde \varphi \colon \Gamma \to \Gamma$ lifting $\varphi$.
	\end{itemize}
\end{coro}

\begin{proof}
	Let $X$ be a geodesic hyperbolic space endowed with a proper and co-compact action by isometries of $\Gamma$.
	We write $W \subset \Gamma \setminus\{1\}$, $\lambda \in \R_+$, and $N \in \N$ for the objects provided by \autoref{res: lifting morphism - killing elements}.
	Note that, up to conjugacy,  there are only finitely many morphisms $\psi \colon \Gamma \to \Gamma$ such that $\lambda_\infty(\psi, U) \leq \lambda$.
	Up to replacing $W$ by a larger subset of $\Gamma \setminus\{1\}$, we can assume that for every non-injective such morphism $\psi \colon \Gamma \to \Gamma$, we have $W \cap \ker \psi \neq \emptyset$.
	
	Let $n \geq N$ be an integer.
	Let $\varphi \colon \Gamma \to \Gamma(n)$ be a morphism.
	Assume that $W \cap \ker (\varphi \circ \alpha)$ is empty for every $\alpha \in \aut \Gamma$.
	We are going to prove that $\varphi$ admits an injective lift.
	According to \autoref{res: lifting morphism - killing elements}, there exist $\alpha \in \aut \Gamma$ and a morphism $\psi \colon \Gamma \to \Gamma$ lifting $\varphi \circ \alpha$ such that $\lambda_\infty(\psi, U) \leq \lambda$.
	Note that $W \cap \ker \psi \subset W \cap \ker(\varphi \circ \alpha)$.
	It follows from our assumption that $W \cap \ker \psi$ is empty.
	Thanks to the adjustment we made on the set $W$ we can conclude that $\psi$ is one-to-one.
	Consequently, $\tilde \varphi = \psi \circ \alpha^{-1}$ is a monomorphism lifting $\varphi$.
\end{proof}

We are finally ready to prove our main goal: for large enough $n$, every automorphism of $\Gamma(n)$ lifts to $\Gamma$. This now follows quickly from \autoref{res: lifting morphism - killing element for hyperbolic groups}, combined with the facts that the $\Gamma(n)$ converge to $\Gamma$ in the space of marked groups and the co-Hopf property of surface groups.

\begin{theo}
\label{thm:onto}
	Let $\Gamma$ be the fundamental group of a closed, orientable, hyperbolic surface.
	There is $N \in \N$ such that, for every exponent $n \geq N$, the morphism $\aut{\Gamma} \to \aut{\Gamma(n)}$ is onto.
\end{theo}

\begin{proof}

	For every $n \in \N$, we denote by $K(n)$ the kernel of the projection $\pi_n \colon \Gamma \onto \Gamma(n)$.
	By \autoref{exa: Root splittings of surface groups}, we may apply \autoref{res: lifting morphism - killing element for hyperbolic groups} with $G = \Gamma$.
	We write $N \in \N$ and $W \subset \Gamma \setminus\{1\}$ for the data given by this corollary.
	Up to increasing the value of $N$ we can assume that for every integer $n \geq N$, the intersection $W \cap K(n)$ is empty (see \autoref{rem: canonical proj asymp injective}).
	
	Let $n \geq N$.
	Let $\varphi \colon \Gamma(n) \to \Gamma(n)$ be an isomorphism.
	Recall that $\pi_n$ induces a map $\aut{\Gamma} \to \aut{\Gamma(n)}$ that we write $\alpha \mapsto \alpha_n$.
	In particular, $\varphi \circ \pi_n \circ \alpha = \varphi\circ \alpha_n \circ \pi_n$, for every $\alpha \in \aut{\Gamma}$.
	As $\varphi$ and $\alpha_n$ are one-to-one, the kernel of $\varphi\circ \alpha_n \circ \pi_n$ coincides with $K(n)$.
	Hence $W \cap \ker (\varphi \circ \pi_n \circ \alpha)$ is empty, for every $\alpha \in \aut{\Gamma}$.
	Applying \autoref{res: lifting morphism - killing element for hyperbolic groups}, there exists a monomorphism $\tilde \varphi \colon \Gamma \to \Gamma$ such that $\pi_n \circ \tilde \varphi = \varphi \circ \pi_n$.
	Since $\Gamma$ is a surface group, it is co-Hopfian.
	Thus the monomorphism $\tilde \varphi \colon \Gamma \into \Gamma$ is actually an automorphism, whence the result.
\end{proof}

Combined with \autoref{res: kernels}, this completes the proof of \autoref{thm: DNB theorem for quotients}.

\begin{proof}[Proof of \autoref{thm: DNB theorem for quotients}]
The natural maps $\mcg[\pm]{S_*}/\DT^n(S_*)\to \aut{\Gamma(n)}$ and $\mcg[\pm]{S}/\DT^n(S)\to \out{\Gamma(n)}$ are injective by \autoref{res: kernels} and surjective by  \autoref{thm:onto}.
\end{proof}

\makebiblio

\Addresses

\end{document}